\documentclass[12pt,a4paper, reqno]{amsart}
\usepackage[utf8]{inputenc}
\usepackage{fullpage}
\usepackage[T1]{fontenc}

\usepackage[final]{microtype}
\usepackage{amsmath}
\usepackage{amsthm}
\usepackage{amsfonts}
\usepackage{amssymb}
\usepackage{bbm}
\usepackage{bm}
\usepackage{graphicx}
\usepackage[shortlabels]{enumitem}
\usepackage{mathrsfs}
\usepackage{mathtools}
\mathtoolsset{showonlyrefs}
\usepackage{pdfpages}
\usepackage{pgfplots}
\pgfplotsset{compat=1.18}
\usepackage{tikz}
\usetikzlibrary{calc,patterns,angles,quotes,plotmarks}
\usepackage{xcolor}
\usepackage[bookmarksdepth=3]{hyperref} 
\usepackage{float}
\usepackage{wasysym}
\hypersetup{
	colorlinks,
	linkcolor=red,
	citecolor=blue,
	urlcolor=blue} 

\theoremstyle{plain} 
\newtheorem{theorem}{Theorem}[section]
\newtheorem{TheoremLetter}{Theorem}
{}

\newtheorem{lemma}[theorem]{Lemma}
\newtheorem{proposition}[theorem]{Proposition}
\theoremstyle{definition}

\theoremstyle{remark}

\newtheorem{remark}[theorem]{Remark}

\numberwithin{equation}{section}

\newcommand{\dd}{\textup{d}}

\newcommand{\R}{\ensuremath{\mathbb{R}}}

\newcommand{\N}{\ensuremath{\mathbb{N}}}

\newcommand{\C}{\ensuremath{\mathbb{C}}}

\newcommand{\calG}{\ensuremath{\mathcal{G}}}
\newcommand{\calO}{\ensuremath{\mathcal{O}}}

\newcommand{\calD}{\ensuremath{\mathcal{D}}}

\newcommand{\calX}{\ensuremath{\mathcal{X}}}
\newcommand{\calV}{\ensuremath{\mathcal{V}}}

\newcommand{\calB}{\ensuremath{\mathcal{B}}}

\newcommand{\sD}{\ensuremath{\mathscr{D}}}

\newcommand{\sF}{\ensuremath{\mathscr{F}}}

\newcommand{\sL}{\ensuremath{\mathscr{L}}}

\newcommand{\sS}{\ensuremath{\mathscr{S}}}

\newcommand{\sfD}{\ensuremath{\mathsf{D}}}

\newcommand{\scrD}{\ensuremath{\mathscr{D}}}

\newcommand{\scrL}{\ensuremath{\mathscr{L}}}

\newcommand{\fraka}{\ensuremath{\mathfrak{a}}}

\newcommand{\Dir}{\textup{Dir}}

\newcommand{\nrm}[1]{\left\lVert #1 \right\rVert}

\newcommand{\changefont}[3]{\fontfamily{#1}\fontseries{#2}\fontshape{#3}\selectfont}
\usepackage[colorinlistoftodos,prependcaption,textsize=tiny]{todonotes}

\begin{document}

\title{$H^\infty$-calculus for the Dirichlet Laplacian on conical domains}

\author[P. Cioica-Licht]{Petru A. Cioica-Licht}
\address[Petru A. Cioica-Licht]{Institute  of Mathematics \\ University of Kassel \\ Heinrich-Plett Str. 40 \\ 34132 Kassel, Germany }
\email{cioica-licht@mathematik.uni-kassel.de}
\author[E. Lorist]{Emiel Lorist}
\address[Emiel Lorist]{Delft Institute of Applied Mathematics\\
Delft University of Technology \\ P.O. Box 5031\\ 2600 GA Delft\\The
Netherlands} \email{e.lorist@tudelft.nl}

\author[P.T. Werner]{P. Tobias Werner}
\address[P. Tobias Werner]{Institute  of Mathematics \\ University of Kassel \\ Heinrich-Plett Str. 40 \\ 34132 Kassel, Germany } \email{twerner@mathematik.uni-kassel.de}

\thanks{The second author was partly financed by the Dutch Research Council (NWO) on the project ``The sparse revolution for stochastic partial differential equations'' with project number \href{https://doi.org/10.61686/ZGRMR99948}{VI.Veni.242.057}. 
The collaboration between the authors has been significantly facilitated by a travel grant of the German Academic Exchange Service (DAAD) for a  visit of the third author to the Analysis Group at TU Delft. 
The first author would like to thank Mark Veraar for numerous discussions on the topic at an early stage of the project. The authors would like to thank Bas Janssens for helpful comments on the differential geometric aspects of the manuscript.}

\changefont{ptm}{m}{n} 


\begin{abstract}
\noindent 
We establish boundedness of the $H^\infty$-calculus for the Dirichlet Laplacian on conical domains in $\R^d$ and corresponding wedges on $L^p$-spaces with mixed weights.
The weights are based on both the distance to the boundary and the distance to the tip/edge of the cone/wedge.
Our main motivation comes from the study of stochastic partial differential equations and associated degenerate deterministic parabolic equations on non-smooth domains.
As a consequence of our analysis, we also obtain maximal $L^p$-regularity for the Poisson equation on conical domains in appropriate weighted Sobolev spaces.
\end{abstract}

\maketitle

{\hypersetup{linkcolor=black}\tableofcontents} 
    
\section{Introduction}
This article is concerned with the Laplace operator on domains with non-smooth boundary. This topic has been investigated extensively in the literature; see, e.g.,
~\cite{BorKon2006, Dau1988,KozMazRos1997,Kozlov2014TheDP,MazRos2010,Nazarov2001,Prss2007HinftycalculusFT, Sol2001} and the references therein.
We are particularly interested in the $H^\infty$-calculus of the Laplace operator on domains with boundaries that exhibit corners and edges.
The $H^\infty$-calculus is known to provide a natural framework for the regularity analysis of (stochastic) partial differential equations ((S)PDEs), as demonstrated in various settings; see, e.g., \cite{DenkDoreHieberPruesVenni_2004, DenKai2013, KalKunWei2006,KunWei2017,PrssSimonett2016,Weis2006} as well as~\cite{AgrestiVeraar2022a,AgrVer2025,NeeVerWei2012,NeeVerWei2012b}.
This is one of the reasons why, over the past decades, boundedness of the $H^\infty$-calculus has been established for a large number of sectorial operators, see the notes of~\cite[Chapter~10]{aibs2} and the references therein.

Driven by applications to (S)PDEs on domains with minimal boundary regularity, the study of the $H^\infty$-calculus of the Laplace operator has recently been extended to  \emph{weighted} Sobolev spaces on domains with $C^1$-boundary~\cite{Lindemulder2024functionalcalculusweightedsobolev, lindemulder2025functionalcalculusweightedsobolev, LindemulderVeraar2020}. The weights measure the distance to the boundary and are crucial for controlling the behaviour of solutions near the boundary; see the introduction of~\cite{lindemulder2025functionalcalculusweightedsobolev} and the references therein.

In this paper we push this analysis beyond $C^1$-domains by studying the boundedness of the $H^\infty$-calculus for  the  Dirichlet Laplacian on weighted $L^p$-spaces on conical domains
\begin{equation}\label{eq:D:intro}
\calD 
:=
\calD_\Omega
:=
\{x\in\R^d\setminus\{0\}\colon \tfrac{x}{|x|}\in \Omega\}; 
\end{equation}
here,  $\Omega$ is an open and connected subset of the unit sphere $S^{d-1}$ admitting a $C^2$-boundary ($d\in\N$, $d\geq 2$). Moreover, the analysis is extended to corresponding wedges $\calD\times\R^m$ with arbitrary codimension $m\in\N$.
The main novelty is that we consider \emph{mixed} power weights based on both the distance $\rho_\calD$ to the boundary $\partial\calD$ and the distance $\rho_\circ$ to the tip $\{0\}$ of $\calD$.
The domain of the operator is formulated in terms of (homogeneous) Sobolev spaces with weights of similar structure.

Our main motivation comes from the regularity analysis of \emph{stochastic} partial differential equations and related partial differential equations with forcing terms degenerating at the boundary. 
Over the last decade, it has been shown in a series of papers~\cite{Cio20,CioKimLee2019, KimLeeSeo2021,KimLeeSeo2022b} that weighted Sobolev spaces with mixed weights based on both the distance to the boundary and the distance to the tip provide a suitable framework for an $L^p$-theory for these equations on angular domains and smooth cones. 
The mixed weights are needed in order to capture the influence and the interplay of the two main sources for spatial singularities of the solution: 
the incompatibility between noise and boundary conditions, which yields  blow-ups of the higher-order derivatives of the solution along the boundary, and the singularities of the boundary, which have a similar effect in their vicinity (we refer to the introductions of~\cite{CioKimKyeLeeLin2018,cioicalicht2024sobolevspacesmixedweights} for details).
The current state-of-the-art results rely on direct calculations by means of refined Green function estimates and  (S)PDE techniques that are typical for the so-called analytic approach to SPDEs. 

With this manuscript we aim to initiate an alternative, operator-theoretic approach to the regularity analysis of SPDEs and related PDEs on non-smooth domains.
Our main goal is the proof of the following result.
For $1<p<\infty$ and $\gamma,\nu\in\R$ we write 
$L^p(\calD,w_{\gamma,\nu}^\calD)$ for the $L^p$-space with respect to the measure with density $w_{\gamma,\nu}^\calD:=\rho_\circ^{\nu-\gamma}\rho_\calD^\gamma$ relative to Lebesgue measure on $\calD$. 
\begin{TheoremLetter}\label{thm:main:Hinf}
Let $\calD:=\calD_\Omega$ be a cone of the form~\eqref{eq:D:intro} with a domain $\Omega$ admitting a $C^2$-boundary and such that $\overline{\Omega}\subsetneq S^{d-1}$. Let $(\lambda_n)_{n\in\N}\subseteq (0,\infty)$ be the ordered sequence of eigenvalues of the Dirichlet Laplace--Beltrami operator on $\Omega$, and set $$\lambda_n^*:= -\tfrac{d-2}{2}+\sqrt{\lambda_n+\tfrac{(d-2)^2}{4}},\quad n\in\N.$$
Let $1<p<\infty$ and let $\gamma,\nu\in\R$ be such that
\begin{align}\label{eq:intro:range:gamma}
\gamma&\in (-1,2p-1)\setminus\{p-1\},    \\
\label{eq:intro:range:nu}
\nu&\in\big(\!-\lambda_1^*p-d,(d+\lambda_1^*)p-d\big)
\setminus
\big\{(2-\lambda_n^*)p-d\,\colon n\in\N  \big\}.
\end{align}
Define the Dirichlet Laplacian $\Delta_\Dir$ on $L^p(\calD,w_{\gamma,\nu}^\calD)$ as $\Delta_\Dir u:=\Delta u$ with domain
\begin{align*}
\sfD(\Delta_\Dir) 
&:=
\big\{u\in L^p(\calD,w_{\gamma,\nu}^\calD)\colon \rho_\circ^{-2} u, \rho_\circ^{-1} u_x,  u_{xx} \in L^p(\calD,w_{\gamma,\nu}^\calD),\,u|_{\partial\calD\setminus\{0\}}=0
\big\}.
\end{align*}
Then $-\Delta_\Dir$ has a bounded $H^\infty$-calculus of angle zero.
\end{TheoremLetter}
Theorem~\ref{thm:main:Hinf} is part of our main result, Theorem~\ref{thm:Laplace:inhom}, and will be proven in Section~\ref{sec:Dirichlet-Laplacian}.
The key in our proof is an extrapolation argument to extend the bounded $H^\infty$-calculus of the Dirichlet Laplacian on  $L^2(\calD)$, which follows directly from the spectral theorem, to weighted $L^p$-spaces using the refined Green function estimates from~\cite{KimLeeSeoRefinedGreenCone}. Due to both the singularity at the tip of $\calD$ and the Dirichlet boundary conditions, these Green function estimates are not of the classical Gaussian form;  they come with additional factors describing the behaviour near the boundary. Our key insight is that, by conjugating the Dirichlet Laplacian with a suitable pointwise multiplication operator, the resulting operator satisfies Gaussian estimates with respect to a weighted doubling measure, allowing us to apply the extrapolation result based on off-diagonal estimates from~\cite{BernicotFreyPetermichl_2016}. Afterwards, it remains to prove that the extrapolated operator is indeed the Dirichlet Laplacian on weighted $L^p$-spaces with domain as in Theorem~\ref{thm:main:Hinf}; for this we first characterise the  homogeneous domain of the Dirichlet Laplacian by means of operator-sum techniques inspired by the analysis in~\cite[Section~5]{Prss2007HinftycalculusFT}. Remarkably, in stark contrast to \cite{Prss2007HinftycalculusFT}, we do not require operator-sum techniques for \emph{non-commuting} operators. 

\medskip

Coming back to our main motivation described above, Theorem~\ref{thm:main:Hinf} has substantial consequences for the regularity analysis of (S)PDEs on non-smooth domains:
\begin{itemize}
\item The admissible ranges
for the weight parameters $\gamma$ and $\nu$ in Theorem \ref{thm:main:Hinf} recover and extend the results from~\cite{ Cio20, CioKimLee2019, CioKimKyeLeeLin2018, KimLeeSeo2021,KimLeeSeo2022b} for the stochastic heat equation on angular domains by means of~\cite{NeeVerWei2012b,NeeVerWei2012}, unveiling their operator-theoretic structure. We refer to Remark~\ref{rem:main:thm:maxReg} for a detailed comparison as well as to Remark~\ref{rem:main:thm} for results on the sharpness of the admissible ranges. 
Moreover, we can deduce an $L^q(L^p)$-theory with differing integrability parameters in time and space.
At this stage, to the best of our knowledge, this is not possible within the analytic approach employed in~\cite{KimLeeSeo2021,KimLeeSeo2022b}.
Both the extension of the range and the freedom in the integrability parameters are particularly relevant for the analysis of non-linear (S)PDEs in critical spaces along the lines of~\cite{AgrestiVeraar2022a, AgrestiVeraar2022b, AgrVer2025}.
In addition, similar results may be derived for stochastic integral equations on angular domains by means of~\cite{DesLon2013}.
So far, the analysis of these equations is  limited to equations on $C^1$-domains or on $\R^d$.
\item The boundedness of the $H^\infty$-calculus with angle less than $\pi$ implies bounded imaginary powers (see~\cite{CowlingDoustMcIntoshYagi1996}). As a consequence, the (homogeneous) domains of fractional powers of the operator can be described as complex interpolation spaces (see~\cite[Theorem~15.3.9]{aibs3}).
Thus, in all the applications described above, we not only obtain abstract formulations in terms of domains of fractional powers of the Laplacian but rather very concrete descriptions of the different components of the equation in terms of appropriate weighted Sobolev spaces.
\item With the bounded $H^\infty$-calculus for the  Dirichlet Laplacian on angular domains at hand, we can set up the operator-valued functional calculus due to~\cite{Kalton2001TheH} and analyse maximal regularity and the $H^\infty$-calculus for more sophisticated operators and domains involving the Laplacian on conical domains, along the lines of~\cite[Section~5]{Prss2007HinftycalculusFT}.
\end{itemize}
To keep the exposition at a reasonable level we postpone all these applications to future papers and concentrate here on the analysis of the $H^\infty$-calculus for the Dirichlet Laplacian on conical domains.

As mentioned before, boundedness of the $H^\infty$-calculus has been established for a large number of operators. 
However, to the best of our knowledge, not much is known about the functional calculus of differential operators in weighted $L^p$-spaces on conical domains.
An exception is the analysis in~\cite{CoriascoSchroheSeiler2007,RoidosSchroheSeiler2021}, where second-order differential operators on manifolds with conical singularities at the boundary are being analysed within certain classes of Kondratiev-type spaces with weights based on the distance to the conical singularities of the boundaries.
The setting therein is not fully comparable to the one considered in this manuscript, particularly in terms of the class of manifolds and differential operators.
The special case of the Dirichlet Laplacian on a conical domain is part of the analysis in~\cite{CoriascoSchroheSeiler2007,RoidosSchroheSeiler2021}. Boundedness of the $H^\infty$-calculus is established within certain weighted $L^p$-spaces. However, the weights are based solely on the distance to the tip and they are put to one outside of a neighbourhood of the tip.
As a consequence, also the weighted Sobolev spaces describing the domain of the operator are different from the ones considered here. 
Moreover, at this stage, it is unclear how to include additional weights based on the distance to the boundary in their approach.

\medskip

We choose the following outline: In Section~\ref{sec:prelim} we fix some preliminaries which will be needed throughout the manuscript. In Section~\ref{sec:WSob} we introduce different classes of weighted Sobolev spaces and discuss the properties and relationships among them that will be needed in the proof of our main result. Section~\ref{sec:Dirichlet-Laplacian} contains our main result, Theorem~\ref{thm:Laplace:inhom}, and its extension to wedges $\calD\times\R^m$ of arbitrary codimension $m\in\N$ is proven in Section~\ref{sec:wedge}. Finally, in Section~\ref{sec:Poisson}, we present a consequence of our analysis on the regularity of the Poisson equation on conical domains in weighted Sobolev spaces, which extends a recent result from~\cite{cioicalicht2024sobolevspacesmixedweights}.

\bigskip

\noindent\textbf{Notation.} Before we start, let us fix some notation.
Throughout this manuscript we write $\N:=\{1,2,3,\ldots\}$, $\R$, and $\C$ for the set of natural, real, and complex numbers, 
respectively; $\N_0:=\N\cup\{0\}$ and $\R^d_+:=(0,\infty)\times\R^{d-1}$ for $d\in\N$. 
 We write $S^{d-1}:=\{x\in \R^d\colon |x| =1\}$ for the unit sphere in $\R^d$ of dimension $d-1$.
 
Let $X$ and $Y$ be normed spaces. Then $\scrL(X,Y)$ denotes the space of all bounded linear operators from $X$ to $Y$, endowed with the usual operator norm $\nrm{\cdot}_{\scrL(X,Y)}$; $\scrL(X):=\scrL(X,X)$. 
Let $(X,Y)$ be an interpolation couple, i.e. let $X$ and $Y$ be continuously embedded in a common Hausdorff topological vector space.
We write $X\cap Y$ for the intersection of the two spaces, endowed with the norm $\nrm{\cdot}_{X\cap Y}:= \nrm{\cdot}_{X}+\nrm{\cdot}_{Y}$,
and we let $[X,Y]_\vartheta$ denote the complex interpolation space with exponent $\vartheta\in [0,1]$.
$X\hookrightarrow Y$ means that $X$ is continuously embedded in $Y$, whereas $X\simeq Y$ means that $X$ and $Y$ are isomorphic.
 We write $(X,\nrm{\cdot})^\sim$ for the completion of $X$ with respect to the norm $\nrm{\cdot}$.

Let $(S,\Sigma,\mu)$ be a measure space, let $1\leq p<\infty$, and let $X$ be a Banach space. Then $L^p(S,\mu;X)$ stands for the space of all (equivalence classes of) strongly $\mu$-measurable, $p$-inte\-grable functions $f\colon S\to X$. 
Moreover, $L^\infty(S,\mu;X)$ stands for the space of all (equivalence classes of) strongly $\mu$-measu\-rable functions $f\colon S\to X$ that are $\mu$-essentially bounded.
Throughout the paper, if $X=\C$, then $X$ is omitted from function spaces.
If $v\colon S\to[0,\infty]$ is measurable, then $v\mu$ stands for the measure with density $v$ 
with respect to $\mu$; $\mu$ is omitted if it is the 
 canonical measure on a Riemannian manifold, see \cite[Section~3.4]{Grigoryan2009_Book}. In particular, we omit the Lebesgue measure in Euclidean settings.
If $(\calX,\delta)$ is a metric space and $\mu$ is a measure on the Borel $\sigma$-algebra $\calB(\calX)$ on $\calX$ (with respect to  $\delta$), we write $L^1_\textup{loc}(\calX;X)$ for the space of all locally integrable functions $f\colon \calX\rightarrow X$.

For a smooth manifold $M$ of dimension $d$ and an arbitrary domain $G\subsetneq M$ with non-empty boundary $\partial G$, we say that $G $ has a $C^k$-boundary and write $G\in C^k$,
if for every $z\in \partial G$ and every chart $(U,\varphi)$ of $M$ covering $z$ it holds that $\varphi(G\cap U)\subseteq\R^d$ admits a $C^k$-boundary in the Euclidean sense. Note that it is enough to check this condition for one particular chart. Indeed, if $(U,\varphi)$ is a chart covering some $z\in\partial G$ such that $\varphi(G\cap U)$ admits a $C^k$-boundary, and $(V,\theta)$ is another chart covering $z$, then $\theta(G\cap U\cap V)$ has a $C^k$-boundary with diffeomorphism $\psi\circ\varphi\circ \theta^{-1}$.

Let $G\subseteq\R^d$ be an open set, $\overline{G}$ being its closure. 
 For $k\in\N_0$ and $G\subseteq\calV\subseteq\overline{G}$ we let $C^k(\calV)$ denote the set of all $k$-times continuously differentiable functions on $G$ such that all derivatives of order up to $k$ have a continuous extension to $\calV$.
$C_c^k(\calV)$ is the space of all functions in $C^k(\calV)$ with support that is compact in $\calV$.
We write $C^\infty(\calV):=\bigcap_{k\in\N_0}C^k(\calV)$ and $C_c^\infty(\calV):=\bigcap_{k\in\N_0}C_c^k(\calV)$.
Moreover, for $k\in \N_0\cup\{\infty\}$, we write 
\[
C_{c,\Dir}^k(\calV)
:= 
\big\{u\in C_c^k(\calV)\colon u=0 \text{ on } \partial G\cap \calV \big\}.
\]
We write $ \scrD'(G;X)$ for the space of $X$-valued distributions.
If $G$ and $G'$ are open subsets of $\R^d$ and $\Xi:G'\rightarrow G$ is a $C^\infty$-diffeomorphism, then we write
\[
u\circ \Xi := \Xi^*u:=\bigl\{C_c^\infty(G')\ni \varphi \mapsto (u,\varphi\circ\Xi^{-1}\cdot |\det D\Xi^{-1}|)\bigr\} \in \scrD'(G';X)
\]
 for the pullback of $u\in \scrD'(G;X)$ w.r.t. $\Xi$. 
For a distribution $u\in \scrD'(G;X) $ and $\alpha\in \N_0^d$, we let $ D^\alpha_x u $ denote the $\alpha$-distributional derivative of $u$ with respect to the variable(s) $x$. 
For $m\in\N$, $D^m_x u$ stands for the vector of all derivatives of order $m$ and $\nrm{D^m_x u}:=\sum_{\lvert\alpha\rvert=m}\nrm{D^\alpha_x u}$; $D_xu:=D_x^1u$, $\Delta_xu:=\sum_{i=1}^d D^{2e_i}_xu$, and $\nabla u:= (D^{e_1}_xu,\ldots,D^{e_d}_xu)$, where $e_i$, $i=1,\ldots,d$, are the unit vectors in $\R^d$.
For a differentiable function $f$ we write $ \partial^\alpha_x f $ for its classical $\alpha$-derivative, where we adopt the usual multi-index notation.
We omit the subscript $x$ if it is clear from the context. 

In general, $C$ will denote a positive finite constant, which may differ 
from one appearance to another. 
By $\lesssim$ we mean that the inequality holds with a positive finite constant, whereas by $\eqsim$ we mean that $\lesssim$ and $\gtrsim$ hold.
If we want to emphasise the dependence of the constants on some parameters $a,b,\ldots$, then we write $C_{a,b,\ldots}$, $\lesssim_{a,b,\ldots}$, and $\eqsim_{a,b,\ldots}$.

\section{Preliminaries}\label{sec:prelim}

\subsection{Linear operators and holomorphic functional calculus}
\label{sec:prelim:Calculus}
For $0<\sigma<\pi$ define the sector 
\[
\Sigma_\sigma=\{z\in \C\setminus\{0\}\colon |\arg(z)|<\sigma\}. 
\]
Let $X$ be a Banach space and let $ A\colon X\supseteq\sfD(A)\rightarrow X $ be a linear operator on $X$ with domain $\sfD(A)$. 
We write $\sigma(A)$ for the spectrum of $A$ and $\rho(A)$ for its resolvent set, while $R(\lambda,A):=(\lambda-A)^{-1}$ is the resolvent of $A$ at $\lambda\in\rho(A)$.
We say that $ A $ is \emph{sectorial} if 
it is densely defined, has dense range, and 
there exists a $ \sigma\in (0,\pi) $ such that $ \sigma(A)\subseteq \overline{\Sigma_\sigma} $ and 
\begin{equation}\label{eq:sectoriality}
M_{\sigma,A}:=\sup_{\lambda\in\C\setminus \overline{\Sigma_\sigma}}\nrm{\lambda R(\lambda,A)}_{\scrL(X)}<\infty.
\end{equation}
The infimum over all possible $ \sigma $ is called the \emph{angle of sectoriality of $A$} and is denoted by $ \omega(A)$. 
A sectorial operator is necessarily closed, since its resolvent set is non-empty. Moreover, it is injective by \cite[Proposition 10.1.8]{aibs2}.
The notion of sectoriality varies in the literature. For instance, the dense domain and dense range conditions are sometimes omitted. 
However, if $X$ is reflexive, then the condition that $A$ is densely defined is always fulfilled if $\sigma(A)\subseteq \overline{\Sigma_\sigma}$ for some $\sigma\in (0,\pi)$ and~\eqref{eq:sectoriality} holds, see \cite[Proposition 10.1.9]{aibs2}.

For $ \sigma\in (0,\pi ) $ let $ H^\infty(\Sigma_\sigma) $ denote the Hardy space of all bounded holomorphic functions $ f\colon \Sigma_\sigma \rightarrow \C $ with norm 
$ \nrm{f}_{H^\infty(\Sigma_\sigma)}:= \sup_{z\in\Sigma_\sigma}|f(z)| $.  Let $ H^1(\Sigma_\sigma) $ denote the Hardy space of all holomorphic functions $ f\colon \Sigma_\sigma \rightarrow \C $ such that
$$
\nrm{f}_{H^1(\Sigma_\sigma)} := \sup_{|\nu| <\sigma } \int_{0}^\infty|f(e^{i\nu} \lambda)|\frac{\dd \lambda }{\lambda}<\infty.
$$
If $ A $ is sectorial, $ \omega(A) <\nu<\sigma $, and $f\in H^1(\Sigma_\sigma)$, then we define 
\[ f(A):=\frac{1}{2\pi i}\int_{\partial\Sigma_\nu}f(\lambda)R(\lambda,A)\dd\lambda,\] 
where $ \partial\Sigma_\nu $ is the positively oriented boundary of $\Sigma_\nu$.
The integral is well-defined as a Bochner integral in $\sL(X)$, since the sectoriality estimate of $A$ yields 
\[ \nrm{f(A)}_{\sL(X)}\le \frac{M_{\nu,A}}{2\pi}\int_{\partial \Sigma_\nu} |f(\lambda)|\frac{|\dd \lambda |}{|\lambda|} \leq \frac{M_{\nu,A}}{\pi}\nrm{f}_{H^1(\Sigma_\sigma)} <\infty.  \]

A sectorial operator $ A $ is said to have a \emph{bounded $ H^\infty $-calculus} if there exists an angle $\sigma\in (\omega(A),\pi)$, such that 
\begin{equation}\label{eq:Hinfty}
\nrm{f(A)}_{\scrL(X)}\le C\nrm{f}_{H^\infty(\Sigma_\sigma)},\quad f\in H^\infty(\Sigma_\sigma) \cap H^1(\Sigma_\sigma),
\end{equation}
with a constant $C\geq0$ that does not depend on $f$. If $A$ has a bounded $H^\infty$-calculus, we write $ \omega_{H^\infty}(A)$ for the infimum over all possible $ \sigma>\omega(A) $, such that \eqref{eq:Hinfty} holds, and say that \emph{$ A $ has a bounded $ H^\infty $-calculus of angle $ \omega_{H^\infty}(A)$}.
For detailed expositions on the $H^\infty$-calculus we refer to~\cite{ Haase2006_FunctionalCalculus,aibs2, Kunstmann2004, PrssSimonett2016}. 

In the sequel, we will need the following elementary lemma on isomorphic transformations of operators. 

\begin{lemma}\label{TransformUnderIsometricIsomorphism} Let $ X $ and $ Y $ be two Banach spaces, let $ (A,\sfD(A)) $ be a linear operator in $ X $, let $ T\in\scrL(X,Y)$ be an isomorphism, and let $(B,\sfD(B)):=(T A T^{-1}, T\sfD(A))$. Then the following assertions hold.
	\begin{enumerate}[label=\textup{(\roman*)}]
 		\item\label{transformation:spectrum:resolvent} $ \rho(A)=\rho(B) $ and $ \sigma(A)=\sigma(B) $.
        \item\label{transformation:sectoriality} If $A$ is sectorial, then $B$ is sectorial with $\omega(A)=\omega(B).$
		\item\label{transformation:calculus} If $A$ has a bounded $H^\infty$-calculus, then so does  $B$ with $\omega_{H^\infty}(B)=\omega_{H^\infty}(A)$  and $f(B)= T f(A) T^{-1}$ for all $f\in H^\infty(\Sigma_\sigma), \sigma\in (\omega_{H^\infty}(A), \pi)$. 
	\end{enumerate}
\end{lemma}
\begin{proof}
    The assertions follow directly from the relation $R(\lambda,B)=TR(\lambda,A)T^{-1}$ and hence 
    \[f(B)=
    \frac{1}{2\pi i}\int_{\partial\Sigma_\sigma} f(\lambda) TR(\lambda,A)T^{-1}\dd \lambda = 
    Tf(A)T^{-1}, \quad f\in H^\infty(\Sigma_\sigma),\sigma\in (\omega_{H^\infty}(A), \pi).\qedhere\]
\end{proof}

In the construction of the Dirichlet Laplacian we will also use the following result. Its elementary proof is left to the reader.

\begin{lemma}\label{lem:isom:unbdd}
Let $X$ and $Y$ be Banach spaces and assume that $(X,Y)$ is an interpolation couple. 
Let $\Dot{A}\in\scrL(Y,X)$ be an isomorphism and let $A$ be the linear operator on $X$ given by
\[
Au:= \Dot A u \text{ with } \sfD(A):= X\cap Y.
\]
Then $A$ is a closed linear operator in $X$ and $X\cap Y=\sfD(A)$ in terms of equivalent norms, i.e., 
\[
\nrm{u}_{X\cap Y} = \nrm{u}_X + \nrm{u}_Y \eqsim \nrm{u}_X + \nrm{Au}_X,\quad u\in \sfD(A).
\]
Moreover, if $X\cap Y$ is dense in $(Y,\nrm{\cdot}_Y)$, then
\[
(\Dot{\sfD}(A),\nrm{\cdot}_{\Dot{\sfD}(A)})
:=
(\sfD(A),\nrm{A\,\cdot}_X)^\sim
=
(Y,\nrm{\cdot}_Y).
\]

\end{lemma}
    
\subsection{Spaces of homogeneous type, Muckenhoupt weights and extrapolation}\label{sec:prelim:extrapol}

Let $(\calX,\delta)$ be a metric space. For $x\in \calX$ and $r>0$ we write 
\[B(x,r):=B^\calX(x,r):=\{y\in\calX\colon \delta(x,y)<r\}\]
for the open ball with radius $r$ centred at $x$. 
Let $\mu$ be a Borel measure on $(\calX,\delta)$, i.e., a measure on the Borel $\sigma$-algebra $\calB(\calX)$, such that $\mu(K)<\infty$ for every compact $K\subseteq\calX$. We say that $\mu$ satisfies the \emph{volume doubling property}, if there exists a finite constant $C\geq 1$, such that
\[
0<\mu(B(x,2r))\leq C \,\mu(B(x,r))<\infty,\quad x\in\calX, r>0.
\]
In this case there exists a number $\tau>0$ such that
\begin{equation}\label{eq:doubling:cons2}
\mu(B(x,r))\lesssim \bigg(\frac{\delta(x,y)+r}{s}\bigg)^\tau \mu(B(y,s)),\quad x,y\in\calX,\, r\geq s >0.
\end{equation}
If $(\calX,\delta)$ is a metric space and $\mu$ is a Borel measure that satisfies the volume doubling property, then we call the triplet $(\calX,\delta,\mu)$ a \emph{(metric) space of homogeneous type}.

\medskip

Let $(\calX,\delta,\mu)$ be a metric space of homogeneous type. 
We call a 
$\mu$-locally integrable 
function $w\colon \calX\to (0,\infty)$ a \emph{weight}. For $1<p<\infty$ we say that a weight $w\colon \calX\to (0,\infty)$ is a \emph{(Muckenhoupt) $A_p$-weight}, denoted by $w\in A_p(\mu)$, if 
\[
[w]_{A_p(\mu)}
:= 
\sup_{B} \bigg(\frac{1}{\mu(B)} \int_Bw\, \dd\mu \bigg)\bigg( \frac{1}{\mu(B)}\int_B w^{-\frac{1}{p-1}}\dd\mu\bigg)^{p-1}<\infty,
\]
the supremum being taken over all open balls $B$ in $\calX$. If $w\in A_p(\mu)$ for some $1<p<\infty$, then the measure $w\mu$ also has the volume doubling property, see, e.g., \cite[Lemma~7.7.1(6)]{MitMitMit2022a}. 

\medskip

In order to establish the boundedness of the $H^\infty$-calculus of the negative Dirichlet Laplacian below we will first obtain this statement on a suitable $L^2$-space and then use the following result in order to extrapolate to more general weighted $L^p$-spaces.

\begin{theorem}\label{functionalcalculus:extrapolation:BFP}
    Let $(\calX,\delta,\mu)$ be a locally compact, separable metric space of homogeneous type. 
    Let $(A, \sfD(A))$ be a sectorial operator on $L^2(\calX,\mu)$ with a bounded $H^\infty$-calculus with $\omega_{H^\infty}(A)<\tfrac{\pi}{2}$. Assume that the semigroup $(e^{-tA})_{t\geq 0}$ generated by $-A$ consists of integral operators with kernels  $K_t \colon \calX\times \calX \rightarrow\C $ for $ t>0$. If there exist finite constants $c,C>0$ such that 
    \begin{equation}\label{eq:GaussianBound:gen}
    |K_t(x,y)|\le \frac{C}{\mu(B^{\calX}(x,\sqrt{t}))}e^{-c\frac{\delta(x,y)^2}{t}},\quad t>0,\,\, x,y\in \calX,
    \end{equation}
    then for all $1<p<\infty$, $\omega_{H^\infty}(A) <\sigma<\pi,$ and all $w\in A_p(\mu)$, 
    \begin{equation}\label{eq:Hinf:extrapol}
    \nrm{f(A)}_{\sL(L^p(\calX,w\mu))} \le C_{p,\sigma,A}\,[w]_{A_p(\mu)}^{\max\big\{\frac{1}{p-1},1\big\}} \nrm{f}_{H^\infty(\Sigma_\sigma)},\quad
    f\in H^\infty(\Sigma_\sigma).
    \end{equation}
\end{theorem}
\begin{proof}
The Gaussian bounds in~\eqref{eq:GaussianBound:gen} imply  
the $L^1(\calX,\mu)$--$L^\infty(\calX, \mu)$ off-diagonal estimate~\cite[Equation (1-1)]{BernicotFreyPetermichl_2016}. Moreover, the boundedness of the $H^\infty$-calculus of $A$ on $L^2(\calX,\mu)$ of angle $\eta$ is equivalent to the $\eta$-accretivity of $A$ (see~\cite[Theorem 11.13]{Kunstmann2004}).   
Therefore, the claim follows from \cite[Theorem 3.1]{BernicotFreyPetermichl_2016} with $p_0\curvearrowleft1$, $ q_0\curvearrowleft\infty$, and $L\curvearrowleft A$. 
\end{proof}

\begin{remark}\label{rem:extrapol:ident}
Theorem~\ref{functionalcalculus:extrapolation:BFP} implies that the boundedness of the $H^\infty$-calculus of $(A,\sfD(A))$ on $L^2(\calX,\mu)$ extrapolates to $L^p(\calX,w\mu)$ in the sense that the semigroup $(e^{-tA})_{t\geq 0}$ generated by $-A$ extrapolates to a strongly continuous semigroup $(S_{p,w}(t))_{t\geq 0}$ on $L^p(\calX,w\mu)$ and the negative of its generator $(-A_{p,w},\sfD(A_{p,w}))$  has a bounded $H^\infty$-calculus on $L^p(\calX,w\mu)$ of the same angle as $(A,\sfD(A))$ on $L^2(\calX,\mu)$.
However, in many settings, even if $A$ is a differential operator and $\sfD(A)$ is a well-known function spaces, the precise description of the extrapolated operator $A_{p,w}$ and its domain $\sfD(A_{p,w})$ is not straightforward. 
\end{remark}
    
\subsection{The underlying domain \texorpdfstring{$\calD$}{D}}\label{sec:domain:D}
In this manuscript we analyse the Dirichlet Laplacian in suitable weighted $L^p$-spaces in conical domains. 
In this section we fix the assumptions on the underlying conical domains and discuss some useful parametrizations and properties of them.

If not explicitly stated otherwise, throughout this manuscript,   
\begin{equation}\label{eq:D:convention}
\calD :=
\calD_\Omega :=
\{x\in\R^d\setminus\{0\}\colon \tfrac{x}{|x|}\in \Omega\},
\end{equation}
where $\Omega$ is a  domain, i.e., an open and connected set, in $S^{d-1}$ with $\overline{\Omega}\subsetneq S^{d-1}$. 
We call $\calD_\Omega$  a \emph{$C^k$-cone}, if $\Omega$ admits a $C^k$-boundary, $k\in \N_0\cup\{\infty\}$. We consider $\Omega$ as being fixed and only mention $\Omega$ explicitly when needed for our analysis.
We endow $S^{d-1}$ with its canonical Riemannian metric, that is, the Riemannian metric on $\R^d$ restricted to $S^{d-1}$, and denote its surface measure by $\bm\sigma_{S^{d-1}}$.
The assumption $\overline{\Omega}\subsetneq S^{d-1}$ implies that there exists a $\mathfrak{p}\in S^{d-1}$ such that $\mathfrak{p}\notin \overline\Omega$. 
For every such $\mathfrak{p}$ we may define a \emph{stereographic projection} $\Pi:=\Pi_{\mathfrak{p}}\colon S^{d-1}\setminus\{\mathfrak{p}\}\to \R^{d-1}$, such that, if we denote by $\bm\sigma_\Omega$ the measure on $\Pi(\Omega)$ induced by the surface measure on $\Omega$ under $\Pi$, then
\begin{equation}\label{eq:surface:measure:stereographic:projection}
    \begin{aligned}
    \bm{\sigma}_{\Omega}(A)&\phantom{:}= \int_A h(y)^{d-1}\dd y, \quad &&A\in \calB(\Pi(\Omega)),\\
    h(y)&:= \frac{2}{1+|y|^2}, \quad &&y\in \Pi(\Omega);
\end{aligned}
\end{equation}
note that the density $h^{d-1}$ does not depend on $\mathfrak{p}$. For the convenience of the reader we provide the precise definition of $\Pi=\Pi_\mathfrak{p}$ in Appendix~\ref{appendix:stereographic:projection}; see also Figure~\ref{fig:stereo:2D} for a sketch of $\Pi$ in $\R^2$.
In what follows, we write $$\calO:= \Pi(\Omega).$$ 
Note that $\Pi\colon\Omega\to\calO$ is a $C^\infty$-diffeomorphism. In particular, $\Omega$ admits a $C^k$-boundary (in $S^{d-1}$) if, and only if, $\calO$ does (in $\R^{d-1}$).

\begin{remark}\label{rem:dim:two:cone}
\begin{figure}
    \centering
\begin{tikzpicture}[scale=2]
    \coordinate (a) at (0,1.5);
    \coordinate (b) at (2,1.5);
    \coordinate (c) at (-1.4,0.5);
    \coordinate (d) at (0,-1);

        \node at (1.8,1.35){$x_1$};
        \node at (-0.15,2.95){$x_2$};
        \coordinate (p) at (0,.75);
        \draw[dashed, thin] (p) -- (-1.5,1.5);
        \draw[dashed, thin] (p) -- (-1.25,1.5);
        \draw[dashed, thin] (p) -- (-1,1.5);
        \draw[dashed, thin] (p) -- (-0.75,1.5);
        \draw[dashed, thin] (p) -- ($(0,1.5) + (157:0.75)$);
        \draw[dashed, thin] (p) -- ($(0,1.5) + (135:0.75)$);
        \draw[dashed, thin] (p) -- ($(0,1.5) + (112:0.75)$);
        \draw[dashed, thin] (p) -- ($(0,1.5) + (68:0.75)$);
        \draw[dashed, thin] (p) -- ($(0,1.5) + (45:0.75)$);
        \draw[dashed, thin] (p) -- ($(0,1.5) + (23:0.75)$);
        \draw[dashed, thin] (p) -- ($(0,1.5) + (0:0.75)$);
        \draw[white, fill=white] (p) -- (.85,1.5) -- (-.75,1.5) -- (-.65,1) -- cycle;
        \draw[dashed, color= gray, ultra thin] (p) -- (-1.5,1.5);
        \draw[dashed, color= gray, ultra thin] (p) -- (-1.25,1.5);
        \draw[dashed, color= gray, ultra thin] (p) -- (-1,1.5);
        \draw[dashed, color= gray, ultra thin] (p) -- (-0.75,1.5);
        \draw[dashed, color= gray, ultra thin] (p) -- ($(0,1.5) + (157:0.75)$);
        \draw[dashed, color= gray, ultra thin] (p) -- ($(0,1.5) + (135:0.75)$);
        \draw[dashed, color= gray, ultra thin] (p) -- ($(0,1.5) + (112:0.75)$);
        \draw[dashed, color= gray, ultra thin] (p) -- ($(0,1.5) + (68:0.75)$);
        \draw[dashed, color= gray, ultra thin] (p) -- ($(0,1.5) + (45:0.75)$);
        \draw[dashed, color= gray, ultra thin] (p) -- ($(0,1.5) + (23:0.75)$);
        \draw[dashed, color= gray, ultra thin] (p) -- ($(0,1.5) + (0:0.75)$);
         \draw (0,1.5) circle (.75);
         \draw[->,  thick] (-1.8,1.5)--(1.8, 1.5);
        \draw[->,  thick] (0,0)--(0,3);
        \pic[draw=red, line width=1.5pt,-, "$\textcolor{red}{\Omega}$", angle eccentricity= 1.3, angle radius= 1.5cm] { angle = b--a--c};
        \draw[ultra thick, color=blue] (-1.5,1.5)--(0.75,1.5);
        \node at (-1.3,1.7){${\color{blue}\Pi(\Omega)}$};
        \node at (.1,0.6){$\mathfrak{p}$};
          
        \end{tikzpicture}   
\caption{The stereographic projection in $d=2$ with $\mathfrak{p}=(0,-1)$. }
    \label{fig:stereo:2D}
\end{figure}
We employ the stereographic projection $ \Pi$ to parametrise $S^{d-1}$
 since it yields a particularly simple expression for the surface measure, with density $h^{d-1}$ on $\R^{d-1}$. Unlike classical polar coordinates, whose Jacobian involves increasingly complicated products of trigonometric functions as the dimension grows, the stereographic projection is a conformal mapping  and has a  dimension-uniform representation. 
    However, if $d=2$, one can simplify our analysis by employing  the classical polar coordinates for the unit sphere, i.e.,
    \[
    S^{1}= \{ x=(\cos\varphi, \sin\varphi)\in \R^2\colon \varphi\in [0,2\pi)\}.
    \]
    In this case, the cones in our consideration are, up to rotation, angular domains taking the form 
    \[\calD_\kappa:= \{x=(r\cos\varphi, r\sin \varphi)\in \R^2\setminus\{0\}\colon 0<r<\infty, 0<\varphi<\kappa\}, \quad \kappa\in (0,2\pi),\]
    and the surface measure is given by the Lebesgue measure restricted to $(0,\kappa)$. Clearly, $\calD_\kappa$ is a $C^\infty$-cone of the form~\eqref{eq:D:convention} with $\Omega:=\{x=(\cos\varphi, \sin\varphi)\in \R^2\colon \varphi\in (0,\kappa)\}$. 
\end{remark}
For a smooth Riemannian manifold $(M,g)$ and an arbitrary domain $G\subsetneq M$ with non-empty boundary $\partial G$ we write $$\rho_G(x):= \textup{dist}_M(x,\partial G), \qquad x \in G$$ for the distance in $M$ (with respect to $g$) of a point $x\in G$ to the boundary $\partial G$. We will exclusively deal with $M=\R^d$ and $M=S^{d-1}$ endowed with the standard Euclidean metric and its restriction to $S^{d-1}$, respectively.
Moreover, for all $\gamma\in\R$, we define
\[
w_\gamma^G(x):=\rho_G(x)^\gamma,\quad x\in G.
\]
Note that as $\Pi\colon S^{d-1}\setminus\{\mathfrak p\}\to \R^{d-1}$ is a $C^\infty$-diffeomorphism, it is in particular bi-Lipschitz on $\overline{\Omega}\subsetneq S^{d-1}$. It follows that
\begin{equation}\label{eq:OmegavsscrO}
    w_\gamma^\Omega(\zeta)\eqsim_{\gamma,\Omega} w_\gamma^{\calO}(\Pi(\zeta))\qquad \zeta\in \Omega.
\end{equation}
Let $\gamma,\nu\in\R$. We consider weights based on both the distance $\rho_\calD$ to the boundary of $\calD$ and on the distance to the vertex, i.e.,
\[
\rho_\circ(x):=\textup{dist}(x,\{0\})=\lvert x \rvert,\quad x\in\calD.
\]
In particular,  we will use power weights of the form  
\[
w_{\gamma,\nu}^\calD(x):=\rho_\circ(x)^\nu\bigg(\frac{\rho_\calD(x)}{\rho_\circ(x)}\bigg)^\gamma, \quad x\in\calD.
\]
We write 
\begin{align*}
\Phi\colon (0,\infty)\times \Omega&\to\R^d\setminus\{0\},
 &&(r,\zeta)\mapsto r\zeta, \\ \Phi_\Pi\colon (0,\infty)\times \calO &\to \R^d\setminus\{0\}, &&(r,y)\mapsto r \Pi^{-1}(y),
\end{align*}
for the transformation of polar coordinates into Cartesian coordinates and its parametrised version.
Distinguishing between the sphere cap with and without parametrisation, we define
\begin{align*}
    \widetilde{\calD}_\Omega
&:=
(0,\infty)\times\Omega 
= 
\Phi^{-1}(\calD_\Omega), \\ \widetilde\calD:=\widetilde\calD_\calO&:=(0,\infty)\times\calO = \Phi_\Pi^{-1}(\calD_\Omega),
\end{align*}
and 
\begin{align*}
w_{\gamma,\nu}^{\widetilde{\calD}_\Omega}(r,\zeta)
&:=
w_\nu^{\R_+}(r) \cdot w_\gamma^{\Omega}(\zeta)
\eqsim
w_{\gamma,\nu}^\calD(\Phi(r,\zeta)),
\qquad &&(r,\zeta)\in\widetilde{\calD}_\Omega,\\
w_{\gamma,\nu}^{\widetilde\calD}(r,y)&:= w_\nu^{\R_+}(r)\cdot w_\gamma^{\calO}(y)\eqsim w_{\gamma,\nu}^\calD(\Phi_\Pi(r,y)) \eqsim w_{\gamma,\nu}^{\widetilde \calD_\Omega}(r,\Pi^{-1}(y)), \qquad &&(r,y)\in \widetilde\calD.
\end{align*}
Moreover,  we let
\begin{align*}
    \Psi\colon \R\times \Omega&\to (0,\infty)\times \Omega,
\qquad && (z,\zeta)\mapsto (e^z,\zeta), \\ \Psi_\Pi\colon \R\times \calO &\to (0,\infty) \times \calO, \qquad && (z,y)\mapsto (e^z,y),
\end{align*}
be the transformation of Euler coordinates into polar coordinates
 and its parametrised variant. Then 
\begin{align*}
    \widehat{\calD}_\Omega&:= \R\times \Omega =\Psi^{-1}(\widetilde{\calD}_\Omega), \\ \widehat \calD:= \widehat \calD_\calO&:=\R\times \calO=\Psi_\Pi^{-1}(\widetilde\calD),
\end{align*}
and 
\begin{align*}
    w_{\gamma,\nu}^{\widehat{\calD}_\Omega}(z,\zeta)
&:=
e^{z\nu} w_\gamma^{\Omega}(\zeta)
=
w_{\gamma,\nu}^{\widetilde{\calD}_\Omega}(\Psi(z,\zeta)),
\qquad &&(z,\zeta)\in \widehat{\calD}_\Omega,\\
w_{\gamma,\nu}^{\widehat\calD}(z,y)&:= e^{z\nu} w_\gamma^\calO(y)
= w_{\gamma,\nu}^{\widetilde\calD}(\Psi_\Pi(z,y))
\eqsim w_{\gamma,\nu}^{\widetilde\calD_\Omega}(z,\Pi^{-1}(y)), \qquad &&(z,y)\in\widehat\calD.
\end{align*}

\begin{figure}[ht!]
    \centering
 \begin{tikzpicture}

    \coordinate (a) at (0,1.5);
    \coordinate (b) at (2,1.5);
    \coordinate (c) at (2,0.5);
    \coordinate (d) at (0,-1);
      \shade[inner color=orange, outer color=white] (a) circle (1.2cm);
      \draw[white, fill=white] (a) -- (b) -- (c) -- cycle;
    \draw[black, xshift= 0cm, yshift=1.5cm,  thick] (2,0) -- (0,0) -- (2,-1);
     
        \pic[draw=red, very thick, -, "$\textcolor{red}{\Omega}$", angle eccentricity= 1.3, angle radius= .75cm] { angle = b--a--c};
        \node at (1,2) {$\calD$ };
        \draw[->, very thick] (-1.5,1.5)--(2.2, 1.5);
        \draw[->, very thick] (0,0)--(0,3);

        \node at (2.1,1.2){$\R$};
        \node at (-.3,2.8){$\R$};
        \begin{scope}[xshift=4cm,yshift=-0.5cm]
                    \shade[ left color = orange!70!white, right color = white] (0,1.5) rectangle (2.5,3);
                    \draw (0,3) -- (2.5,3);
                    \draw (0,1.5)--(2.5,1.5);
            \draw[->,very thick] (0,2)--(3, 2);
        \draw[<-, very thick] (0,3.5)--(0,0.5);
        \draw[color=red, very thick] (1,1.5)--(1,3);
        \node at (3.05,1.65){$\R_+$};
        \node at (-.3,3.3){$\R$};
        \node at (1.5,2.25){${\color{red}\Pi(\Omega)}$};
        \node at (2.5,2.5){$\widetilde\calD$};
        \end{scope}
        \begin{scope}[xshift=9cm,yshift=-0.5cm]
            \shade[left color = orange!40!white, right color = white] (0,1.5) rectangle (2.5,3);

            \draw (0,3) -- (2.5,3);
            \draw (0,1.5)--(2.5,1.5);
            \draw[->,very thick] (-.5,2)--(3, 2);
            \draw[<-, very thick] (.5,3.5)--(.5,0.5);
            \draw[color=red, very thick] (1,1.5)--(1,3);
            \node at (2.95,1.7){$\R$};
            \node at (.2,3.3){$\R$};
            \node at (1.5,2.25){${\color{red}\Pi(\Omega)}$};
            \node at (2.5,2.5){$\widehat\calD$};
        \end{scope}
        \draw[<->, bend right=45] (2,-.2) to (4,-.2);
        \node at (3,0){$\Phi_\Pi$};
        \draw[<->, bend right=45] (7,-.2) to (9,-.2);
        \node at (8,0){$\Psi_\Pi$};

\end{tikzpicture}
    \caption{Illustration of $\calD,\widetilde{\calD},$ and $\widehat\calD$ with $d=2$.}
    \label{fig:enter-label1}
\end{figure}
Next, we describe how these transformations act on weighted $L^p$-spaces.
To keep the exposition accessible, we work exclusively with the parametrised maps $\Phi_\Pi$ and $\Psi_\Pi$ and avoid vector-valued distributions on manifolds.
Let $X$ be a Banach space. We write 
 \begin{align*}
 T_\Phi&:=(\scrD'(\calD;X)\ni u \mapsto u\circ  \Phi_\Pi \in \scrD'(\widetilde{\calD};X))\\
 \intertext{and}
 T_{\Psi}&:=(\scrD'(\widetilde{\calD};X)\ni u \mapsto u\circ \Psi_\Pi \in \scrD'(\widehat{\calD};X) )
 \end{align*}
for the transformations corresponding to the diffeomorphisms $\Phi$ and $\Psi$ above. For $1<p<\infty$ we have that 
 \[T_\Phi\in\sL(L^p(\calD,w_{\gamma,\nu}^\calD;X), L^p(\widetilde{\calD},w_{\gamma,\nu+d-1}^{\widetilde{\calD}}\cdot h^{d-1};X))\]
 and 
 \[T_\Psi \in \scrL(L^p(\widetilde{\calD},w_{\gamma,\nu}^{\widetilde{\calD}}\cdot h^{d-1};X),L^p(\widehat{\calD},w_{\gamma,\nu+1}^{\widehat{\calD}}\cdot h^{d-1};X))\]
 are isomorphisms, where $h$ is as in \eqref{eq:surface:measure:stereographic:projection}. 
For all $ a\in \R $ let 
 \[
 M_a:= ( C_c^\infty(\widehat{\calD})\ni \varphi \mapsto \{(z,\zeta)\mapsto e^{az}\varphi(z,\zeta)\}\in  C_c^\infty(\widehat{\calD}))
 \]
 and for all $u\in\scrD'(\widehat{\calD};X)$ let $$(M_au)(\varphi):=u(M_a\varphi),\qquad \varphi\in C_c^\infty(\widehat{\calD}).$$
The operator $M_a \in \scrL( L^p(\widehat{\calD},w_{\gamma,\nu}^{\widehat{\calD}}\cdot h^{d-1};X),L^p(\widehat{\calD},w_{\gamma,\nu-ap}^{\widehat{\calD}}\cdot h^{d-1};X))$ is an isometric isomorphism. 
Let $T_a\colon \scrD'(\calD;X)\rightarrow \scrD'(\widehat{\calD};X)$ be given by
\begin{equation}\label{transformation:T_a}
	T_a:=T_{\Phi,\Psi,a}:=M_a\circ T_\Psi\circ T_\Phi.
\end{equation}
By construction, 
\begin{equation}\label{eq:Ta:isom:Lp}
T_a \colon L^p(\calD,w_{\gamma,\nu}^\calD;X)\to L^p(\widehat{\calD},w_{\gamma,\nu+d-ap}^{\widehat{\calD}}\cdot h^{d-1};X)
\end{equation}
 is an isomorphism for $a\in\R$.
 Using \eqref{eq:OmegavsscrO} and that, due to the boundedness of $\calO$,  $h$ is bounded from above and below on $\calO$, we have
\begin{equation}\label{eq:stereographicProjection:isom:Lp:surface:measure}
    L^p(\Omega,w_\gamma^\Omega;X)\simeq_{\Omega,\gamma} L^p(\calO,w_\gamma^{\calO}h^{d-1};X) \simeq_{\Omega}L^p(\calO,w_\gamma^{\calO};X).
\end{equation}
This and \eqref{eq:Ta:isom:Lp} hence prove that 
\[T_a\colon L^p(\calD,w_{\gamma,\nu}^\calD;X)\to L^p(\widehat\calD,w_{\gamma,\nu+d-ap}^{\widehat\calD};X)\]
is an isomorphism for all $a\in \R$. 
For the mapping behaviour of  weak derivatives under $T_a$, we refer to Lemma~\ref{lem:DifferentialCalculus:EulerCoordinate} and Lemma~\ref{DifferentialCalculus:Euclid:Euler:Sobolev:normequivalence}.

\medskip

We will need the following auxiliary result in order to identify the parameter ranges for which we obtain boundedness of the $H^\infty$-calculus for the negative Dirichlet Laplacian by means of extrapolation arguments.
In the statement and throughout the proof, for all $0<r<\infty$ we write 
\[
B(x,r)
:= 
\{y\in\calD\colon \lvert y-x\rvert<r\}
=
B^{\R^d}(x,r)\cap\calD
\]
for all $x\in\overline{\calD}$, whereas 
\[
S(x,r):=\Big\{y\in\calD\colon \bigl\lvert \lvert y\rvert-\lvert x\rvert\bigr\rvert<r,\,\, \mathrm{dist}_{S^{d-1}}\bigl(\tfrac{x}{\lvert x\rvert},\tfrac{y}{\lvert y\rvert}\bigr)
<\tfrac{r}{\lvert x\rvert}\Big\}
\]
for all $x\in\overline{\calD}\setminus\{0\}$; $S(0,r):=B(0,r)$.
\begin{figure}
    \centering
    \begin{tikzpicture}
    \coordinate (a) at (0,0);
    \coordinate (b) at (8,0);
    \coordinate (c) at (7,5);

    \draw[black, very thick] (b)--(a)--(c);
    
    \coordinate (x) at (6,2);
    \draw[thick] (x) circle (0.75);
    \draw[thick] (x) circle (1.5);
    \node at (x) [below]{$x$};

    \coordinate (d) at (8,1.475);
    \coordinate (e) at (8, 3.975);

         \pic[draw = gray, thick, dashed, -,angle eccentricity = 0.5, angle radius= 5.45 cm]{ angle = b--a--c};
        \pic[draw = gray,  thick, dashed, -,angle eccentricity = 0.5, angle radius= 7.2 cm]{ angle = b--a--c};

        \draw[gray,  thick, dashed] (a)--(d);
        \draw[gray,  thick, dashed] (a)--(e);

        \coordinate (f) at (8,2.65);
        \pic[draw = red,  thick, -,angle eccentricity = 0.5, angle radius= 6.325 cm]{ angle = f--a--e};

        \coordinate (af1) at ($(a)!0.645!(f)$);
        \coordinate (af2) at ($(a)!0.75!(f)$);
        \draw[red, thick](af1)--(af2);

        \fill (x) circle (2pt);
        
        \pic[draw = black, very thick, -,angle eccentricity = 0.5, angle radius= 5.45 cm]{ angle = d--a--e};
        \pic[draw = black, very thick, -,angle eccentricity = 0.5, angle radius= 7.2 cm]{ angle = d--a--e};

        \coordinate (ad1) at ($(a)!0.67!(d)$);
        \coordinate (ad2) at ($(a)!0.885!(d)$);
        \draw[black, very thick] (ad1)--(ad2);

        \coordinate (ae1) at ($(a)!0.61!(e)$);
        \coordinate (ae2) at ($(a)!0.807!(e)$);
        \draw[black, very thick] (ae1)--(ae2);
    
        \draw[->,thick] (8,2.25)--(7,2);
        \node at (8.8,2.25){$S(x,r)$};

        \node at (5.4,1.9){$\textcolor{red}{r}$};
        \node at (6,2.5){$\textcolor{red}{r}$};
    \end{tikzpicture}
    \caption{Illustration of Lemma \ref{lem:doubling}\ref{it:ball:windshield} for $d=2$.}
\end{figure}

\begin{lemma}\label{lem:doubling}
Let $\calD:=\calD_\Omega$ be a cone of the form~\eqref{eq:D:convention}, where $\Omega\subseteq S^{d-1}$ is a domain admitting a $C^1$-boundary and such that $\overline{\Omega}\subsetneq S^{d-1}$. Then the following assertions hold. 
\begin{enumerate}[label=\textup{(\roman*)}]
\item\label{it:ball:windshield} We have 
\[
B(x,\tfrac{r}{\pi})\subseteq S(x,r)\subseteq B(x,3r),\qquad x\in \overline{\calD},\, r>0.
\]
\item\label{it:mixedweights:doubling} Let $-1<\gamma<\infty$ and let $-d<\nu<\infty$. Then
\begin{equation}\label{eq:doubling:windshield}
w_{\gamma,\nu}^\calD(S(x,2r))
\lesssim_{\gamma,\nu}
w_{\gamma,\nu}^\calD(S(x,r)), \qquad x\in\calD,\, r>0.
\end{equation}
In particular, $w_{\gamma,\nu}^\calD$ satisfies the volume doubling property, and
\begin{equation}\label{eq:ball:windshield:equiv}
w_{\gamma,\nu}^\calD(B(x,r))
\eqsim_{\gamma,\nu}
w_{\gamma,\nu}^\calD(S(x,r)),\qquad x\in\overline{\calD},\,r>0.
\end{equation}

\item\label{it:mixedweights:Ap}     Let $-1<\Theta<\infty$, let $-d<\theta<\infty$, and let $1<p<\infty$. If 
\begin{equation}\label{weight:mixed:Muckenhoupt:range}\begin{aligned}
         -(1+\Theta) &<\gamma<(1+\Theta)(p-1),\\
     -(d+\theta) &<\nu<(d+\theta)(p-1).
\end{aligned}
 \end{equation}
 then $w_{\gamma,\nu}^\calD\in A_p(w_{\Theta,\theta}^\calD)$.
\end{enumerate}
\end{lemma}

\begin{proof}
Note that for any two non-antipodal points $p_1,p_2 \in S^{d-1}$, the shortest path in $S^{d-1}$ from $p_1$ to $p_2$ lies in the intersection of $S^{d-1}$ and the two-dimensional plane spanned by $p_1,p_2$, and the origin. Therefore, computations for $\mathrm{dist}_{S^{d-1}}(p_1,p_2)$ reduce to $d=2$. 

 We first prove~\ref{it:ball:windshield}. For $x=0$ the statement is trivial.
 Let $x\in\overline{\calD}\setminus\{0\}$ and let $r>0$. To verify the first inclusion, let $y\in B(x,r/\pi)$. Then  
 \[
 \big| | y| - \lvert x\rvert \big| \leq \lvert y-x\rvert <r.
 \]
Moreover, if $r/\pi>\lvert x\rvert$, then 
 \[
 \lvert x\rvert\cdot 
 \mathrm{dist}_{S^{d-1}}\big(\tfrac{x}{\lvert x\rvert},\tfrac{y}{\lvert y\rvert}\big)
 \leq 
 \lvert x\rvert\cdot  \pi 
 <
 r,
 \]
whereas if $r/\pi\leq \lvert x\rvert$,
\[
 \mathrm{dist}_{S^{d-1}}\big(\tfrac{x}{\lvert x\rvert},\tfrac{y}{\lvert y\rvert}\big)
 \leq 
 \arcsin\big(\tfrac{r}{\pi\lvert x\rvert}\big) 
 \leq
 \frac{1}{2}\frac{r}{\lvert x\rvert}
 < \frac{r}{\lvert x\rvert}.
\]
Thus, the first inclusion holds. To prove the second inclusion, let $y\in S(x,r)$. In this case, if $r\leq \lvert x\rvert$, then
\[
\lvert x-y\rvert \leq \big| \lvert x\rvert - \lvert y\rvert\big| + (\lvert x\rvert \lor \lvert y\rvert)\cdot \mathrm{dist}_{S^{d-1}}\big(\tfrac{x}{\lvert x\rvert},\tfrac{y}{\lvert y\rvert}\big)
< r + (r+\lvert x\rvert)\cdot \tfrac{r}{\lvert x\rvert}
\leq 3r,
\]
whereas if $r>\lvert x\rvert$, then
\[
\lvert x-y\rvert 
\leq 
\lvert x\rvert + \lvert y\rvert
<
\lvert x\rvert + \lvert x\rvert + r
<
3r.
\]
Thus, in both cases, $y\in B(x,3r)$, so that also the second inclusion in~\ref{it:ball:windshield} is verified.

For \ref{it:mixedweights:doubling}, we note that by \ref{it:ball:windshield}, to verify~\eqref{eq:ball:windshield:equiv} with $x\in\calD$ and that $w_{\gamma,\nu}^\calD$ has the volume doubling property, it suffices to show that~\eqref{eq:doubling:windshield} holds.
To this end, note that for all $x\in\calD$ and all $r>0$, 
\begin{align}
    w_{\gamma,\nu}^\calD(S(x,r))
    &\eqsim
    w_{\gamma,\nu+d-1}^{\widetilde{\calD}_\Omega} (\Phi^{-1}(S(x,r)))\notag\\[.5em]
    &\eqsim_\gamma
    w_{\nu+d-1}^{\R_+}(B^{\R_+}(\lvert x\rvert,r)) \cdot 
    w_\gamma^{\Omega}\Big(B^{\Omega}\big(\tfrac{x}{|x|},\tfrac{r}{\lvert x\rvert}\big)\Big)\label{eq:split:S}\\
    &\eqsim_{\gamma,\Omega} w_{\nu+d-1}^{\R_+}(B^{\R_+}(\lvert x\rvert,r)) \cdot 
    w_\gamma^{\calO}\Big(B^{\calO}\big(\Pi\big(\tfrac{x}{|x|}\big),\tfrac{r}{\lvert x\rvert}\big)\Big). \notag
\end{align}
Thus,~\eqref{eq:doubling:windshield} holds whenever both $w_{\nu+d-1}^{\R_+}$ and $w_\gamma^{\calO}$ 
satisfy the volume doubling property. 
This is the case whenever $\gamma>-1$ and $\nu>-d$, which follows, e.g., from~\cite[Example~7.1.6]{Grafakos2014classical} and a localization argument.
The extension to $x\in \partial\calD$ follows using the dominated convergence theorem. 

Finally, to prove~\ref{it:mixedweights:Ap} we note that by \ref{it:mixedweights:doubling}, for all $-1<\Theta<\infty$, all $-d<\theta<\infty$, and all $\gamma,\nu\in\R$ satisfying~\eqref{weight:mixed:Muckenhoupt:range}, 
\begin{align*}
[w_{\gamma,\nu}^\calD]_{A_p(w_{\Theta,\theta}^\calD)}
&\lesssim_{\gamma,\nu,\Theta,\theta,p} 
\sup_{\substack{x\in\calD\\ r>0}}
\bigg\{
\bigg(\frac{1}{w_{\Theta,\theta}^\calD(S(x,r))} \int_{S(x,r)} w_{\gamma,\nu}^\calD\,\dd w_{\Theta,\theta}^\calD\bigg)\,\times\\
&\qquad\qquad\qquad\qquad \times\,
\bigg(\frac{1}{w_{\Theta,\theta}^\calD(S(x,r))} \int_{S(x,r)} \big(w_{\gamma,\nu}^\calD\big)^{-\frac{1}{p-1}}\,\dd w_{\Theta,\theta}^\calD\bigg)^{p-1}\bigg\}\\[1em]
&\lesssim
[w_\nu^{\R_+}]_{A_p(w_{\theta+d-1}^{\R_+})}
\cdot
[w_\gamma^{\calO}]_{A_p(w_\Theta^{\calO})}
<\infty,
\end{align*}
where the finiteness may be verified along the lines of~\cite[Example~7.1.7]{Grafakos2014classical} and a localization argument.
\end{proof}

\section{Weighted Sobolev spaces}\label{sec:WSob}

In this section we collect definitions and properties of suitable (vector-valued) weighted  Sobolev spaces that we need for our analysis. Recall that, throughout this manuscript, $\calD=\calD_\Omega$ is a cone of the form~\eqref{eq:D:convention}, where $\Omega\subseteq S^{d-1}$ is a domain, such that $\overline{\Omega}\subsetneq S^{d-1}$. Throughout this section, we borough the notation regarding the domain $\calD$, its parametrizations and related weights from Section~\ref{sec:domain:D}.

\subsection{Uniformly weighted Sobolev spaces}\label{sec:Sobolev:uniform}
We start with the definition and properties of weighted Sobolev spaces with weights that are chosen uniformly for all derivatives, i.e., the integrability condition is the same for all derivatives.

Let $G\subseteq\R^d$ be a domain, i.e.,  a connected open set.
Let $1<p<\infty$, let $k\in\N_0$, let $X$ be a Banach space, and let $w\in L^1_\textup{loc}(G)$ be a weight such that $w^{-\frac{1}{p-1}}\in L^1_\textup{loc}(G)$. Then
\[
W^{k,p}(G,w;X):=\big\{ u\in L^p(G,w;X)\colon D^\alpha u\in L^p(G,w;X),\,\alpha\in\N_0^d,\lvert\alpha\rvert\leq k\big\}
\]
is the (vector-valued) weighted Sobolev space 
of order $k$, which becomes a Banach space when endowed with the norm
\[
\nrm{u}_{W^{k,p}(G,w;X)}:=\sum_{\lvert\alpha\rvert\leq k}\nrm{D^\alpha u}_{L^p(G,w;X)},\quad u\in W^{k,p}(G,w;X).
\]
We write $W^{k,p}(G,w):=W^{k,p}(G,w;\C)$ and omit the weight if $w\equiv 1$.

In order to formulate boundary value problems within the weighted Sobolev spaces introduced above, we need to make sense of traces on the boundary $\partial G$ of $G$. We will do so by following the expositions in \cite{LindemulderVeraar2020} and \cite{lindemulder2025functionalcalculusweightedsobolev}. Let $k\in \N$ and let $G\subsetneq \R^d$ be a bounded domain. Then the trace $\textup{Tr}= \textup{Tr}|_{\partial G}$ onto the boundary $\partial G$ of $G$ is well-defined on $W^{k,p}(G,w_{\gamma}^G;X)$, provided 
\begin{equation}\label{eq:trace:range:Ck}
        \gamma \in (-1,kp-1)\setminus\{jp-1\colon j\in \N\}  \quad \text{ and }\quad G\in C^k
\end{equation}
or 
\begin{equation}\label{eq:trace:range:C1}
    \gamma\in ((k-1)p-1,kp-1) \quad \text{and}\quad G\in C^1
\end{equation}
see \cite[Section 3]{LindemulderVeraar2020} or \cite[Section 3]{lindemulder2025functionalcalculusweightedsobolev}, respectively.
Hence, in either case  we can define 
\[W^{k,p}_\Dir(G,w_\gamma^G;X):= \{u \in W^{k,p}(G,w_\gamma^G;X)\colon \textup{Tr}(u)=0\}. \]
It is worth noting that, whenever both \eqref{eq:trace:range:Ck} and \eqref{eq:trace:range:C1} hold, i.e., if 
\[
\gamma\in((k-1)p-1,kp-1)\quad \text{and}\quad G\in C^k,
\]
then the weighted Sobolev spaces with Dirichlet boundary conditions appearing in \cite{LindemulderVeraar2020} and in \cite{lindemulder2025functionalcalculusweightedsobolev} coincide. In this case, both spaces are endowed with the same norm and share a common dense subspace, $C_{c,\Dir}^k(\overline{G}).$

Below we will particularly make use of the weighted Sobolev spaces $W^{k,p}(\widehat\calD,w_{\gamma,\nu}^{\widehat\calD};X)$ on $\widehat\calD$ with $1<p<\infty$, $k\in\N$, $\gamma,\nu\in \R$.
In this case,
\[
W^{k,p}(\widehat{\calD},w_{\gamma,\nu}^{\widehat{\calD}};X) \hookrightarrow W^{k,p}(\calO,w_\gamma^\calO; L^p(\R, e^{\nu z}\dd z ;X));
\]
recall that $\calO=\Pi(\Omega)$, see Section~\ref{sec:domain:D}. Thus, applying the above trace result with values in $L^p(\R,e^{\nu z}\dd z;X)$, if~\eqref{eq:trace:range:Ck} or~\eqref{eq:trace:range:C1} holds with $G\curvearrowleft \calO$, then 
\begin{align*}
W^{k,p}_\textup{Dir}(\widehat{\calD},w_{\gamma,\nu}^{\widehat{\calD}};X)
&:=
W^{k,p}(\widehat{\calD},w_{\gamma,\nu}^{\widehat{\calD}};X)
\cap
W^{k,p}_\Dir(\calO,w_\gamma^{\calO};L^p(\R,e^{\nu z}\dd z;X))\\
&\phantom{:}=
\big\{u\in W^{k,p}(\widehat{\calD},w_{\gamma,\nu}^{\widehat{\calD}};X)\colon \textup{Tr}(u)=0\big\}
\end{align*}
is a well-defined closed subspace of $W^{k,p}(\widehat{\calD},w_{\gamma,\nu}^{\widehat{\calD}};X)$.

The following lemma will be used in Section~\ref{sec:Laplace:hom} for the characterization of the domain of the Dirichlet Laplacian. The proof is exactly the same as the proof of \cite[Corollary 3.19]{LindemulderVeraar2020}.

\begin{lemma}\label{lem:mixedDeriv}
 Let $1<p<\infty$, $\gamma\in (-1,2p-1)\setminus\{p-1\}$, and assume that $\calO\in C^2$. Then
  
    \[ W^{2,p}(\R; L^p(\calO,w_\gamma^\calO))
    \cap 
    L^p(\R;W^{2,p}(\calO,w_\gamma^\calO))
    =
    W^{2,p}(\widehat{\calD},w_{\gamma,0}^{\widehat{\calD}}) \]
    with equivalent norms.

\end{lemma}

\subsection{Kondratiev-type weighted Sobolev spaces}\label{sec:Kondratiev}

In Section~\ref{sec:Dirichlet-Laplacian} we will be looking for suitable spaces to describe the domain of the Dirichlet Laplacian on $\calD$. In particular, we are going to choose $W^{2,p}_\Dir(\widehat\calD,w_{\gamma,\nu}^{\widehat\calD})$ as the domain of a suitable transformation of the Laplacian into Euler coordinates. When transforming back, we are going to end up with weighted Sobolev spaces that fall outside of the class considered in Section \ref{sec:Sobolev:uniform} as the powers on $\rho_\circ$ will vary with the order of the derivative. They remind of the classical Kondratiev spaces with an additional power weight of the distance to the boundary of the domain. In this section we introduce these spaces and discuss some fundamental properties.

Let $X$ be a Banach space, let  $1<p<\infty$, let $k\in \N_0$, and let $\gamma,\nu\in\R$.
Then,
\[
\Dot{W}^{k,p}(\calD,w_{\gamma,\nu}^\calD;X)
:=
\big\{u\in\scrD'(\calD;X)\colon D^\alpha u\in L^p(\calD,w_{\gamma,\nu+(\lvert\alpha\rvert-k)p}^\calD;X), \,\alpha\in\N_0^d,\,\lvert\alpha\rvert\leq k\big\},
\]
which becomes a Banach space when endowed with the norm
\[
\nrm{u}_{\Dot{W}^{k,p}(\calD,w_{\gamma,\nu}^\calD;X)}
:=
\sum_{\lvert\alpha\rvert\leq k}\nrm{D^\alpha u}_{L^p(\calD,w_{\gamma,\nu+(\lvert\alpha\rvert-k)p}^\calD;X)},
\quad u\in\Dot{W}^{k,p}(\calD,w_{\gamma,\nu}^\calD;X)
.
\]
Note that, 
due to Lemma~\ref{DifferentialCalculus:Euclid:Euler:Sobolev:normequivalence},
\begin{equation}\label{eq:Ta2:isom:noDir}
T_a\colon \Dot{W}^{k,p}(\calD,w_{\gamma,\nu}^\calD;X)
\to
W^{k,p}(\widehat{\calD},w_{\gamma,\nu+d-(k+a)p}^{\widehat{\calD}};X)
\end{equation}
 is an isomorphism for  $a\in\R$. 
Thus, 
whenever $\gamma$ and $\calO$ are such that \eqref{eq:trace:range:Ck} or \eqref{eq:trace:range:C1} holds with $G\curvearrowleft \calO$,
we have for  $\nu_a:=\nu+d-(k+a)p$ that 
\begin{align}
\Dot{W}^{k,p}_\Dir(\calD,w_{\gamma,\nu}^\calD;X)
&:= 
T_a^{-1}\big(W^{k,p}_\Dir(\widehat\calD,w_{\gamma,\nu_a}^{\widehat\calD};X)\big)\label{eq:Ta2:isom}\\
&\phantom{:}\simeq
\big\{
u\in \Dot{W}^{k,p}(\calD,w_{\gamma,\nu}^\calD;X)
\colon 
T_au\in W^{k,p}_\Dir(\calO,w_\gamma^{\calO};L^p(\R,e^{\nu_az}\dd z;X))
\big\}\notag\\
&\phantom{:}=
\big\{
u\in \Dot{W}^{k,p}(\calD,w_{\gamma,\nu}^\calD;X)
\colon \mathrm{Tr}(T_a u)=0
\big\}\notag
\end{align}
is a well-defined closed subspace of $\Dot{W}^{k,p}(\calD,w_{\gamma,\nu}^\calD;X)$, independent of $a\in\R$.
We will see in Proposition~\ref{prop:Laplace:hom} that, for a certain range of $\gamma,\nu\in\R$, $\Dot{W}_\Dir^{2,p}(\calD,w_{\gamma,\nu}^\calD)$ is a homogeneous space in the sense that
\[
\nrm{u}_{\Dot{W}_\Dir^{2,p}(\calD,w_{\gamma,\nu}^\calD)}
\eqsim\nrm{D^2u}_{L^p(\calD,w_{\gamma,\nu}^\calD)},
\quad
u\in \Dot{W}_\Dir^{2,p}(\calD,w_{\gamma,\nu}^\calD).
\]

The following can be said about density of smooth functions for the spaces introduced above. We refer to the introduction for the notation we use for spaces of smooth functions. 

\begin{lemma}\label{lem:smooth:dense}
Let $1<p<\infty$, $\nu\in\R$, $k\in\N$, and $\gamma\in(-1,kp-1)\setminus\{jp-1\colon j\in\N\}$. Then the following assertions hold.
\begin{enumerate}[label=\textup{(\roman*)}]
    \item\label{it:smooth:dense:Dir} If $\calO \in C^k$, then $C_{c,\Dir}^k(\overline{\calD}\setminus\{0\})$ is a dense subspace of both $\Dot{W}^{k,p}_\Dir(\calD,w_{\gamma,\nu}^\calD)$ and $\Dot{W}^{k,p}_\Dir(\calD,w_{\gamma,\nu}^\calD)\cap L^p(\calD,w_{\gamma,\nu}^\calD)$.
    \item\label{it:smooth:dense:test} If $\calO \in C^1$ and $\gamma>(k-1)p-1$, then $C_c^\infty(\calD)$ is a dense subspace of both $\Dot{W}^{k,p}_\Dir(\calD,w_{\gamma,\nu}^\calD)$ and $\Dot{W}^{k,p}_\Dir(\calD,w_{\gamma,\nu}^\calD)\cap L^p(\calD,w_{\gamma,\nu}^\calD)$.
\end{enumerate}
\end{lemma}
\begin{proof} We first prove~\ref{it:smooth:dense:Dir} in three steps.

\medskip

\noindent\emph{Step 1.} We start by proving that $C_{c,\Dir}^\infty(\overline{\calD}\setminus\{0\})$ is dense in $\Dot{W}^{k,p}_\Dir(\calD,w_{\gamma,\nu}^\calD)$.
Since for all $a\in\R$,
\[
T_a(C_{c,\Dir}^\infty(\overline{\calD}\setminus\{0\}))=C_{c,\Dir}^\infty(\overline{\widehat\calD})
\]
by~\eqref{eq:Ta2:isom}, we only have to check that $C_{c,\Dir}^\infty(\overline{\widehat\calD})$ is dense in $W^{k,p}_\Dir(\widehat\calD,w_{\gamma,0}^{\widehat\calD})$.
This follows from a standard localisation argument and~\cite[Proposition~4.8]{roodenburg2025complexinterpolationweightedsobolev}, which shows that, if $\gamma$ is as required in our assumptions, then $C^\infty_{c,\Dir}(\overline{\R^d_+})$ is dense in $W^{k,p}_\Dir(\R^d_+,w_\gamma^{\R^d_+})$.

\medskip

\noindent\emph{Step 2.} Let $v\in W^{k,p}_\Dir(\widehat{\calD},w_{\gamma,0}^{\widehat\calD})\cap L^p(\widehat{\calD},w_{\gamma,kp}^{\widehat\calD})$. Let $(\widehat\chi_n)_{n\in\N}\subseteq C^\infty(\R)$ be such that for all $n\in\N$, $\widehat\chi_n(z)\in [0,1]$ for $z\in\R$, $\widehat\chi_n(z)=1$ for $\lvert z\rvert\leq n$, $\widehat\chi_n(z)=0$ for $\lvert z\rvert\geq n+2$, and $\lvert\partial^\alpha \widehat\chi_n(z)\rvert \leq C_\alpha<\infty$ for all $z\in\R$ and all $\alpha\in\N_0$. 
Let $v_n(z,y):=\widehat\chi_n(z)v(z,y)$, $(z,y)\in\R\times \calO$, $n\in\N$.
By the properties of $(\widehat\chi_n)_{n\in\N}$, the $v_n$ have compact supports in $\R\times \overline\calO$. 
Moreover, using the Leibniz formula we obtain for all $\alpha\in\N_0^d$ with $\lvert\alpha\rvert\leq k$, 
\[
\lvert D^\alpha v_n\rvert \lesssim \sum_{\beta\leq\alpha}\, \lvert D^\beta v \rvert \text{ on } \widehat{\calD}, \quad n\in\N,
\quad
\text{and}
\quad
D^\alpha v_n\to D^\alpha v \text{ a.e. on $\widehat\calD$ for } n\to\infty.
\]
Thus, since $v\in W^{k,p}_\Dir(\widehat{\calD},w_{\gamma,0}^{\widehat\calD})\cap L^p(\widehat{\calD},w_{\gamma,kp}^{\widehat\calD})$, so are the $v_n$ and~
\[
\nrm{v-v_n}_{W^{k,p}_\Dir(\widehat{\calD},w_{\gamma,0}^{\widehat\calD})\cap L^p(\widehat{\calD},w_{\gamma,kp}^{\widehat\calD})}
\longrightarrow 0 \quad \text{ as }\quad  n\to \infty
\]
by the dominated convergence theorem.
Thus, by~\eqref{eq:Ta:isom:Lp} and~\eqref{eq:Ta2:isom}, if we set $\chi_n:= T_0^{-1} \, \widehat\chi_n \in C^\infty_c(\overline\calD \setminus \{0\})$, $n\in\N$, then for all $u\in \Dot{W}^{k,p}(\calD,w_{\gamma,\nu}^\calD)\cap L^p(\calD,w_{\gamma,\nu}^\calD)$ it holds that 
\[
\nrm{u-\chi_n u}_{\Dot{W}^{k,p}(\calD,w_{\gamma,\nu}^\calD)\cap L^p(\calD,w_{\gamma,\nu}^\calD)} \longrightarrow 0 \quad \text{ as }\quad  n\to \infty,
\]
where $\textrm{supp}(\chi_n)\subseteq \{ x\in \overline{\calD}\colon  e^{-n-2}\leq |x|\leq e^{n+2}\}$, $n\in\N$.

\medskip

\noindent\emph{Step~3.} Let $u\in \Dot W_\Dir^{k,p}(\calD,w_{\gamma,\nu}^\calD)\cap L^p(\calD,w_{\gamma,\nu}^\calD)$, let $(\chi_n)_{n\in\N}$ be as constructed in Step~2 and let $(u_n)_{n\in\N}\subseteq C_{c,\Dir}^\infty(\overline{\calD}\setminus\{0\})$ be a sequence approximating $u$ in $\Dot W^{k,p}_\Dir(\calD,w_{\gamma,\nu}^\calD)$; the latter exists due to Step~1.
Let $(u_{m,n})_{m,n\in\N}:=(\chi_mu_n)_{m,n\in\N}\subseteq C_{c,\Dir}^\infty(\overline{\calD}\setminus\{0\})$.
Then, 
\begin{align*}
        &\nrm{u - \chi_mu_n}_{\Dot{W}^{k,p}(\calD,w^\calD_{\gamma,\nu})\cap L^p(\calD,w^\calD_{\gamma,\nu})}\\
        &\le \nrm{u-\chi_mu}_{\Dot{W}^{k,p}(\calD,w^\calD_{\gamma,\nu})\cap L^p(\calD,w^\calD_{\gamma,\nu})}
        +
        \nrm{\chi_mu-\chi_mu_n}_{\Dot{W}^{k,p}(\calD,w^\calD_{\gamma,\nu})\cap L^p(\calD,w^\calD_{\gamma,\nu})}\\
        &\le \nrm{u-\chi_mu}_{\Dot{W}^{k,p}(\calD,w^\calD_{\gamma,\nu})\cap L^p(\calD,w^\calD_{\gamma,\nu})}
        +
        C_{\textrm{supp}(\chi_m),\nu,k}
        \nrm{\chi_m(u-u_n)}_{\Dot{W}^{k,p}(\calD,w^\calD_{\gamma,\nu})}\\
        &\le \nrm{u-\chi_mu}_{\Dot{W}^{k,p}(\calD,w^\calD_{\gamma,\nu})\cap L^p(\calD,w^\calD_{\gamma,\nu})}
        + 
        C_{\textrm{supp}(\chi_m),\nu,k}C_{\chi_1,k}
        \nrm{u-u_n}_{\Dot{W}^{k,p}(\calD,w^\calD_{\gamma,\nu})}.
    \end{align*}
Hence, for every $\varepsilon>0$, if we choose $m=m(\varepsilon)\in\N$ and $n=n(m,\varepsilon,\nu,k)\in \N$ large enough, we can construct a function $u_\varepsilon:=\chi_{m}u_{n}$, such that 
\[
\nrm{u - u_\varepsilon}_{\Dot{W}^{k,p}(\calD,w^\calD_{\gamma,\nu})\cap L^p(\calD,w^\calD_{\gamma,\nu})}<\varepsilon,
\]
which proves~\ref{it:smooth:dense:Dir}.

Part~\ref{it:smooth:dense:test} may be proven mutatis mutandis. For Step~1 use that $T_a(C_c^\infty(\calD)) = C_c^\infty(\widehat\calD)$, $a\in\R$, and that $C_c^\infty(\widehat\calD)$ is dense in $W^{k,p}_\Dir(\widehat\calD,w_{\gamma,0}^{\widehat\calD})$ if $(k-1)p-1<\gamma<kp-1$.
The latter follows by a localization argument and \cite[Lemma~3.4]{Lindemulder2024functionalcalculusweightedsobolev}, which yields that $C_c^\infty(\R_+^d)$ is dense in $W_\Dir^{k,p}(\R_+^d,w_{\gamma}^{\R_+^d})$, provided $(k-1)p-1<\gamma<kp-1$. 
\end{proof}

We will also make use of the following consequences of Hardy's inequality.

\begin{lemma}\label{lem:Hardy:double}
Let $1<p<\infty$, let $k\in\N$, let $\gamma,\nu\in\R$, and assume that {\color{magenta}$\calO\in C^1$}. Then 
\begin{align}
\Dot W^{k,p}(\calD,w_{\gamma,\nu}^\calD)&\hookrightarrow \Dot W^{k-1,p}(\calD,w_{\gamma-p,\nu-p}^\calD), \quad &&\gamma > p-1,\label{eq:Hardy:angle:high}\\
\Dot W^{k,p}_\Dir(\calD,w_{\gamma,\nu}^\calD) &\hookrightarrow \Dot W^{k-1,p}_{\Dir}(\calD,w_{\gamma-p,\nu-p}^\calD),\quad && (k-1)p-1<\gamma<kp-1,\label{eq:Hardy:angle:arbitrary:Dir}
\end{align}
where $\Dot{W}^{0,p}_\Dir(\calD,w^\calD_{\gamma-p,\nu-p}):= L^p(\calD,w_{\gamma-p,\nu-p}^\calD)$.
\end{lemma}
\begin{proof}
For \eqref{eq:Hardy:angle:high}, by \eqref{eq:Ta2:isom:noDir}, it suffices to show
\begin{align*}
W^{k,p}(\widehat\calD,w_{\gamma,0}^{\widehat\calD})
&\hookrightarrow
W^{k-1,p}(\widehat\calD,w_{\gamma-p,0}^{\widehat\calD}),
\quad &&\gamma>p-1.
\end{align*}
Let $u \in W^{k,p}(\widehat\calD,w_{\gamma,0}^{\widehat\calD})$ and $|\alpha| \leq k-1$. Then we have 
\begin{equation*}
D^\alpha u \in W^{1,p}(\widehat{\calD},w_{\gamma,0}^{\widehat{\calD}})
\hookrightarrow
W^{1,p}(\calO,w_\gamma^{\calO};L^p(\R))
\end{equation*}
 and hence
$D^\alpha u \in L^p(\widehat\calD,w_{\gamma-p,0}^{\widehat\calD})$ by  \cite[Section 8.8]{Kufner1985} and a standard Fubini argument. Therefore, $u \in W^{k-1,p}(\widehat\calD,w_{\gamma-p,0}^{\widehat\calD})$. The proof of \eqref{eq:Hardy:angle:arbitrary:Dir} is similar, using that $C_c^\infty(\widehat\calD)$ is dense in $W_\Dir^{k,p}(\widehat \calD,w_{\gamma,\nu}^{\widehat\calD})$ for $\gamma>(k-1)p-1$.
\end{proof}

\subsection{Krylov-Kondratiev type weighted Sobolev spaces}\label{sec:Krylov:Kondratiev}

The weighted Sobolev spaces $\Dot{W}^{k,p}(\calD,w_{\gamma,\nu}^\calD)$ as well as $\Dot{W}^{k,p}_\Dir(\calD,w_{\gamma,\nu}^\calD)$ introduced in Section~\ref{sec:Kondratiev} remind of Kondratiev spaces, as the powers on the distance to the boundary singularity changes with the order of the derivatives. 
However, the power of the distance to the boundary was chosen to be constant.
In this section we discuss weighted Sobolev spaces with mixed power weights based on both the distance to the boundary singularity and the distance to the boundary, where both powers change with the order of the derivatives. 
Such spaces have been recently used to study stochastic PDEs and related degenerate PDEs on smooth cones, see, e.g.,~\cite{cioicalicht2024sobolevspacesmixedweights,CioKimLee2019, KimLeeSeo2022b, KimLeeSeo2021,Cio20}.
Since both powers vary with the order of the derivative, they remind of both the classical Kondratiev spaces and the spaces used by Krylov and collaborators for the study of stochastic PDEs and associated degenerate PDEs on (smooth) domains, see, e.g.,~\cite{Krylov1999c,Lot2000,Kim2004} for the latter.

We follow~\cite{cioicalicht2024sobolevspacesmixedweights} and write $H^s_{p,\Theta,\theta}(\calD)$ for the following Krylov-Kondratiev type weighted Sobolev spaces: 
Let $1<p<\infty$ and let $\Theta,\theta\in\R$. For $k\in\N_0$ we write
\[
H^k_{p,\Theta,\theta}(\calD)
:=
\big\{
u\in\scrD'(\calD)\colon D^\alpha u \in L^p(\calD,w^\calD_{\lvert\alpha\rvert p+\Theta-d,\lvert\alpha\rvert p+\theta-d}),\,\alpha \in\N_0^d,\lvert\alpha\rvert\leq k
\big\},
\]
endowed with the norm
\[
\nrm{u}_{H^k_{p,\Theta,\theta}(\calD)}
:=
\sum_{\lvert\alpha\rvert\leq k}
\nrm{D^\alpha u}_{L^p(\calD,w^\calD_{\lvert\alpha\rvert p+\Theta-d,\lvert\alpha\rvert p+\theta-d})},
\quad 
u\in H^k_{p,\Theta,\theta}(\calD).
\]
For $s\in (0,\infty)\setminus \N$ with $s\in (k,k+1)$, $k\in\N$, we define $H^s_{p,\Theta,\theta}(\calD)$ as a complex interpolation space, i.e., 
\[
H^s_{p,\Theta,\theta}(\calD)
:=
\big[H^k_{p,\Theta,\theta}(\calD),H^{k+1}_{p,\Theta,\theta}(\calD) \big]_{s-k}.
\]
We refer to~\cite[Section~3]{cioicalicht2024sobolevspacesmixedweights} for alternative definitions of the spaces $H^s_{p,\Theta,\theta}(\calD)$ and their fundamental properties. 
Note that the results in~\cite[Sections~3.1-3.7]{cioicalicht2024sobolevspacesmixedweights} extend mutatis mutandis from angular domains $\calD=\calD_\kappa$ as considered therein to the cones $\calD=\calD_\Omega$ considered here with $\Omega$ admitting a $C^1$-boundary. In particular, $C_c^\infty(\calD)$ is dense in $H^s_{p,\Theta,\theta}(\calD)$ for all $s\geq 0$ and all $\Theta,\theta\in\R$, see~\cite[Theorem~3.8]{cioicalicht2024sobolevspacesmixedweights}. The following lemma relates the Kondratiev type and the Krylov-Kondratiev type spaces. It mainly relies on Hardy's inequality and its consequences from Lemma \ref{lem:Hardy:double}. It will be used, in particular, in order to establish a result on the regularity of the Poisson equation on $\calD$, see Theorem \ref{thm:Poisson:Krylov}.

\begin{lemma}\label{lem:Krylov:Kondratiev:eq2}
Let  $1<p<\infty$, let $k\in\N$, let $\nu\in\R$, let $(k-1)p-1<\gamma <kp-1$, and assume that $\calO\in C^1$.
Then 
\begin{equation}
    \Dot W_\Dir^{k,p}(\calD,w_{\gamma,\nu}^\calD)= H^k_{p,\gamma+d-kp,\nu+d-kp}(\calD) \quad \text{with equivalent norms.}
\end{equation}
\end{lemma}
\begin{proof}

Let $(k-1)p-1<\gamma<kp-1$. Since as mentioned above and as proven in Lemma~\ref{lem:smooth:dense}, $C_c^\infty(\calD)$ is dense in both $H^k_{p,\gamma+d-kp,\nu+d-kp}(\calD)$ and $\Dot{W}_\Dir^{k,p}(\calD,w_{\gamma,\nu}^\calD)$, it suffices to prove 
\[
\sum_{\lvert\alpha\rvert\leq k} \nrm{D^\alpha u}_{L^p(\calD,w_{\gamma,\nu +(\lvert\alpha\rvert-k)p}^\calD)} 
\eqsim 
\sum_{\lvert\alpha\rvert\leq k} \nrm{D^\alpha u}_{L^p(\calD,w_{\gamma +(\lvert\alpha\rvert-k)p,\nu+(\lvert\alpha\rvert-k)p}^\calD)},\quad u\in C_c^\infty(\calD).
\]
On the one hand, `$\lesssim$' holds since $\rho_\calD\leq \rho_\circ$, so that $L^p(\calD,w_{\tilde{\gamma}_0,\tilde{\nu}}^\calD)\hookrightarrow L^p(\calD,w_{\tilde{\gamma}_1,\tilde{\nu}}^\calD)$ for all $\tilde{\gamma}_0,\tilde{\gamma}_1,\tilde\nu\in\R$ with $\tilde\gamma_0\leq\tilde\gamma_1$. 
On the other hand, for `$\gtrsim$' it follows from Lemma~\ref{lem:Hardy:double}
that
\[
\Dot{W}^{k,p}_\Dir(\calD,w_{\gamma,\nu}^\calD)
\hookrightarrow \Dot W_\Dir^{k-1,p}(\calD,w_{\gamma-p,\nu-p}^\calD)\hookrightarrow \dots
\hookrightarrow
L^p(\calD,w_{\gamma-kp,\nu-kp}^\calD).
\]
\end{proof}

\begin{remark}\label{rem:Hardy:KK}
For $p-1<\gamma<2p-1$,  Lemma~\ref{lem:Krylov:Kondratiev:eq2} shows, on the one hand, that every $u\in L^1_{\mathrm{loc}}(\calD)$ with
\begin{equation}\label{eq:Kon:KryKon:2a}
\nrm{u}_{L^p(\calD,w_{\gamma-2p,\nu-2p}^\calD)}
+
\nrm{Du}_{L^p(\calD,w_{\gamma-p,\nu-p}^\calD)}
<\infty,
\end{equation}
has trace zero (in the sense introduced in Section~\ref{sec:Kondratiev}).
On the other hand, it provides a refined description of the behaviour near the boundary of those functions in $u\in \Dot{W}^{2,p}(\calD,w_{\gamma,\nu}^\calD)$ that have trace zero, i.e., such $u$ have  to satisfy~\eqref{eq:Kon:KryKon:2a}. 
Moreover, Lemma~\ref{lem:Krylov:Kondratiev:eq2} yields alternative descriptions of $\Dot{W}^{2,p}_\Dir(\calD,w_{\gamma,\nu}^\calD)$, i.e., it shows that for $p-1<\gamma<2p-1$,
\begin{align*}
\Dot{W}^{2,p}_\Dir(\calD,w_{\gamma,\nu}^\calD)
&=
\big\{u\in H^1_{p,\gamma+d-2p,\nu+d-2p}(\calD) \colon D^2u\in L^p(\calD,w_{\gamma,\nu}^\calD)\big\}\\
&=
\big\{u\in \Dot{W}^{1,p}_\Dir(\calD,w_{\gamma-p,\nu-p}^\calD) \colon D^2u\in L^p(\calD,w_{\gamma,\nu}^\calD)\big\}\\
&=\{u\in \Dot W_\Dir^{1,p}(\calD,w_{\tilde\gamma,\nu-p}^\calD)\colon D^2 u \in L^p(\calD,w_{\tilde\gamma+p,\nu}^\calD)\}
\end{align*}
with $-1<\tilde\gamma:=\gamma-p<p-1.$ 
On the other hand, for $-1<\gamma<p-1$,
\begin{align*}
\Dot{W}^{2,p}_\Dir(\calD,w_{\gamma,\nu}^\calD)
&=
\big\{u\in H^1_{p,\gamma+d-p,\nu+d-2p}(\calD) \colon D^2u\in L^p(\calD,w_{\gamma,\nu}^\calD)\big\}\\
&=
\big\{u\in \Dot{W}^{1,p}_\Dir(\calD,w_{\gamma,\nu-p}^\calD) \colon D^2u\in L^p(\calD,w_{\gamma,\nu}^\calD)\big\}.
\end{align*}
Hence, the difference between the two $\gamma$-ranges is the norm used to control $D^2u$.

\end{remark}

\section{The Dirichlet Laplacian on conical domains}\label{sec:Dirichlet-Laplacian}
In this section we establish the boundedness of the $H^\infty$-calculus of the (suitably defined)  Dirichlet Laplacian on the weighted $L^p$-spaces $L^p(\calD,w_{\gamma,\nu}^\calD)$ for a proper range of parameters $\gamma,\nu\in\R$. 
Recall again that, throughout this manuscript, $\calD=\calD_\Omega$ is a cone of the form~\eqref{eq:D:convention}, where $\Omega\subseteq S^{d-1}$ is a domain, such that $\overline{\Omega}\subsetneq S^{d-1}$. In this section we assume, in addition, that $\Omega$ admits a $C^2$-boundary and denote the  Dirichlet Laplace--Beltrami operator on $\Omega$ by $\Delta_\Dir^\Omega$. Moreover, we let $0<\lambda_1<\lambda_2\le \ldots$ be the ordered sequence of eigenvalues of $-\Delta_\Dir^{\Omega}$ granted by Theorem \ref{Dirichlet:Laplace:Beltrami:functionalcalculus} and set   
\[\lambda^*_n:=-\tfrac{d-2}{2}+\sqrt{\lambda_n+\tfrac{(d-2)^2}{4}}, \quad n\in\N.\]
As before, we borough the notation regarding the domain $\calD$, its parametrizations and related weights from Section~\ref{sec:domain:D}.
Our goal is the proof of the following main result, see also Section~\ref{sec:wedge} for a generalisation to wedges with arbitrary codimension.
\begin{theorem}\label{thm:Laplace:inhom}
Let  $1<p<\infty$ and let $\gamma,\nu\in\R$ be such that
\begin{equation}\label{eq:range:gamma-nu}\begin{aligned}
     \gamma&\in (-1,2p-1)\setminus\{p-1\},\\ \nu&\in\R\setminus\Big\{\Big(\tfrac{2+d}{2}\pm\sqrt{\lambda_n+\tfrac{(d-2)^2}{4}}\Big)p-d\,\colon n\in\N  \Big\}.
\end{aligned}
\end{equation}
Define the Dirichlet Laplacian $\Delta_\Dir^{p,\gamma,\nu}$ on $L^p(\calD,w_{\gamma,\nu}^\calD)$ as 
\[\Delta_\Dir^{p,\gamma,\nu}u:=\Delta u \quad \text{with}\quad \sfD(\Delta_\Dir^{p,\gamma,\nu}):=\Dot W^{2,p}_\Dir(\calD,w_{\gamma,\nu}^\calD)\cap L^p(\calD,w_{\gamma,\nu}^\calD).   \]
Then the following assertions hold.
\begin{enumerate}[\textup{(\roman*)}]
\item\label{it:Laplace:inhom:closed} $(\Delta_\Dir^{p,\gamma,\nu},\sfD(\Delta_\Dir^{p,\gamma,\nu}))$ is closed.
\item\label{it:Laplace:inhom:equivnrms} 
For $u\in \sfD(\Delta_\Dir^{p,\gamma,\nu})$ we have 
\[
\nrm{u}_{\Dot{W}^{2,p}(\calD,w_{\gamma,\nu}^\calD)}
+
\nrm{u}_{L^p(\calD,w_{\gamma,\nu}^\calD)}
\eqsim
\nrm{\Delta u}_{L^p(\calD,w_{\gamma,\nu}^\calD)} + \nrm{u}_{L^p(\calD,w_{\gamma,\nu}^\calD)}.
\]
\item\label{it:Laplace:hom:domain}  $\Dot{\sfD}(\Delta_\Dir^{p,\gamma,\nu}) := (\sfD(\Delta_\Dir^{p,\gamma,\nu}),\nrm{\Delta\,\cdot\,}_{L^p(\calD,w_{\gamma,\nu}^\calD)})^\sim=\Dot{W}^{2,p}_\Dir(\calD,w_{\gamma,\nu}^\calD)$. 
\item\label{it:Laplace:inhom:calculus}  If, in addition,
\begin{equation}\label{eq:range:calc:nu}
-\lambda_1^*p-d<\nu<\big(d+\lambda_1^*\big)p-d,    
\end{equation}
then $-\Delta_\Dir^{p,\gamma,\nu}$ is a sectorial operator with a bounded $H^\infty$-calculus of angle zero.
 In particular, $\Delta_\Dir^{p,\gamma,\nu}$ generates an analytic semigroup $(S(t))_{t\geq 0}$ on $L^p(\calD,w_{\gamma,\nu}^\calD)$, which is consistent over $p,\gamma,\nu$.
\end{enumerate}
\end{theorem}
Our strategy for proving Theorem~\ref{thm:Laplace:inhom} is the following: In Section~\ref{sec:FunCalc:L2} we establish the boundedness of the $H^\infty$-calculus in $L^2(\calD)$ via form methods. In Section~\ref{sec:FunCalc:extrapolate} we extrapolate this functional calculus to suitable weighted $L^p$-spaces by means of Theorem~\ref{functionalcalculus:extrapolation:BFP}.
In Section~\ref{sec:Laplace:hom} we will prove that the homogeneous  Dirichlet Laplacian, i.e., the Laplacian as an operator from $\Dot{W}^{2,p}_\Dir(\calD,w_{\gamma,\nu}^\calD)$ to $L^p(\calD,w_{\gamma,\nu}^\calD)$, is an isomorphism---provided~\eqref{eq:range:gamma-nu} holds. 
The latter proves the first three parts of Theorem~\ref{thm:Laplace:inhom} by means of Lemma~\ref{lem:isom:unbdd}. 
Finally, in Section~\ref{sec:main:proof}, in order to establish the boundedness of the $H^\infty$-calculus for the negative Dirichlet Laplacian introduced in Theorem~\ref{thm:Laplace:inhom}, we show that if,  in addition,~\eqref{eq:range:calc:nu} holds, then it coincides with the extrapolated operator from Section~\ref{sec:FunCalc:extrapolate}.

Before we present the proof in detail, we make some remarks on Theorem~\ref{thm:Laplace:inhom}, its relevance and consequences. 
As outlined in the introduction, boundedness of the $H^\infty$-calculus with an angle less than $\pi/2$ implies maximal regularity. In the following remark we relate the results on maximal regularity that may be obtained from Theorem~\ref{thm:Laplace:inhom} to the existing literature. 

\begin{remark}\label{rem:main:thm:maxReg}~
\begin{enumerate}[\textup{(\roman*)}]
\item\label{it:rem:main:maxreg}
The regularity of elliptic and parabolic PDEs on conical domains has been studied extensively within the framework of Sobolev spaces with power weights based on the distance to the boundary singularities. We refer to~\cite{Dau1988,KozMazRos1997,Sol2001,Nazarov2001,BorKon2006,MazRos2010,Kozlov2014TheDP} and the literature therein, just to name a few.
Recently, maximal regularity has also been established for SPDEs and related PDEs with degenerating forcing at the boundary within the framework of the weighted Sobolev spaces $H^s_{p,\Theta,\theta}(\calD)$ of Krylov-Kondratiev type~\cite{CioKimKyeLeeLin2018,CioKimLee2019,Cio20,KimLeeSeo2021,KimLeeSeo2022b,cioicalicht2024sobolevspacesmixedweights}. 
However, these results have been obtained predominantly by means of 
(S)PDE techniques rather than by exploring the properties of the operators in the equations.
Exceptions may be found in~\cite{Prss2007HinftycalculusFT,KoehneSaalWestermann2021,koehne2024optimalregularitystokesequations}, where operator sum techniques are employed to establish maximal regularity, however, without establishing boundedness of the $H^\infty$-calculus.

Theorem~\ref{thm:Laplace:inhom} reveals the operator-theoretic structure behind the results mentioned above. 
The admissible range therein is large enough to show that, at least for the (stochastic) heat equation with Dirichlet boundary condition in $C^2$-cones, the results obtained so far fall into the framework of~\cite{NeeVerWei2012}, i.e., in these situations, from an operator-theoretic perspective, maximal regularity follows from the boundedness of the $H^\infty$-calculus of the leading operator.
For instance, the corresponding results in~\cite{KimLeeSeo2021,KimLeeSeo2022b} on the maximal regularity of the stochastic heat equation in weighted Sobolev spaces with mixed weights follow from Theorem~\ref{thm:Laplace:inhom} with 
\begin{equation}\label{eq:range:gamma-nu:KimLeeSeo}
p-1<\gamma<2p-1
\quad\text{and}\quad
(2-\lambda^*_1)p-d <\nu<(d+\lambda_1^*)p-d
\end{equation}
by means of~\cite{NeeVerWei2012},
noticing that the restriction in \eqref{eq:range:gamma-nu} amounts to
\[\Big(\tfrac{d+2}{2}-\sqrt{\lambda_n+\tfrac{(d-2)^2}{4}}\Big)p-d = (2-\lambda^*_n)p-d,\]
and 
\[\Big(\tfrac{d+2}{2}+\sqrt{\lambda_n+\tfrac{(d-2)^2}{4}}\Big)p-d = (d+\lambda^*_n)p-d.\]
Also, our range is large enough to cover the related results from
\cite{Nazarov2001,Prss2007HinftycalculusFT,CioKimKyeLeeLin2018,CioKimLee2019,Cio20,LoristVeraar2021}, where the (stochastic) heat equation is studied in weighted $L^p$-Sobolev spaces with weights based solely on the distance to the tip, always under the restriction that
\begin{equation}\label{eq:range:nu:others}
(2-\lambda^*_1)p-d <\nu<(d+\lambda_1^*)p-d.
\end{equation}
Note that~\eqref{eq:range:calc:nu} is less restrictive and closed under $L^p$-duality.
It is also worth noting that, in contrast to the admissible range~\eqref{eq:range:nu:others} obtained in earlier results, the  range for the parameter $\nu$ in Theorem~\ref{thm:Laplace:inhom}, for which boundedness of the $H^\infty$-calculus is established, is not limited to the range between two resonant frequencies. Instead, we obtain the natural range, for which extrapolation is possible, and exclude all resonant frequencies from it.

In addition, Theorem~\ref{thm:Laplace:inhom} allows for several extensions of the results mentioned above. In particular, as already mentioned in the introduction, other than in~\cite{KimLeeSeo2021,KimLeeSeo2022b}, we are able to allow for differing  integrability parameters with respect to time and space, which, at least at this stage, is not possible within the other approaches used so far.
Moreover, since the admissible range in Theorem~\ref{thm:Laplace:inhom} is considerably larger than~\eqref{eq:range:gamma-nu:KimLeeSeo}, we have more freedom with the parameters, which is particularly important when moving towards nonlinear (S)PDEs~\cite{AgrestiVeraar2022a,AgrestiVeraar2022b, AgrVer2025}.
Finally, it is worth noting that our result can also be used to obtain maximal regularity results for (stochastic) Volterra equations on angular domains by means of~\cite{DesLon2013}. To the best of our knowledge, these would be the first results for such stochastic evolution equations on conical domains.

\item\label{it:rem:main:MaierSaal} 
In addition to the references mentioned in the introduction, boundedness of the $H^\infty$-calculus for a variant of the Neumann Laplacian on wedge domains is claimed in~\cite[Proposition~4]{MaierSaal2014}. The proof technique employed therein would also extend to the Dirichlet case considered here.
However, the proof is unfortunately not correct due to an algebraic mistake in the calculations.
In our notation, the authors write $-\Delta_\Dir = T_{a+2}^{-1} (A+B+L_0) T_a,$ where $A+B+L_0$ is an operator sum with a bounded $H^\infty$-calculus. They then conclude 
$R(\lambda, -\Delta_\Dir)= T_a^{-1} R(\lambda,A+B+L_0)T_{a+2}$, which is unfortunately only valid for $\lambda=0$.

\item Note that, since $0\notin\rho(-\Delta_\Dir^{p,\gamma,\nu})$, the characterisation of the homogeneous domain provided in part~\ref{it:Laplace:hom:domain} of Theorem~\ref{thm:Laplace:inhom} is an important building block in the derivation of (stochastic) maximal regularity via~\cite{NeeVerWei2012}.
\item\label{it:rem:main:dim:2} 
If $d=2$ and $S^1$ is parametrised as in Remark \ref{rem:dim:two:cone}, then the Laplace--Beltrami operator on $\Omega$ is just the Laplacian on the interval $(0,\kappa)$.
In this case, the spectrum is explicitly given by 
\[\sigma(-\Delta_\Dir^{(0,\kappa)})=\Big\{\frac{\pi^2}{\kappa^2}n^2\colon n\in \N\Big\}.\]
\end{enumerate}
\end{remark}

The following remark addresses the restrictions we impose on the weight parameters. 

\begin{remark}\label{rem:main:thm}~
\begin{enumerate}[\textup{(\roman*)}]
\item\label{it:rem:main:range} The constraints on the weight parameters $\gamma$ and $\nu$ in~\eqref{eq:range:gamma-nu} stem from our proof that the Laplacian is an isomorphism from $\Dot{W}^{2,p}_\Dir(\calD,w_{\gamma,\nu}^\calD)$ to $L^p(\calD,w_{\gamma,\nu}^\calD)$ given in Section~\ref{sec:Laplace:hom}.
We prove the latter by first transforming the operator into cylindrical-type  coordinates and then employing operator sum techniques inspired by~\cite{Prss2007HinftycalculusFT}.
This reduces the analysis of the Dirichlet Laplacian on $\calD$ to properties of the Dirichlet Laplace--Beltrami operator $\Delta_\Dir^\Omega$ on the intersection of the unit sphere $S^{d-1}$ and $\calD$.
In particular, we use that $-\Delta_\Dir^{\Omega}$ has a bounded $H^\infty$-calculus of angle zero in $L^p(\Omega,w_\gamma^{\Omega})$, which we get from~\cite{LindemulderVeraar2020} exactly for the $\gamma$-range specified in~\eqref{eq:range:gamma-nu}, see Theorem \ref{Dirichlet:Laplace:Beltrami:functionalcalculus}.
Moreover, when transforming into cylindrical coordinates to get rid of the radial weight, the power on $\rho_\circ$ translates into a shift that depends on~$\nu$.
This forces us to exclude certain values that depend on the spectrum of $-\Delta_\Dir^{\Omega}$ from the admissible range of $\nu$, yielding the restriction on $\nu$ in~\eqref{eq:range:gamma-nu}.
 In addition, an adaptation of the ideas in \cite{KoehneSaalWestermann2021} ensures that the exclusion of 
\[
\nu\in \Big\{\Big(\tfrac{2+d}{2}\pm \sqrt{\lambda_n+\tfrac{(d-2)^2}{4}}\Big)p-d\colon n\in\N\Big\}
\]
is a necessary consequence of the choice of our (homogeneous) domain, see Proposition~\ref{prop:counterexample} below.
The additional restriction~\eqref{eq:range:calc:nu} on $\nu$  comes from the application of the extrapolation result in Theorem~\ref{functionalcalculus:extrapolation:BFP} to our situation in   Theorem~\ref{functional:calculus:laplace:wedge:Muckenhoupt}.

\item\label{it:rem:main:sharp} 
As pointed out in Remark \ref{rem:main:thm:maxReg}, the range of admissible parameters in Theorem~\ref{thm:Laplace:inhom} is large enough to cover the maximal ranges considered in the literature on (S)PDEs on conical domains and even goes beyond that. However, due to our particular choice of the domains of the Dirichlet Laplacian, which are tailor-made for capturing the behaviour of singularities at the tip of the domain, we do not cover some well-studied situations, for which boundedness of the $H^\infty$-calculus of the negative Dirichlet Laplacian has been already established in an unweighted setting. 
For instance, the case $\calD=\R_+^2$, $p=2$, and $\gamma=\nu=0$, i.e., the case where the base space $L^p(\calD,w_{\gamma,\nu}^\calD)$ coincides with the unweighted space $L^2(\calD)$, is excluded in Theorem~\ref{thm:Laplace:inhom}.
However, it is well-known that if the domain of the negative Dirichlet Laplacian in $L^2(\calD)$ is set to 
$\{ u\in W_0^{1,2}(\calD)\colon \Delta u\in L^2(\calD) \}$ (instead of our choice $\sfD(\Delta_\Dir^{2,0,0})$ in Theorem~\ref{thm:Laplace:inhom}), then this operator does have a bounded $H^\infty$-calculus of angle zero, see also Proposition~\ref{functionalcalculus:variationalsetting} below.
The results in~\cite{Wood_2007} indicate that the latter might even hold for certain $p\in (2-\delta',4+\delta)$, as for this range maximal regularity can be established in the usual unweighted $L^p$-Sobolev spaces.
\end{enumerate}
\end{remark}

\subsection{Boundedness of the \texorpdfstring{$H^\infty$}{H}-calculus in unweighted \texorpdfstring{$L^2$}{L}-spaces}\label{sec:FunCalc:L2}
    We first consider the negative Dirichlet Laplacian in the unweighted space $L^2(\calD)$. In this case we obtain the desired boundedness of the $H^\infty$-calculus easily by employing form methods.
    Our standard reference concerning forms and associated operators is~\cite{isem27_HarmonicAnalysis}. 

    Let $W_0^{1,2}(\calD)$ be the closure of $C_c^\infty(\calD)$ in $W^{1,2}(\calD)$.
    Consider the sesquilinear form     
    \[\fraka\colon W_0^{1,2}(\calD)\times W_0^{1,2}(\calD)\to\C,\quad (u,v)\mapsto \fraka(u,v):= \int_\calD \nabla u \cdot \overline{\nabla v}\,\dd x, \]
    which is obviously bounded. The linear operator $(L,\sfD(L))$ associated with $\fraka$ in $L^2(\calD)$ is called the \emph{negative Dirichlet Laplacian in $L^2(\calD)$} and is given by

    \[ \sfD(L)=\{ u\in W_0^{1,2}(\calD)\colon \Delta u\in L^2(\calD) \},\quad L u = -\Delta u,  \]
    The following proposition collects some well-known properties of the negative Dirichlet Laplacian.
    Note that the statement also holds if we replace $\calD$ by any open subset of $\R^d$. 
    \begin{proposition}\label{functionalcalculus:variationalsetting}
      Let $(L,\sfD(L))$ be the negative Dirichlet Laplacian in $L^2(\calD)$ introduced above. Then the following holds. 
    \begin{enumerate}[label=\textup{(\roman*)}]
        \item\label{it:L:calculus} $L$ has a bounded $H^\infty$-calculus of angle $0$.
        \item\label{it:L:semigroup} The Dirichlet Laplacian $-L$ generates a $C_0$-semigroup $(e^{-tL})_{t\geq 0}$ of contractions, which extends to a bounded analytic semigroup of any angle $0<\vartheta <\frac{\pi}{2}$.
        \item\label{it:L:kernel} For all $t>0$, $e^{-tL}\in\sL(L^2(\calD))$ possesses a nonnegative kernel 
        $G_t^\calD\colon\calD\times\calD\to\R$, i.e.,
        \[
        e^{-tL}f = \int_\calD G_t^\calD(\cdot,y)f(y)\,\dd y, \quad f\in L^2(\calD). 
        \]
    \end{enumerate}
    \end{proposition}
    \begin{proof}
    Since the form $\fraka$ is symmetric, accretive, and sectorial of angle $0$, the operator $(L,\sfD(L))$ is densely defined, self-adjoint, and $\sigma(L)\subseteq[0,\infty)$. Thus,~\ref{it:L:calculus} is a straightforward application of \cite[Proposition~10.2.23]{aibs2}.
    To see~\ref{it:L:semigroup}, note that~\ref{it:L:calculus} implies that $-L$ generates a bounded analytic semigroup  of any angle $0<\vartheta<\frac{\pi}{2}$, see, e.g., \cite[Theorem~G.5.2]{aibs2}. The contractivity of the semigroup on $[0,\infty)$ follows, e.g., from the Lumer-Phillips theorem.  For a detailed proof of~\ref{it:L:kernel} we refer to~\cite[Theorem~4.2.4]{Are06}. 
    \end{proof}

\subsection{Extrapolation of the functional calculus}\label{sec:FunCalc:extrapolate}
For the extrapolation of the functional calculus, we will use the following  heat kernel estimate, which may be found in~\cite[Theorem~2.1]{KimLeeSeoRefinedGreenCone} (see also~\cite{Carslaw1952ConductionOH,BanSmi1997}). 
In the sequel, we write
\[
    R_{t,x}:=\frac{|x|}{|x|+\sqrt{t}} \quad \text{and}\quad J_{t,x}:=\frac{\rho_\calD(x)}{\rho_\calD(x)+\sqrt{t}}, \qquad t>0,\,\, x\in\calD.
\] 
    \begin{proposition}
     Let $(G_t^\calD)_{t>0}$ be as in Proposition~\ref{functionalcalculus:variationalsetting}\ref{it:L:kernel}. 
     Then there exists a constant $c>0$, such that for all $ 0<\lambda<\lambda_1^*$,     \begin{equation}\label{Greenfunction:cone:estimate:Refined}
        G_t^\calD(x,y) \lesssim_{\lambda} R_{t,x}^{\lambda-1}R_{t,y}^{\lambda-1}J_{t,x}J_{t,y} \frac{e^{-c\frac{|x-y|^2}{t}}}{t^{d/2}}, \quad t>0,\,\,x,y\in\calD.
    \end{equation}    
 \end{proposition}

In order to  extrapolate the $H^\infty$-calculus from $L^2(\calD)$ to appropriate weighted $L^p$-spaces by means of Theorem~\ref{functionalcalculus:extrapolation:BFP}, we need to translate \eqref{Greenfunction:cone:estimate:Refined} and link the objects $R_{t,x}$ and $J_{t,x}$ therein to the volume of a ball with respect to appropriate doubling measures. This will be done by means of the following lemma. 
We will extensively use the notation introduced in Section~\ref{sec:domain:D}.
In particular, 
\[
B(x,r):=B^\calD(x,r),\quad x\in\calD,\quad r>0.
\]

\begin{lemma}\label{mixedweight:estimate:ball}
Let $-d<\nu<\infty$ and let $-1<\gamma<\infty$. Then 
\begin{equation}
    w_{\gamma,\nu}^\calD(B(x,r))
    \eqsim_{\nu,\gamma} 
    r^d(\rho_\circ(x)+r)^{\nu-\gamma}(\rho_{\calD}(x)+r)^\gamma, \quad x\in \calD,\,r>0.
\end{equation}
\end{lemma}
\begin{proof} 
The fact that $\Omega$ admits a $C^2$-boundary and $\overline{\Omega}\subsetneq S^{d-1}$ ensures that there exist $K,r_0 >0$ such that for all $z\in \partial\Omega$ there exists an injective $\psi=\psi_z$ mapping $B^{S^{d-1}}(z,r_0)$ onto a set $D^z\subseteq \R^{d-1}$, such that  $ \psi(z)=0$,
\begin{align*}
    \psi(B^{S^{d-1}}(z,r_0)\cap\Omega) &\subseteq \R^{d-1}_+,\\
    \psi(B^{S^{d-1}}(z,r_0)\cap\partial \Omega)&= D^z\cap \partial\R^{d-1}_+,
\end{align*}
and $\psi\in C^2(\overline{B^{S^{d-1}}(z,r_0)})$, $\psi^{-1}\in C^2(\overline{D^z})$
 with $\nrm{\psi}_{C^2},\nrm{\psi^{-1}}_{C^2}\le K$.  
 Without loss of generality assume that $r_0<2.$
We consider three different cases. 

\medskip

\noindent \emph{Case 1. Balls close to the origin.} Assume that $x\in \calD$ and $r>0$ are such that $\frac{4}{r_0} r\geq \rho_\circ(x)$. 
In this case,
\[
\rho_\calD(x)+r \eqsim r \eqsim \rho_\circ(x)+r.
\]
This and the volume doubling property of $w_{\gamma,\nu}^\calD$
in the form of \eqref{eq:doubling:cons2}
yield that
\[
	w_{\gamma,\nu}^\calD(B(x,r))
    \eqsim_{\gamma,\nu} 
    w_{\gamma,\nu}^\calD(B(0,r))
    \eqsim_{\gamma,\nu} 
    r^{d+\nu}
    \eqsim_{\gamma,\nu} 
    r^d(\rho_\circ(x)+r)^{\nu-\gamma}(\rho_\calD(x)+r)^\gamma.
\]

\medskip

\noindent \emph{Case 2. Interior balls.} Assume that $x\in \calD$ and $r>0$ are such that $\frac{4}{r_0} r<\rho_\circ(x)$ and $2r< \rho_\calD(x)$. 
Then for $z\in B(x,r)$
\begin{equation}\label{eq:mixedweights:equiv1}
\rho_\circ(x)+r \eqsim \rho_\circ(z) \quad\text{and}\quad \rho_\calD(z)\eqsim\rho_\calD(x)+r,
\end{equation}
so $w_{\gamma,\nu}^\calD(z) \eqsim (\rho_\circ(x)+r)^{\nu-\gamma}(\rho_\calD(x)+r)^\gamma$. Hence

\[
w_{\gamma,\nu}^\calD(B(x,r))\eqsim_{\gamma,\nu}(\rho_\circ(x)+r)^{\nu-\gamma}(\rho_\calD(x)+r)^\gamma |B(x,r)| \eqsim r^d(\rho_\circ(x)+r)^{\nu-\gamma}(\rho_\calD(x)+r)^\gamma.
\]

\medskip

\noindent \emph{Case 3. Balls close to the boundary.} Finally, assume that $x\in\calD$ and $r>0$ are such that $\frac{4}{r_0} r<\rho_\circ(x)$ and $2r\geq \rho_\calD(x)$. 
By Lemma~\ref{lem:doubling} and its proof,
\begin{equation}\label{eq:sim:mu:mu:D}
w_{\gamma,\nu}^\calD(B(x,r))
\eqsim_{\gamma,\nu}
 w_{\gamma,\nu}^\calD(S(x,r))
\eqsim_\gamma
w_{\nu+d-1}^{\R_+}(B^{\R_+}(\lvert x\rvert,r)) \cdot     w_\gamma^{\Omega}\Big(B^{\Omega}\big(\tfrac{x}{|x|},\tfrac{r}{\lvert x\rvert}\big)\Big).
\end{equation}
Since $2r\leq \frac{4}{r_0} r<\rho_\circ(x)$, we have $s\eqsim |x|$ for all $s\in B^{\R_+}(\lvert x\rvert,r)$ and therefore

\begin{equation}\label{eq:sim:mu:nu:rplus}
w_{\nu+d-1}^{\R_+}(B^{\R_+}(\lvert x\rvert,r))
\eqsim_\nu
\lvert x\rvert^{\nu+d-1} r.
\end{equation}
Let $z\in \partial \Omega$ be such that $\mathrm{dist}_{S^{d-1}}(\frac{x}{|x|}, z)=\mathrm{dist}_{S^{d-1}}(\frac{x}{|x|}, \partial\Omega)$. 
Since $\Omega$ has a $C^2$-boundary, there exists $\psi$ as above such that $\psi(z)=0$ and $\nrm{\psi}_{C^2},$ $ \nrm{\psi^{-1}}_{C^2}\le K$. In particular, both $\psi$ and $\psi^{-1}$ are Lipschitz continuous and therefore \cite[Lemma 3.24]{Grigoryan2009_Book} ensures 
\[|\psi(z')-\psi(z'')| \eqsim \mathrm{dist}_{S^{d-1}}(z',z'')\eqsim |z'-z''|, \quad z',z'' \in B^{S^{d-1}}(z,r_0).\]
As a consequence,
there exist $0<c\le C<\infty$ such that 
\[\R^{d-1}_+\cap B\big(0,\tfrac{cr}{|x|}\big) \subseteq \psi\Big(\Omega\cap B^{S^{d-1}}\big(z,\tfrac{r}{|x|}\big)\Big)\subseteq\R^{d-1}_+\cap B\big(0, \tfrac{Cr}{|x|}\big).\]
This, the assumption that $ r/\lvert x\rvert <r_0/4$, and the volume doubling property of $w_\gamma^{\Omega}$ yields
\begin{equation}\label{eq:sim:mu:nu:kappa}
w_\gamma^{\Omega}\Big(B^{\Omega}\big(\tfrac{x}{|x|},\tfrac{r}{\lvert x\rvert}\big)\Big)
\eqsim_\gamma
w_\gamma^{\Omega}\Big(B^{\Omega}\big(z,\tfrac{r}{\lvert x\rvert}\big)\Big)
\eqsim_{\psi} w_\gamma^{\R^{d-1}_+}\Big(B^{\R^{d-1}_+}\big(0,\tfrac{r}{|x|}\big)\Big)
\eqsim_\gamma
\Big(\frac{r}{\lvert x\rvert}\Big)^{\gamma+d-1}.
\end{equation}
Thus, inserting~\eqref{eq:sim:mu:nu:rplus} and~\eqref{eq:sim:mu:nu:kappa} in~\eqref{eq:sim:mu:mu:D} leads to
\[
w_{\gamma,\nu}^\calD(B(x,r))
\eqsim_{\gamma,\nu}
\Big(\frac{r}{\lvert x\rvert}\Big)^{\gamma+d-1}
\lvert x\rvert^{\nu+d-1} r
=
r^d\, \lvert x\rvert^{\nu-\gamma}\, r^\gamma
\eqsim_{\gamma,\nu}
r^d\, (\rho_\circ(x)+r)^{\nu-\gamma}\, (\rho_\calD(x)+r)^\gamma,
\]
where in the last step we used $|x|\eqsim\rho_\circ(x)+r$ and $\rho_\calD(x)+r\eqsim r$.
\end{proof}

Now we can extrapolate the $H^\infty$-calculus of the negative Dirichlet Laplacian on $L^2(\calD)$ as follows.

\begin{theorem}\label{functional:calculus:laplace:wedge:Muckenhoupt}
Let $(L,\sfD(L))$ be the negative Dirichlet Laplacian on $L^2(\calD)$. Let $1<p<\infty$, let $\displaystyle 0<\lambda<\lambda_1^*$, and let $0<\sigma<\pi$. Then for $v\in A_p( w_{2,2\lambda}^\calD)$ we have
\begin{equation}\label{eq:extrapol:ApGeneral}
\nrm{f(L)}_{\scrL(L^p(\calD,v w_{2-p,(2-p)\lambda}^\calD))}
\lesssim_{p,\sigma,\lambda} 
[v]_{A_p( w_{2,2\lambda}^\calD)}^{\max\{\frac{1}{p-1},1\}}\nrm{f}_{H^\infty(\Sigma_\sigma)},\qquad f\in H^\infty(\Sigma_\sigma).
\end{equation}
In particular, if $\gamma,\nu \in\R$ satisfy
\begin{equation}\label{eq:extrapol:range}
-1-p < \gamma < 2p-1\quad\text{and}\quad-d-\lambda_1^*p < \nu < (d+\lambda_1^*)p-d,
\end{equation}
then
\begin{equation}\label{eq:extrapol:ApPower}
\nrm{f(L)}_{\sL(L^p(\calD, w_{\gamma,\nu}^\calD))}\lesssim_{ p,\sigma,\gamma,\nu} \nrm{f}_{H^\infty(\Sigma_\sigma)}, \quad f\in H^\infty(\Sigma_\sigma).
\end{equation}
\end{theorem}
\begin{proof}
 
    We already know from Proposition~\ref{functionalcalculus:variationalsetting} that $L$ has a bounded $H^\infty$-calculus of angle $0$ on $L^2(\calD)$. 
    Moreover, $-L$ generates the semigroup $(e^{-tL})_{t\geq0}$, which has kernel representations with kernels $(G_t^\calD)_{t>0}$ that satisfy the estimate~\eqref{Greenfunction:cone:estimate:Refined}. Rearranging this estimate by means of  Lemma~\ref{mixedweight:estimate:ball} yields that there exists a constant $\mathfrak{c} >0$ such that for all $t> 0$, $x,y\in \calD$, 
  \begin{align}\label{eq:Green:rearranged}
 	G^\calD_t(x,y)
    &\lesssim_{\lambda} R_{t,x}^{\lambda-1}R_{t,y}^{\lambda-1} J_{t,x}J_{t,y} \frac{e^{-\mathfrak{c}\frac{|x-y|^2}{t}}}{t^{d/2}}\notag\\
    &= \frac{|x|^\lambda \Big(\frac{\rho_\calD(x)}{\rho_\circ(x)}\Big)  |y|^\lambda \Big(\frac{\rho_\calD(y)}{\rho_\circ(y)}\Big) \,\, e^{-\mathfrak{c}\frac{|x-y|^2}{t}}}{\sqrt{t^{d/2}(|x|+\sqrt{t})^{2\lambda-2}(\rho_\calD(x)+\sqrt{t})^2}\sqrt{t^{d/2}(|y|+\sqrt{t})^{2\lambda-2}(\rho_\calD(y)+\sqrt{t})^2}}\notag\\
    &\eqsim_{\lambda} \frac{w^\calD_{1,\lambda}(x)w^\calD_{1,\lambda}(y)}{\sqrt{ w_{2,2\lambda}^\calD(B(x,\sqrt{t}))}  \sqrt{ w_{2,2\lambda}^\calD(B(y,\sqrt{t}))}} e^{-\mathfrak{c}\frac{|x-y|^2}{t}}.    
 \end{align}
For a function $u\colon \calD\rightarrow\C$ let $M u(x):= w^\calD_{1,\lambda}(x) u(x)$, $x\in\calD$. 
Then
the corresponding operator $M \in \sL(L^2(\calD,  w_{2,2\lambda}^\calD),L^2(\calD))$ is an isometric isomorphism. 
Thus, since $(L,\sfD(L))$ has a bounded $H^\infty$-calculus of angle $0$ on $L^2(\calD)$, so does the operator $(M^{-1}L M,M^{-1}\sfD(L))$ on $L^2(\calD,w_{2,2\lambda}^\calD)$, see Lemma~\ref{TransformUnderIsometricIsomorphism}. Moreover, 
the semigroup generated by $-M^{-1}L M$ on $L^2(\calD,w_{2,2\lambda}^\calD)$ satisfies for all $t>0$,
\[
e^{-tM^{-1}LM}f
=
M^{-1}e^{-tL}Mf
=
\int_\calD \calG_t^\calD(\cdot,y)f(y)\,w_{2,2\lambda}^\calD(y)\,\dd y,\quad f\in L^2(\calD,w_{2,2\lambda}^\calD),
\]
with 
\[
\calG_t^\calD(x,y) =\frac{G_t^\calD(x,y)}{w_{1,\lambda}^{\calD}(x)w_{1,\lambda}^{\calD}(y)}  ,\quad t>0,\, x,y\in\calD.
\]
By \eqref{eq:Green:rearranged} these kernels satisfy
\[
\calG^\calD_t(x,y)
\lesssim_{\lambda} 
\frac{e^{-\mathfrak{c}\frac{|x-y|^2}{t}}}{\sqrt{w_{2,2\lambda}^\calD(B(x,\sqrt{t}))}  \sqrt{w_{2,2\lambda}^\calD(B(y,\sqrt{t}))}},   
\]
so that, by the doubling property of $w^\calD_{2,2\lambda}$ (see~\eqref{eq:doubling:cons2}), for any constant $0<c<\mathfrak{c}$, there exists a finite constant $C>0$ that only depends on $\lambda$ and $c$, such that
\[
|\calG_t^\calD(x,y)|
\le 
C \frac{e^{-c\frac{|x-y|^2}{t}}}{w_{2,2\lambda}^\calD(B(x,\sqrt{t}))},\quad t>0,\,\,x,y\in\calD.
\]
Thus, an application of Theorem~\ref{functionalcalculus:extrapolation:BFP} with $(A,\sfD(A))\curvearrowleft (M^{-1}LM,M^{-1}\sfD(L))$ and $\mu\curvearrowleft w^\calD_{2,2\lambda}$
yields that
\[ 
\nrm{f(M^{-1}LM)}_{\scrL(L^p(\calD,v w_{2,2\lambda}^\calD))}
\lesssim_{p,\sigma,\lambda} 
[v]_{A_p(w_{2,2\lambda}^\calD)}^{\max\{\frac{1}{p-1},1\}}\nrm{f}_\infty,\quad v\in A_p(w_{2,2\lambda}^\calD),\,\, f\in H^\infty(\Sigma_\sigma).
\]
This, together with the identity $f(M^{-1}LM)=M^{-1}f(L)M$ and the fact that $M \in \sL(L^p(\calD, w^\calD_{2,2\lambda}),L^p(\calD,w^\calD_{2-p,(2-p)\lambda}))$ is an isometric isomorphism, proves~\eqref{eq:extrapol:ApGeneral} for all $0<\lambda<\lambda_1^*$, which, in turn, yields~\eqref{eq:extrapol:ApPower} for all $\gamma,\nu\in\R$ satisfying~\eqref{eq:extrapol:range} due to Lemma~\ref{lem:doubling}\ref{it:mixedweights:Ap}. 
\end{proof}

\begin{remark}
    Theorem~\ref{functional:calculus:laplace:wedge:Muckenhoupt} shows that, for the full range~\eqref{eq:extrapol:range},
    the semigroup generated by $(-L,\sfD(L))$ extrapolates to an analytic semigroup on $L^p(\calD,w_{\gamma,\nu}^\calD)$.
    Estimate~\eqref{eq:extrapol:ApPower} yields that the negative of its generator has a bounded $H^\infty$-calculus of angle zero.
    However, at this stage, it is not clear whether this extrapolated operator can indeed be identified as a Laplacian with Dirichlet boundary conditions and whether its domain can be characterized in terms of (weighted) Sobolev spaces, cf.~Remark~\ref{rem:extrapol:ident}.
    In what follows, we establish such a characterisation under the restriction that~\eqref{eq:range:gamma-nu} and~$\eqref{eq:range:calc:nu}$ hold.
    This means that we exclude all
    \[
    \nu
    \in
    \{(2-\lambda_n^*)p-d\,\colon n\in\N  \},
    \]
    which is due to the 
     choice of the domain,
    see Remark~\ref{rem:main:thm}. Moreover, we  exclude
    \[
    -p-1<\gamma< -1, 
    \]
    which seems reasonable, since for this range, the result already fails to hold for the half-space, see~\cite[Proposition~5.16]{LindemulderVeraar2020}.
\end{remark}

\begin{remark}

The proof strategy used in Theorem~\ref{functional:calculus:laplace:wedge:Muckenhoupt} extends beyond the present setting, since refined heat kernel bounds are available for a broad class of differential operators.
For instance, for a wide class of differential operators subject to Dirichlet boundary conditions, the generated semigroup has an integral representation whose
kernels typically satisfy the two-sided estimate
\[
  k_t(x,y) \eqsim J_{t,x} J_{t,y} \, \frac{e^{-c\frac{|x-y|^2}{t}}}{t^{d/2}}, 
\]
where the factors $J_{t,x},J_{t,y}$ encode the parabolic distance to the boundary, similar to Proposition \ref{functionalcalculus:variationalsetting},  cf.~\cite{Davies_1989_HeatKernels_spectral_theory,cho2006}. 
An adaptation of the proof of Lemma~\ref{mixedweight:estimate:ball} (in particular, Step 3 therein)  entails that, e.g., for a bounded Lipschitz domain $G$ it holds 
\[w_\gamma^G(B(x,r)) \eqsim_\gamma r^d(\rho_G(x)+r)^\gamma,\quad x\in G, r>0.\]
Therefore, refined kernel bounds can be combined with the conjugation argument in the proof of Theorem \ref{functional:calculus:laplace:wedge:Muckenhoupt} to obtain weighted functional calculi for a wide range of operators, enlarging the range of admissible weights beyond those
obtainable by extrapolation from the unweighted setting.
\end{remark}

\subsection{The homogeneous Dirichlet Laplacian}\label{sec:Laplace:hom}

So far, we have seen that the negative Dirichlet Laplacian $(L,\sfD(L))$ has a bounded $H^\infty$-calculus of angle zero in $L^2(\calD)$, which may be extrapolated to weighted spaces $L^p(\calD,w)$ for suitable weights $w\in L^1_\textup{loc}(\calD)$ by means of Theorem~\ref{functional:calculus:laplace:wedge:Muckenhoupt}. 
However, at this stage, it is not clear how the extrapolated operator relates to the negative Dirichlet Laplacian from Theorem~\ref{thm:Laplace:inhom}. 
In order to establish this connection, we first consider the \emph{homogeneous} Dirichlet Laplacian and prove the following result.

\begin{proposition}\label{prop:Laplace:hom}
    Let $1<p<\infty$ and let $\gamma,\nu\in\R$ satisfy~\eqref{eq:range:gamma-nu}.
Then the \emph{homogeneous Dirichlet Laplacian}
    \[
    \Dot{\Delta}_\Dir^{p,\gamma,\nu}\colon \Dot{W}^{2,p}_\Dir(\calD,w_{\gamma,\nu}^\calD)\to L^p(\calD,w_{\gamma,\nu}^\calD),\quad u\mapsto\Delta u,
    \]
    is an isomorphism. In particular, for $u\in \Dot{W}^{2,p}_\Dir(\calD,w_{\gamma,\nu}^\calD)$ we have
    \begin{equation}\label{eq:homL:normequiv}
    \nrm{u}_{\Dot{W}^{2,p}(\calD,w_{\gamma,\nu}^\calD)}
    \eqsim
    \nrm{\Delta u}_{L^p(\calD,w_{\gamma,\nu}^\calD)}
    \eqsim
    \nrm{D^2 u}_{L^p(\calD,w_{\gamma,\nu}^\calD)}.
    \end{equation}
\end{proposition}

For the proof of Proposition~\ref{prop:Laplace:hom} we use a technique inspired by the proof of~\cite[Corollary~5.2]{Prss2007HinftycalculusFT}.
The main idea is to reduce the statement about the Dirichlet Laplacian on the cone $\calD$ to properties of the functional calculus of the Dirichlet Laplace--Beltrami operator $-\Delta_\Dir^{\Omega}$ on $\Omega$, the unit sphere intersected with $\calD$, in suitable weighted $L^p$-spaces. By $\Delta^\calO$ we denote the Laplace--Beltrami operator, represented in stereographic coordinates, see Appendix~\ref{appendix:Dirichlet:Laplace:Beltrami}.  
To prepare for the proof, first note that for all $a\in\R$, the Laplacian in Euler coordinates has the form
\[
\Delta u = T_a^{-1}M_2^{-1}((D_z-a)^2+(d-2)(D_z-a)+\Delta^{\calO})T_a u,\quad u\in\scrD'(\calD).
\]
Recall that for all $1<p<\infty$, $a\in\R$, and $\gamma,\nu\in\R$, 

\[T_a\in\scrL(L^p(\calD,w_{\gamma,\nu}^\calD),L^p(\widehat{\calD},w_{\gamma,\nu+d-ap}^{\widehat{\calD}}))\]
and 
\[M_2\in\scrL(L^p(\widehat{\calD},w_{\gamma,\nu+d-ap}^{\widehat{\calD}}) ,L^p(\widehat{\calD},w_{\gamma,\nu+d-(2+a)p}^{\widehat{\calD}}))\]
are isomorphisms.
Moreover, also the bounded operator 
\[T_a\in\scrL(\Dot{W}^{2,p}_\Dir(\calD,w_{\gamma,\nu}^\calD),
W^{2,p}_\Dir(\widehat{\calD},w_{\gamma,\nu+d-(2+a)p}^{\widehat{\calD}}))\]
is an isomorphism,  provided $\gamma\in(-1,2p-1)\setminus \{p-1\}$ (see~\eqref{eq:Ta2:isom} with $k=2$). 
Thus, Proposition~\ref{prop:Laplace:hom} holds if, and only if, for $\gamma,\nu\in\R$ satisfying \eqref{eq:range:gamma-nu} the mapping
\[
W^{2,p}_\Dir(\widehat{\calD},w_{\gamma,\nu+d-(2+a)p}^{\widehat{\calD}})\ni u\mapsto \big((D_z-a)^2+(d-2)(D_z-a)+\Delta^\calO \big) u\in L^p(\widehat{\calD},w_{\gamma,\nu+d-(2+a)p}^{\widehat{\calD}})
\]
is an isomorphism for some (and hence for all) $a\in\R$.
In particular, if we want to get rid of the weight with respect to the distance to the tip of $\calD$, we may choose 
$a=a(\nu,p,d)\in\R$ such that $\nu+d-(2+a)p=0$, and it is enough to show that 
\begin{equation}\label{eq:B+A:isom}
W^{2,p}_\Dir(\widehat{\calD},w_{\gamma,0}^{\widehat{\calD}})\ni u\mapsto \big((D_z-a)^2+(d-2)(D_z-a)+\Delta^\calO\big) u\in L^p(\widehat{\calD},w_{\gamma,0}^{\widehat{\calD}})
\end{equation}
is an isomorphism for this particular $a$.

To prove the latter for $\gamma,\nu\in\R$ satisfying~\eqref{eq:range:gamma-nu}, we will make use of three auxiliary results. 
For a precise definition of the UMD property for Banach spaces, its properties and relevance we refer to \cite[Chapter 4]{aibs1}. $L^p$-spaces have the UMD property and so do $L^p$-spaces of $X$-valued functions if $X$ is a UMD Banach space, see \cite[Proposition 4.2.15]{aibs1}.

\begin{lemma}\label{PruSim_B}
        Let $a,b,c \in \C$, let $P_{a,b,c}(z):= az^2+bz+c, z\in \C$, let $1<p<\infty$, and let $X$ be a UMD Banach space. 
        Let $B_{a,b,c}$ be the unbounded linear operator on $L^p(\R;X)$ given by
        \[
        B_{a,b,c}u:=P_{a,b,c}(D)u=aD^2u+bDu+cu,
        \quad \text{with}\quad \sfD(B_{a,b,c}):=
        W^{2,p}(\R;X).
        \]
        Then 
        $B_{a,b,c}$ is a closed operator with
        $\sigma(B_{a,b,c})= P_{a,b,c}(i\R)$. Moreover, if $\Re(a)<0 $ and $\Re(c)>0$, then $B_{a,b,c}$ has a bounded $H^\infty$-calculus with $\omega_{H^\infty}(B_{a,b,c})<\frac{\pi}{2}$.  
    \end{lemma}
    \begin{proof}
      This result is a consequence of Mikhlin's multiplier theorem \cite[Theorem 5.3.18]{aibs1}
      in the same manner as \cite[Theorem~5.1]{GalVer2017}. 
    \end{proof}

The next lemma addresses the Laplace--Beltrami operator part in~\eqref{eq:B+A:isom}. See Appendix \ref{appendix:Dirichlet:Laplace:Beltrami} for a definition of $\Delta^\calO$.

\begin{lemma}\label{PruSim_L}
    Let $X$ be a UMD Banach space, let $1<p<\infty$, and let $\gamma\in(-1,2p-1)\setminus \{p-1\}$. 
    Let $\Delta_\Dir^{\calO,X}$ be the unbounded operator on $L^p(\calO,w_\gamma^\calO;X)$ given by 
    \[\Delta_\Dir^{\calO,X}u:= \Delta^\calO u, \quad\text{with}\quad \sfD(\Delta_\Dir^{\calO,X}):= W^{2,p}_\Dir(\calO,w_\gamma^\calO;X). \]
    Then the following holds.
    \begin{enumerate}[label=\textup{(\roman*)}]
		\item\label{PruSim_L:it:1} $\Delta_\Dir^{\calO,X}$ is a closed linear operator  with
        \[
        \nrm{\Delta^\calO u}_{L^p(\calO,w_\gamma^\calO;X)} + \nrm{u}_{L^p(\calO,w_\gamma^\calO;X)}
        \eqsim
        \nrm{u}_{W^{2,p}_\Dir(\calO,w_\gamma^\calO;X)},
        \quad 
        u\in W^{2,p}_\Dir(\calO,w_\gamma^\calO;X).
        \]
        Moreover, for all $\varphi>0$ there exists a $\tilde{\lambda}\in \R$ such that for all $\lambda>\tilde{\lambda}$ the operator $\lambda-\Delta_\Dir^{\calO,X}$ has a bounded $H^\infty$-calculus of angle $\omega_{H^\infty}(\lambda-\Delta_\Dir^{\calO,X})\le \varphi$.
		\item\label{PruSim_L:it:2} Let $X=\C$ or let $X=L^p(S)$, where $(S,\Sigma,\mu)$ is a $\sigma$-finite measure space.
        Then $\sigma(-\Delta_\Dir^{\calO,X})=\{\lambda_n\colon n\in\N\}\subseteq (0,\infty)$ with $0<\lambda_{k}\le \lambda_{k+1}<\infty, k\in\N.$ For all $\lambda> -\lambda_1$, the operator $\lambda-\Delta_\Dir^{\calO,X}$ has a bounded $H^\infty$-calculus of angle zero. 
        
		\end{enumerate} 
\end{lemma}
\begin{proof}
	Part \ref{PruSim_L:it:1} follows from Theorem \ref{Dirichlet:Laplace:Beltrami:vectorvalued:functionalcalculus}. 
    For the proof of \ref{PruSim_L:it:2}, assume first that $X=\C$. Then~\ref{PruSim_L:it:2} is an immediate consequence of Theorem \ref{Dirichlet:Laplace:Beltrami:functionalcalculus}.
    To verify~\ref{PruSim_L:it:2} for $X=L^p(S)$ note that for all $\lambda\in\C\setminus\sigma(-\Delta_\Dir^{\calO,\C})$, 
    \[
    (R_\lambda f)(\cdot,x) := R(\lambda,-\Delta_\Dir^{\calO,\C})f(\cdot,x), \quad x\in S,\quad f\in L^p(\calO,w_\gamma^\calO;L^p(S)),
    \]
    defines a resolvent for $-\Delta_\Dir^{\calO,L^p(S)}$ at $\lambda$. 
    Moreover, applying Fubini's theorem yields that all relevant properties of this resolvent are inherited from the scalar case.
\end{proof}
\begin{remark}
    If $d=2$ and $S^1$ is parametrised as in Remark \ref{rem:dim:two:cone}, then the Laplace--Beltrami operator on $\Omega$ is just the Laplacian on an interval. In this case, Lemma \ref{PruSim_L} drastically simplifies, as it directly follows from \cite[Theorem 6.1 and Corollary 6.2]{LindemulderVeraar2020} instead of Theorem \ref{Dirichlet:Laplace:Beltrami:vectorvalued:functionalcalculus} and Theorem \ref{Dirichlet:Laplace:Beltrami:functionalcalculus}.
\end{remark}
Put together, the two lemmas above yield the following result.

\begin{lemma}\label{PruSim:domain:B+L}
Let $1<p<\infty$, let $\gamma\in (-1,2p-1)\setminus\{p-1\}$, and
let $\Delta_\Dir^\calO:=\Delta_\Dir^{\calO,X}$ be as introduced in Lemma~\ref{PruSim_L} with $X=L^p(\R)$.
Moreover, let $a<0$,  $b,c\in\C$, and let $B:=B_{a,b,c}$ be as defined in Lemma~\ref{PruSim_B} with $X=L^p(\calO,w_\gamma^\calO)$.
Then the operator sum $B+(-\Delta_\Dir^\calO)$ with domain
\[
\sfD(B+(-\Delta_\Dir^\calO))
    :=
\sfD(B)\cap \sfD(-\Delta_\Dir^\calO)
\]
is a closed operator in $L^p(\R;L^p(\calO,w_\gamma^\calO))$ and
\begin{equation}\label{domainBplusLSobolevSpace}
	\sfD(B+(-\Delta_\Dir^\calO))
    =
    W^{2,p}_{\textup{Dir}}(\widehat{\calD},w_{\gamma,0}^{\widehat{\calD}})
    \quad\text{with equivalent norms}.
\end{equation}
\end{lemma}
\begin{proof}
Lemma~\ref{PruSim_B} yields that there is an $\omega\in\R$ such that $\omega+B$ has a bounded $H^\infty$-calculus with angle $\omega_{H^\infty}(\omega+B)<\frac{\pi}{2}$.
Moreover, Lemma~\ref{PruSim_L} yields that $\lambda+(-\Delta_\Dir^\calO)$ has a bounded $H^\infty$-calculus of angle zero for all $\lambda>-\lambda_1$.
Thus, \cite[Theorem~8.5]{Prss1993EvolutionaryIE} yields the asserted closedness of the operator sum.
The identity~\eqref{domainBplusLSobolevSpace} together with the equivalence of the norms follows from Lemma~\ref{lem:mixedDeriv} together with the closedness of both $B$ and $-\Delta_\Dir^\calO$.
\end{proof}

Now we are ready to present the proof of the main result in this section.

\begin{proof}[Proof of Proposition~\ref{prop:Laplace:hom}]
Let $1<p<\infty$ and let $\gamma,\nu\in\R$ satisfy~\eqref{eq:range:gamma-nu}. We only have to prove that $\Dot{\Delta}_\Dir^{p,\gamma,\nu}$ is an isomorphism, which automatically implies ~\eqref{eq:homL:normequiv}.
Let $a:=\frac{\nu+d}{p}-2$, so that $\nu+d-(2+a)p=0$.
As explained above, 
proving that $\Dot{\Delta}_\Dir^{p,\gamma,\nu}$ is an isomorphism, is equivalent to proving that the operator in~\eqref{eq:B+A:isom} is isomorphic. 
To verify the latter,
let $B$ be as in Lemma~\ref{PruSim_B} with $$ a\curvearrowleft -1,\quad b\curvearrowleft-(d-2-2a),\quad c\curvearrowleft -a(a-d+2)$$ and $X \curvearrowleft L^p(\calO,w_{\gamma}^\calO)$ and let $ -\Delta_\Dir^\calO := -\Delta_\Dir^{\calO,X}$  be as in Lemma~\ref{PruSim_L} with $ X\curvearrowleft L^p(\R) $. 
These two lemmas together with \cite[Theorem~8.5]{Prss1993EvolutionaryIE} yield
\[
\sigma(B+(-\Delta_\Dir^\calO))\cap\R
\subseteq
\{\sigma(B)+\sigma(-\Delta_\Dir^\calO)\}\cap\R
=
\Big\{\lambda_n-a(a-d+2)\colon n\in\N\Big\}.
\]
In particular, $0\notin\sigma(B+(-\Delta_\Dir^\calO))$ whenever $a(a-d+2)\notin\{\lambda_n\colon n\in\N\}$, i.e., whenever
\[
\nu\notin\Big\{ \Big(\tfrac{2+d}{2}\pm\sqrt{\lambda_n +\tfrac{(d-2)^2}{4}}\Big)p-d\colon n\in\N\Big\},
\]
which is exactly our assumption on  $\nu$.
Therefore, since we also assume that $\gamma\in (-1,2p-1)\setminus\{p-1\}$,  Lemma~\ref{PruSim:domain:B+L} yields that 
\begin{align*}
  B+(-\Delta_\Dir^\calO)
\colon 
W^{2,p}_\Dir(\widehat{\calD},w_{\gamma,0}^{\widehat{\calD}})
\to
L^p(\widehat{\calD},w_{\gamma,0}^{\widehat{\calD}}),
\intertext{given by}
u\mapsto -\big((D_z-a)^2+(d-2)(D_z-a)+\Delta^\calO\big)u,  
\end{align*}
is an isomorphism and the assertion holds. 
\end{proof}

\begin{remark}
The proof strategy above is inspired by \cite[Section~5]{Prss2007HinftycalculusFT}, where a similar technique is employed for the analysis of the heat equation on wedges.
However, since we only consider the elliptic equation on conical domains, we avoid many of the complications therein, as our operators $B$ and $-\Delta_\Dir^\calO$ are resolvent commuting.
This gives us more flexibility, which results in an extended range of admissible parameters $\nu\in\R$. 
A similar strategy was employed in \cite{KoehneSaalWestermann2021,koehne2024optimalregularitystokesequations} for the Neumann case. 
In both the Dirichlet and the Neumann case, it is crucial that the boundary condition is fully incorporated 
in the corresponding Dirichlet or Neumann Laplace--Beltrami operators on $\calO$.

Although we concentrate on the Laplacian, it is likely that the proof technique above can be modified in order to obtain appropriate isomorphisms for larger classes of differential operators.
A crucial tool is the application of \cite[Theorem~8.5]{Prss1993EvolutionaryIE} in the proof of Lemma~\ref{PruSim:domain:B+L}. This result can be modified as follows: Instead of assuming that both $A$ and $B$ have bounded imaginary powers on some UMD space $X$, we may assume that $A$ has a bounded $H^\infty$-calculus and $B$ is $R$-sectorial such that the sum of the respective angles is smaller than $\pi$. This and \cite[Theorem 6.3]{Kalton2001TheH} ensure that $A+B$ with domain $\sfD(A)\cap\sfD(B)$ is closed, which was previously ensured by the Dore-Venni Theorem. The rest of the proof of \cite[Theorem~8.5]{Prss1993EvolutionaryIE} stays the same.
This implies that instead of considering the sum $B+(-\Delta_\Dir^{\calO})$ in the proof above, we could sum up $B+C$ for any (reasonable) linear operator $C$ that is $R$-sectorial and is resolvent commuting with $B$. 
\end{remark}

\subsection{Proof of the main theorem}\label{sec:main:proof}

Using the results from the previous sections, we are now able to prove our main theorem concerning the functional calculus of the negative Dirichlet Laplacian.

\begin{proof}[Proof of Theorem~\ref{thm:Laplace:inhom}]
Parts~\ref{it:Laplace:inhom:closed} and \ref{it:Laplace:inhom:equivnrms} are immediate consequences of Proposition~\ref{prop:Laplace:hom} and Lemma~\ref{lem:isom:unbdd} with $X\curvearrowleft L^p(\calD,w_{\gamma,\nu}^\calD)$ and $Y\curvearrowleft \Dot W^{2,p}_\Dir(\calD,w_{\gamma,\nu}^\calD)$. 
For part~\ref{it:Laplace:hom:domain} we need, in addition, that $\Dot{W}^{2,p}_\Dir(\calD,w_{\gamma,\nu}^\calD)\cap L^p(\calD,w_{\gamma,\nu}^\calD)$ is dense in $\Dot W^{2,p}_\Dir (\calD,w_{\gamma,\nu}^\calD)$, which is guaranteed by Lemma~\ref{lem:smooth:dense}.

It remains to prove~\ref{it:Laplace:inhom:calculus}.
To this end, fix $1<p<\infty$ and let $\gamma,\nu\in\R$ satisfy~\eqref{eq:range:gamma-nu} and~\eqref{eq:range:calc:nu}.
Moreover, let $(L,\sfD(L))$ be the negative Dirichlet Laplacian in $L^2(\calD)$ introduced in Proposition~\ref{functionalcalculus:variationalsetting}.
Then, due to Theorem~\ref{functional:calculus:laplace:wedge:Muckenhoupt}, the semigroup $(S(t))_{t\geq 0}:=(e^{-tL})_{t\geq 0}$ generated by $(-L,\sfD(L))$ extrapolates to a strongly continuous semigroup $(S_{p,\gamma,\nu}(t))_{t\geq 0}$ on $L^p(\calD,w_{\gamma,\nu}^\calD)$. We write $(L_{p,\gamma,\nu},\sfD(L_{p,\gamma,\nu}))$ for the negative of the generator of $(S_{p,\gamma,\nu}(t))_{t\geq 0}$. Due to the consistency of the semigroups $S$ and $S_{p,\gamma,\nu}$ and of the corresponding resolvents, Estimate~\eqref{eq:extrapol:ApPower} from Theorem~\ref{functional:calculus:laplace:wedge:Muckenhoupt} yields that $(L_{p,\gamma,\nu},\sfD(L_{p,\gamma,\nu}))$ is sectorial and has a bounded $H^\infty$-calculus of angle zero. Thus, in order to obtain the assertion, it suffices to show that
\[
(\Delta_\Dir^{p,\gamma,\nu},\sfD(\Delta_\Dir^{p,\gamma,\nu}))
= 
(-L_{p,\gamma,\nu},\sfD(L_{p,\gamma,\nu})).
\]
We first show that $\Delta_\Dir^{p,\gamma,\nu}\subseteq -L_{p,\gamma,\nu}$, i.e., that
\begin{equation}\label{eq:inclusion1}
\sfD(\Delta_\Dir^{p,\gamma,\nu})\subseteq \sfD(L_{p,\gamma,\nu}) 
\quad\text{and}\quad
-L_{p,\gamma,\nu}u=\Delta u,\,\, u\in \sfD(\Delta_\Dir^{p,\gamma,\nu}).    
\end{equation}
To this end,
we first note that,
due to the consistency of the semigroups, 
for all $u\in \sfD(\Delta_\Dir^{p,\gamma,\nu}) \cap \sfD(L)$ and $t>0$,
\begin{align*}
S_{p,\gamma,\nu}(t)u - u
=
S(t)u - u
&=
\int_0^t S(s)(-L)u\,\mathrm{d}s\\
&=
\int_0^t S(s) \Delta u\,\mathrm{d}s
=
\int_0^t S_{p,\gamma,\nu}(s) \Delta u\,\mathrm{d}s,
\end{align*}
where the integrals converge in $L^p(\calD,w_{\gamma,\nu}^{\calD})\cap L^2(\calD)$.
Thus,
for all $u\in \sfD(\Delta_\Dir^{p,\gamma,\nu}) \cap \sfD(L)$,
since $\Delta u\in L^p(\calD,w_{\gamma,\nu}^\calD)$ and $S_{p,\gamma,\nu}$ is a strongly continuous semigroup on $L^p(\calD,w^\calD_{\gamma,\nu})$, we have 
\[
\frac{S_{p,\gamma,\nu}(t)u - u}{t}
\longrightarrow 
\Delta u
\quad \text{in} \quad L^p(\calD,w_{\gamma,\nu}^\calD)\quad  \text{ for } t\searrow  0.
\]
Therefore, 
$-L_{p,\gamma,\nu}u = \Delta u$ 
for all $u\in \sfD(\Delta_\Dir^{p,\gamma,\nu}) \cap \sfD(L)$.
In order to obtain~\eqref{eq:inclusion1} we use the density of $\sfD(\Delta_\Dir^{p,\gamma,\nu}) \cap \sfD(L)$ in $\sfD(\Delta_\Dir^{p,\gamma,\nu})$, which follows from Lemma~\ref{lem:smooth:dense} and the fact that 
$C_{c,\Dir}^2(\overline{\calD}\setminus\{0\})\subseteq \sfD(L)$.
Let $u\in \sfD(\Delta_\Dir^{p,\gamma,\nu})$ and let $(u_n)_{n\in\N}\subseteq \sfD(\Delta_\Dir^{p,\gamma,\nu}) \cap \sfD(L)$ be such that $u_n\to u$ in $\sfD(\Delta_\Dir^{p,\gamma,\nu})$. Then, 
\[
u_n\to u \text{ in } L^p(\calD,w_{\gamma,\nu}^\calD) \quad \text{and}\quad -L_{p,\gamma,\nu} u_n=\Delta u_n\to \Delta u\text{ in } L^p(\calD,w_{\gamma,\nu}^\calD)\qquad (n\to\infty).
\]
Thus, since $L_{p,\gamma,\nu}$ is a closed operator in $L^p(\calD,w_{\gamma,\nu}^\calD)$, it follows that $u\in \sfD(L_{p,\gamma,\nu})$ and $-L_{p,\gamma,\nu}u=\Delta u$. Since $u\in\sfD(\Delta_\Dir^{p,\gamma,\nu})$ was arbitrary,~\eqref{eq:inclusion1} is proven.

In order to show the reverse direction, i.e., that $-L_{p,\gamma,\nu}\subseteq \Delta_\Dir^{p,\gamma,\nu}$, it suffices to verify that $1-\Delta_\Dir^{p,\gamma,\nu}$ is surjective and that $1+L_{p,\gamma,\nu}$ is injective. The latter follows from the fact that $L_{p,\gamma,\nu}$ is sectorial. 
So it only remains to prove that $1-\Delta_\Dir^{p,\gamma,\nu}$ is surjective. We argue as follows:
First, let  $f\in L^p(\calD,w_{\gamma,\nu}^\calD)\cap L^2(\calD)$. Then, $u:=R(1,-L)f\in\sfD(L)$ satisfies $u-\Delta u = f$. Since the resolvents $R(1,-L)$ and $R(1,-L_{p,\gamma,\nu})$ are consistent, we also get that $u\in\sfD(L_{p,\gamma,\nu})\subseteq L^p(\calD,w_{\gamma,\nu}^\calD)$. 
In particular, $u-f\in L^p(\calD,w_{\gamma,\nu}^\calD)$, so that, by Proposition~\ref{prop:Laplace:hom}, there exists a unique $v\in\Dot{W}^{2,p}_\Dir(\calD,w_{\gamma,\nu}^\calD)$, such that $\Delta v = u-f =\Delta u\in L^2(\calD)$. Thus, in the sense of distributions,
\[
(u,\Delta\psi) = (\Delta u,\psi) = (\Delta v,\psi) = (v,\Delta\psi), \quad \psi\in L^2(\calD).
\]
Therefore,
since for all $\varphi\in C_c^\infty(\calD)$ there exists $\psi\in L^2(\calD)$, such that $\Delta\psi=\varphi$, we have 
\[
(v,\varphi)= (u,\varphi),\quad\varphi\in C_c^\infty(\calD),
\]
which shows that $v=u$ and therefore $u\in \Dot{W}^{2,p}_\Dir(\calD,w_{\gamma,\nu}^\calD)\cap L^p(\calD,w_{\gamma,\nu}^\calD) = \sfD(\Delta_\Dir^{p,\gamma,\nu})$ and $u-\Delta u=f$.

Now let $f\in L^p(\calD,w_{\gamma,\nu}^\calD)$. Then, since $L^p(\calD,w_{\gamma,\nu}^\calD)\cap L^2(\calD)\subseteq L^p(\calD,w_{\gamma,\nu}^\calD)$ is dense, we can choose a sequence $(f_n)_{n\in\N}\subseteq L^p(\calD,w_{\gamma,\nu}^\calD)\cap L^2(\calD)$, such that $\nrm{f_n-f}_{L^p(\calD,w_{\gamma,\nu}^\calD)}\to 0$ for $n\to\infty$.  
Then, by the continuity of the resolvent,
\[
u_n:=R(1,-L_{p,\gamma,\nu})f_n\longrightarrow R(1,-L_{p,\gamma,\nu})f=:u\,\,\text{ in }L^p(\calD,w_{\gamma,\nu}^\calD)\qquad (n\to\infty).
\]
Therefore, since, by the first part, $u_n\in \sfD(\Delta_\Dir^{p,\gamma,\nu})$, $n\in\N$, and $\Delta_\Dir^{p,\gamma,\nu}\subseteq -L_{p,\gamma,\nu}$, 
\[
(1-\Delta_\Dir^{p,\gamma,\nu})u_n 
=
(1+L_{p,\gamma,\nu})u_n
\longrightarrow
(1+L_{p,\gamma,\nu})u 
= 
f
\,\,\text{ in }L^p(\calD,w_{\gamma,\nu}^\calD)\quad (n\to\infty).
\]
Thus, since $\Delta_\Dir^{p,\gamma,\nu}$ is closed, 
$u\in \sfD(\Delta_\Dir^{p,\gamma,\nu})$ and $u-\Delta u=f$.
This shows the required surjectivity, which concludes the proof.
\end{proof}
\subsection{Sharpness: The necessity of singular weights}
In the next proposition we adapt an idea from \cite{KoehneSaalWestermann2021}, showing that the restriction for $\nu$ in \eqref{eq:range:gamma-nu} is a necessary consequence, which comes from the choice of the domain.
For a Banach space $X$ we define $\sS'(\R;X)$ as the space of $X$-valued tempered distributions. For $u\in \sS'(\R;X)$ we write $\sF u$ and $\sF^{-1}u$ for the Fourier transform and the inverse Fourier transform of $u$, respectively. 
For more information on the Fourier transform and  operator-valued Fourier multiplier operators we refer to \cite[Section~2.4]{aibs1} and \cite[Section~4.3]{PrssSimonett2016}, respectively. 
\begin{proposition}\label{prop:counterexample}
    Let $1<p<\infty, \gamma\in (-1,2p-1)\setminus\{p-1\}$, let 
    \[\nu\in \Big\{\Big(\tfrac{2+d}{2}\pm \sqrt{\lambda_n+\tfrac{(d-2)^2}{4}}\Big)p-d\colon n\in\N\Big\}.\]
    Define 
    \[
    \Dot \Delta_\Dir^{p,\gamma,\nu}\colon \Dot W^{2,p}_\Dir(\calD,w_{\gamma,\nu}^\calD)\to L^p(\calD,w_{\gamma,\nu}^\calD), \quad u\mapsto \Delta u,
    \]
    and let $\Delta_\Dir^{p,\gamma,\nu}$ be the unbounded operator on $L^p(\calD,w_{\gamma,\nu}^\calD)$ given by 
\[\Delta_\Dir^{p,\gamma,\nu}u:=\Delta u \quad \text{with}\quad \sfD(\Delta_\Dir^{p,\gamma,\nu}):=\Dot W^{2,p}_\Dir(\calD,w_{\gamma,\nu}^\calD)\cap L^p(\calD,w_{\gamma,\nu}^\calD).   \]
Then the following assertions hold.

\begin{enumerate}[label=\textup{(\roman*)}]
    \item\label{it:counterexample:1} $\displaystyle\Dot\Delta_\Dir^{p,\gamma,\nu}$ is not an isomorphism.
    \item\label{it:counterexample:2} $\displaystyle\varrho(-\Delta_\Dir^{p,\gamma,\nu})=\emptyset$. 
\end{enumerate}
\end{proposition}

\begin{proof}
We begin with~\ref{it:counterexample:1}. Assume that $\Dot\Delta_\Dir^{p,\gamma,\nu}$ is an isomorphism. 
Let $B$ and $-\Delta_\Dir^\calO$ be as defined in the proof of Proposition~\ref{prop:Laplace:hom}, let $P(z)=-z^2-(d-2-2a)z-a(a-d+2)$, $z\in\C$, with $a:=\frac{\nu+d}{p}-2$,
and let $B+(-\Delta_\Dir^\calO)$ be defined as the corresponding operator sum with domain $$\sfD(B+(-\Delta_\Dir^\calO)):= \sfD(B)\cap\sfD(-\Delta_\Dir^\calO)= W^{2,p}_\Dir(\widehat{\calD},w_{\gamma,0}^{\widehat\calD})$$ 
Since $\Dot\Delta_\Dir^{p,\gamma,\nu}$ was assumed to be an isomorphism, so is $B-\Delta_\Dir^\calO\colon W^{2,p}_\Dir(\widehat{\calD},w_{\gamma,0}^{\widehat{\calD}})\to L^p(\widehat{\calD},w_{\gamma,0}^{\widehat{\calD}})$, cf.~\eqref{eq:B+A:isom}.
Therefore, for any $f\in L^p(\widehat\calD,w_{\gamma,0}^{\widehat\calD})$ there exists a unique $u\in W^{2,p}_\Dir(\widehat\calD,w_{\gamma,0}^{\widehat\calD})$ with $(B-\Delta_\Dir^\calO)u=f.$ 
It follows that 
\[-\sF^{-1}[(-P(i\cdot)-(-\Delta_\Dir^\calO))\sF u]=f,\]
which, in turn, ensures that  
\[-\sF^{-1}[R(-P(i\cdot),-\Delta_\Dir^\calO)\sF f]= u.\]
However, by the choice of $\nu$ (and $a$), it holds that $-P(0)=a(a-d+2)\in \sigma(-\Delta_\Dir^{\calO})$.
As a consequence, the continuous symbol
\[
\R\setminus\{0\}\ni \xi\mapsto R(-P(i\xi),-\Delta_\Dir^{\calO})\in \sL(L^p(\calO,w_{\gamma}^{\calO})),
\]
is not in $L^\infty(\R;\sL(L^p(\calO,w_{\gamma}^{\calO})))$.
Thus, due to~\cite[Proposition~4.3.2]{PrssSimonett2016}, the corresponding Fourier multiplier operator $-\sF^{-1}[R(-P(i\cdot),-\Delta_\Dir^\calO)\sF]$ is not bounded in $L^p(\widehat{\calD},w_{\gamma,0}^{\widehat{\calD}})$.
This contradicts the assumption that $B-\Delta_\Dir^\calO$ is an isomorphism and therefore proves \ref{it:counterexample:1}.

To prove \ref{it:counterexample:2}, we first note that both $\sigma(-\Delta_\Dir^{p,\gamma,\nu})$ and $\varrho(-\Delta_\Dir^{p,\gamma,\nu})$ are $\R_+$-scaling invariant subsets of $\C$. To see this, let $J_\lambda f(x):= f(\lambda x)$ for a function $f \in L^p(\calD,w_{\gamma,\nu}^\calD)$ and $\lambda>0$. Assume that $\mu \in\varrho(-\Delta_\Dir^{p,\gamma,\nu})\setminus\{0\}$. Then for $\lambda>0$
    \begin{equation}
        \begin{aligned}
            \nrm{R(\mu\lambda^2,-\Delta_\Dir^{p,\gamma,\nu})}_{\sL(L^p(\calD,w_{\gamma,\nu}^\calD))}&= \nrm{R(\mu\lambda^2, -\lambda^2 J_\lambda \Delta_\Dir^{p,\gamma,\nu}J_\lambda^{-1})}_{\sL(L^p(\calD,w_{\gamma,\nu}^\calD))}\\
            &=\frac{1}{\lambda^2}\nrm{R(\mu, -J_\lambda \Delta_\Dir^{p,\gamma,\nu}J_\lambda^{-1})}_{\sL(L^p(\calD,w_{\gamma,\nu}^\calD))}\\
            &=\frac{1}{\lambda^2}\nrm{J_\lambda R(\mu,-\Delta_\Dir^{p,\gamma,\nu})J_\lambda^{-1}}_{\sL(L^p(\calD,w_{\gamma,\nu}^\calD))}\\
            &\le  \frac{1}{\lambda^2}\nrm{R(\mu,-\Delta_\Dir^{p,\gamma,\nu})}_{\sL(L^p(\calD,w_{\gamma,\nu}^\calD))}.
        \end{aligned}
    \end{equation}
As $\varrho(-\Delta_\Dir^{p,\gamma,\nu})$ is open and $\R_+$-scaling invariant, it is ensured that it contains some open sector, i.e., there exist $-\pi\le \phi<\psi\le \pi$ such that 
$$\varrho(-\Delta_\Dir^{p,\gamma,\nu})\supseteq\{z\in\C\setminus \{0\}\colon \arg(z)\in(\phi,\psi)\}. $$
    This ensures that $-\Delta_\Dir^{p,\gamma,\nu}$ is, up to rotation, sectorial in the sense of \cite[Definition~10.1.1]{aibs2}.
    Now let $u\in \Dot W^{2,p}_\Dir(\calD,w_{\gamma,\nu}^\calD)\cap L^p(\calD,w_{\gamma,\nu}^\calD)$. Then, for  $\mu\in \varrho(-\Delta_\Dir^{p,\gamma,\nu})$ and $\lambda>0$,
    \begin{equation}\label{eq:counterexample:1}
        \begin{aligned}
            \nrm{u}_{\Dot W^{2,p}(\calD,w_{\gamma,\nu}^\calD)}&=\lambda^{\frac{d+\nu}{p}-2}\nrm{J_\lambda u}_{\Dot W^{2,p}(\calD,w_{\gamma,\nu}^\calD)}\\
        &\le C \lambda^{\frac{d+\nu}{p}-2}\nrm{(\mu+\Delta_\Dir^{p,\gamma,\nu})J_\lambda u}_{L^p(\calD,w_{\gamma,\nu}^\calD)}\\
        &=C\lambda^{\frac{d+\nu}{p}-2}\nrm{J_\lambda(\mu+\lambda^2\Delta_{\Dir}^{p,\gamma,\nu})u}_{L^p(\calD,w_{\gamma,\nu}^\calD)}\\
        &=C \lambda^{\frac{d+\nu}{p}}\nrm{J_\lambda\big(\tfrac{\mu}{\lambda^2}+\Delta_\Dir^{p,\gamma,\nu}\big)u}_{L^p(\calD,w_{\gamma,\nu}^\calD)}\\
        &= C\nrm{\big(\tfrac{\mu}{\lambda^2}+\Delta_\Dir^{p,\gamma,\nu}\big)u}_{L^p(\calD,w_{\gamma,\nu}^\calD)}.
        \end{aligned}
    \end{equation}
    Letting $\lambda\rightarrow \infty$ gives $$\nrm{u}_{\Dot W^{2,p}(\calD,w_{\gamma,\nu}^\calD)}\lesssim \nrm{\Delta_\Dir^{p,\gamma,\nu}u}_{L^p(\calD,w_{\gamma,\nu}^\calD)},\qquad  u\in \Dot W^{2,p}_\Dir (\calD,w_{\gamma,\nu}^\calD)\cap L^p(\calD,w_{\gamma,\nu}^\calD), $$
    which proves, in particular, that $-\Delta_\Dir^{p,\gamma,\nu}$ is injective. This and \cite[Proposition~10.1.9]{aibs2} hence prove that $-\Delta_\Dir^{p,\gamma,\nu}$ has dense range. Therefore, for all $f\in L^p(\calD,w_{\gamma,\nu}^\calD)$ there exists a sequence $(u_n)_{n\in\N}$ in $\Dot W^{2,p}_\Dir (\calD,w_{\gamma,\nu}^\calD)\cap L^p(\calD,w_{\gamma,\nu}^\calD) $ such that $f_n:= -\Delta_\Dir^{p,\gamma,\nu}u_n \rightarrow f$ in $L^p(\calD,w_{\gamma,\nu}^\calD)$ as $n\to\infty$. 
    Then \eqref{eq:counterexample:1} ensures that $(u_n)_{n\in\N}$ is a Cauchy sequence in $\Dot W^{2,p}_\Dir(\calD,w_{\gamma,\nu}^\calD)$ with $u_n \to u $ in $ \Dot W^{2,p}_\Dir(\calD,w_{\gamma,\nu}^\calD) $ as $ n\rightarrow\infty$ and $-\Delta_\Dir^{p,\gamma,\nu}u=f$. However, then $u\in \Dot W^{2,p}_\Dir(\calD,w_{\gamma,\nu}^\calD)$, which contradicts \ref{it:counterexample:1}.
\end{proof}
\subsection{Extension to wedge domains}\label{sec:wedge}
In this section we extend the results from Theorem~\ref{thm:Laplace:inhom}\ref{it:Laplace:inhom:calculus} on the $H^\infty$-calculus of the Dirichlet Laplacian from the $C^2$-cone $\calD=\calD_\Omega$ to the corresponding wedge domain $\calD\times\R^m$ with arbitrary co-dimension $m\in\N$.
We introduce the following additional notation: Let $v\colon \R^m\to (0,\infty)$ be a weight. Let $\widetilde\Delta_\Dir^{p,\gamma,\nu}$ be the unbounded operator on $L^p(\calD,w_{\gamma,\nu}^\calD;L^p(\R^m,v))$ given by 
\[\widetilde\Delta_\Dir^{p,\gamma,\nu}u:=\Delta u \quad \text{with}\quad \sfD(\widetilde\Delta_\Dir^{p,\gamma,\nu}):=\Dot W^{2,p}_\Dir(\calD,w_{\gamma,\nu}^\calD;L^p(\R^m,v))\cap L^p(\calD,w_{\gamma,\nu}^\calD;L^p(\R^m,v)),\]
where we note that the Laplacian here only acts on $\calD$. 
Define $\Delta^{p,v}$ as the   unbounded operator on $L^p(\R^m,v;L^p(\calD,w_{\gamma,\nu}^\calD))$ given by 
    \[ \Delta^{p,v}u:= \Delta u \quad\text{with}\quad \sfD(\Delta^{p,v}):= W^{2,p}(\R^m, v;L^p(\calD,w_{\gamma,\nu}^\calD)),\]
    where we note that the Laplacian here only acts on $\R^m$. 

\begin{theorem}\label{thm:Laplace:inhom:wedge}
    Let $m \in \N$, 
    let  $1<p<\infty$, let $v\in A_p(\R^m)$, and let $\gamma,\nu\in\R$ be such that
    \[\begin{aligned}
        \gamma&\in (-1,2p-1)\setminus\{p-1\},\\
        \nu&\in (-d-\lambda_1^*p,(d+\lambda_1^*)p-d)\setminus \{ (2-\lambda_n^*)p-d\colon n\in \N\}.
    \end{aligned}\]
Let $\Delta_\Dir^{p,\gamma,\nu,v}$ be the unbounded operator on $L^p(\calD \times \R^{m},w_{\gamma,\nu}^\calD v)$ given by 
\[  \Delta_\Dir^{p,\gamma,\nu,v}u:= \widetilde\Delta^{p,\gamma,\nu}_\Dir u+\Delta^{p,v}u\quad \text{with}\quad \sfD(\Delta_\Dir^{p,\gamma,\nu,v}):= \sfD(\widetilde\Delta_\Dir^{p,\gamma,\nu})\cap\sfD(\Delta^{p,v}).\]
Then $-\Delta_\Dir^{p,\gamma,\nu, v}$ is a sectorial operator with a bounded $H^\infty$-calculus of angle zero. 
\end{theorem}
\begin{proof}
    A similar argument as in the proof of Lemma \ref{PruSim_L}\ref{PruSim_L:it:2} combined with Theorem~\ref{thm:Laplace:inhom}\ref{it:Laplace:inhom:calculus} ensures that  $-\widetilde \Delta_{\Dir}^{p,\gamma,\nu}$ is  a sectorial operator with a bounded $H^\infty$-calculus of angle $0$. 
    On the other hand, since $L^p(\calD,w_{\gamma,\nu}^\calD)$ is a UMD space,  $-\Delta^{p,v}$ is a sectorial operator with a bounded $H^\infty$-calculus of angle $0$. The proof is similar to the proof of Lemma \ref{PruSim_B}.

    Therefore, since the operators $\widetilde\Delta_\Dir^{p,\gamma,\nu}$ and $\Delta^{p,v}$ are resolvent commuting,  it follows from 
    \cite[Theorem 16.3.10]{aibs3} and \cite[Proposition 7.5.4]{aibs2}
    that $-\Delta_\Dir^{p,\gamma,\nu,v}$  
    is a sectorial operator with a bounded $H^\infty$-calculus of angle zero. This proves the claim. 
\end{proof}

\section{The Poisson equation on conical domains}\label{sec:Poisson}
As a consequence of our analysis in Section~\ref{sec:Laplace:hom}, we obtain the following theorem on the (higher order) regularity of the Poisson equation with zero Dirichlet boundary condition on $\calD$ within the scale of Krylov-Kondratiev type weighted Sobolev spaces $H^s_{p,\Theta,\theta}(\calD)$ introduced in Section~\ref{sec:Krylov:Kondratiev}.
It generalises~\cite[Theorem~4.1(ii)]{cioicalicht2024sobolevspacesmixedweights}, where only the case $p=d=2$ has been considered.
As before, also in this section, $\calD=\calD_\Omega$ is a cone of the form~\eqref{eq:D:convention}, where $\Omega\subseteq S^{d-1}$ is a domain, such that $\overline{\Omega}\subsetneq S^{d-1}$. As in Section~\ref{sec:Dirichlet-Laplacian}, we additionally assume that $\Omega$ admits a $C^2$-boundary. We denote the  Dirichlet Laplace--Beltrami operator on $\Omega$ by $\Delta_\Dir^\Omega$ and we write $(\lambda_n)_{n\in\N}$ for the ordered sequence of the eigenvalues of $-\Delta_\Dir^{\Omega}$.

\begin{theorem}\label{thm:Poisson:Krylov}
    Let $1<p<\infty$, $0\leq s< \infty$, and let $\Theta,\theta\in\R$ be such that
    \begin{equation}\label{eq:range:Theta:theta}
    d-1<\Theta<p+d-1\qquad\text{and}\qquad 
    \tfrac{\theta+p}{p}\notin \Big\{ \tfrac{d+2}{2}\pm\sqrt{\lambda_n+\tfrac{(d-2)^2}{4}}: \ n\in\N\Big\}.
\end{equation}
Then for all $f\in H^s_{p,\Theta+p,\theta+p}(\calD)$ there exists a unique $u\in H^{s+2}_{p,\Theta-p,\theta-p}(\calD)$ such that
\begin{equation}\label{eq:PoissonD}
\Delta u = f \quad \text{on } \calD.
\end{equation}
Moreover, 
\begin{equation}\label{eq:a-priori}
\nrm{ u}_{H^{s+2}_{p,\Theta-p, \theta-p}(\calD) } \lesssim \nrm{f}_{H^{s}_{p,\Theta+p, \theta+p}(\calD)}    
\end{equation}
with a constant that does not depend on $f$ and $u$.
\end{theorem}
\begin{proof}
Let $\gamma:=\Theta-d+p$ and let $\nu:=\theta-d+p$. Then, the restriction~\eqref{eq:range:Theta:theta} on $\Theta$ and $\theta$ imply that $\gamma$ and $\nu$ satisfy~\eqref{eq:range:gamma-nu} with the additional restriction that $p-1<\gamma<2p-1$. 
The latter implies that, due to Lemma~\ref{lem:Krylov:Kondratiev:eq2}, $H^2_{p,\Theta-p,\theta-p}(\calD)=\Dot{W}^{2,p}_\Dir(\calD,w_{\gamma,\nu}^\calD)$.
Moreover, by definition, $H^0_{p,\Theta+p,\theta+p}(\calD)=L^p(\calD,w_{\gamma,\nu}^\calD)$. Thus, by Proposition~\ref{prop:Laplace:hom}, the statement holds for $s=0$. It therefore also holds for $s\in\N$ due to \cite[Theorem~4.1(i)]{cioicalicht2024sobolevspacesmixedweights}, which allows us to lift the regularity. Finally, the statement for $s\in (0,\infty)\setminus\N$ follows by interpolation, see~\cite[Theorem~3.11]{cioicalicht2024sobolevspacesmixedweights}. 
Note that the results in \cite{cioicalicht2024sobolevspacesmixedweights} we used here carry over to the $d$-dimensional case mutatis mutandis. 
\end{proof}
\begin{remark}
    The recent results \cite[Theorem 6.2 \& Theorem 6.4]{lindemulder2025functionalcalculusweightedsobolev} and an adaptation of Theorem \ref{Dirichlet:Laplace:Beltrami:functionalcalculus} ensure that Theorem \ref{PruSim_L} also holds for $C^1$-domains and $\gamma\in (p-1,2p-1)$. Therefore, Theorem \ref{thm:Poisson:Krylov} also holds for $C^1$-cones. 
\end{remark}

\appendix
\section{Supplementary material}
\subsection{Stereographic projections}\label{appendix:stereographic:projection}
In this appendix we construct  stereographic projections that allow for the parametrization of subsets of the unit sphere. Although this seems to be well-known, for the convenience of the reader, we fill in the details. 
Let $ \mathfrak p\in S^{d-1}$ and let $H_{\mathfrak p} = \{x\in \R^d\colon \langle x,\mathfrak p\rangle =0\}$ be the hyperplane orthogonal to $\mathfrak p$ and containing the origin.  
We define the \emph{stereographic projection with pole $\mathfrak p$} as 
\[\pi_{\mathfrak p}:= \Big( S^{d-1}\setminus\{ \mathfrak p\} \ni x \mapsto \frac{\langle x,\mathfrak p\rangle \mathfrak p -x}{\langle x,\mathfrak p\rangle -1} \in H_{\mathfrak p}\Big).\]
For $h \in H_{\mathfrak p}$ we define 
\[ \pi_{\mathfrak p}^{-1}(h):= \frac{1}{1+|h|^2}\big[ (|h|^2-1)\mathfrak p+2h\big].\]
 $\pi_{\mathfrak p}^{-1}$ is well-defined as a mapping from $H_{\mathfrak p}$ to $S^{d-1}\setminus\{\mathfrak p\}$.  Moreover, it holds  that $\pi_\mathfrak{p}^{-1}$ is inverse to $\pi_\mathfrak{p}$, i.e. $\pi_\mathfrak{p}\circ \pi_\mathfrak{p}^{-1} =\textup{Id}_{H_\mathfrak{p}}$ and $\pi^{-1}_\mathfrak{p}\circ\pi_\mathfrak{p} = \textup{Id}_{S^{d-1}\setminus\{\mathfrak{p}\}}$. 
In addition we identify $H_\mathfrak{p}$ with $\R^{d-1}$ in the following way. 
Choose an orthonormal basis $\{v^1,...,v^{d-1}\}$ of $H_\mathfrak{p}$, then every $h\in H_\mathfrak{p}$ has a unique representation $h= \sum_{j=1}^{d-1} \langle h,v^j\rangle v^j. $
This defines an isomorphism $R_\mathfrak{p}\colon H_\mathfrak{p}\rightarrow \R^{d-1}, y\mapsto (\langle y,v^1\rangle,...,\langle y,v^{d-1}\rangle)$ with inverse $R_\mathfrak{p}^{-1}\colon \R^{d-1}\ni y \mapsto  \sum_{j=1}^{d-1} y_j v^j\in H_\mathfrak{p}$. 
We denote 
\[\Pi_\mathfrak{p}:= R_\mathfrak{p}\circ \pi_\mathfrak{p}=\Big(S^{d-1}\setminus\{\mathfrak{p}\}\ni x \mapsto R_\mathfrak{p}(\pi_\mathfrak{p}(x)) \in \R^{d-1}\Big).\]
Next, we investigate Riemannian metrics associated to these transformations. In particular, we are interested in the representation of the canonical spherical metric on $S^{d-1}\setminus\{\mathfrak{p}\}$ (that is the Euclidean metric on $\R^d$ restricted to $S^{d-1}\setminus\{\mathfrak{p}\}$).
Writing $x= \frac{1}{1+|h|^2}\big[ (|h|^2-1)\mathfrak{p}+2h\big]$, it holds  that 
\[ \frac{\partial  x_j}{\partial  h_k}=\frac{2}{(1+|h|^2)^2}\Big((1-|h|^2)\mathfrak{p}_jh_k -2h_jh_k\Big)+\frac{2}{1+|h|^2}(h_k \mathfrak{p}_j +\delta_{kj}).\]
This, using the notation $\langle h,\dd h\rangle = \sum_{j=1}^d h_j \dd h_j$, and \cite[Section 3.12]{Grigoryan2009_Book} ensure that
\[\begin{aligned}
    \dd x_j 
    &= \frac{4}{(1+|h|^2)^2}(\mathfrak{p}_j-h_j)\langle h,\dd h\rangle + \frac{2}{1+|h|^2}\dd h_j. 
\end{aligned}\]
This and the fact that $\mathfrak{p}$ is orthogonal to $H_\mathfrak{p}$ hence yield 
\[\begin{aligned}
    g_{S^{d-1}}&= \sum_{j=1}^d (\dd x_j)^2 
    =\frac{16}{(1+|h|^2)^3} \langle \mathfrak{p}, \dd h\rangle  \langle h,\dd h\rangle  +\frac{4}{(1+|h|^2)^2} \sum_{j=1}^d (\dd h_j)^2.\\
\end{aligned}\]
Next, writing $h = \sum_{j=1}^{d-1} y_j v^j$, it follows that $\frac{\partial h_\ell}{\partial y_k} = v^k_\ell,$
where $v^k_\ell$ denotes the $\ell$-th entry of $v^k$. Therefore it holds that $\dd h_j = v^k_j\dd y_k$, $(\dd h_j)^2= \sum_{k,l=1}^{d-1}v^k_j \dd y_k v^\ell_j \dd y_\ell$.
This and the fact that $\{v^1,...,v^{d-1}\}$ is an orthonormal basis of $H_\mathfrak{p}$, which in turn is orthogonal to $\mathfrak{p}$, ensure that
\[\begin{aligned}
    g_{S^{d-1}}
    &= \frac{16}{(1+|y|^2)^3}\Big(\underbrace{\sum_{j=1}^{d} \mathfrak{p}_j\sum_{k=1}^{d-1} v^k_j \dd y_k}_{=0}\Big)\Big( \sum_{j=1}^{d}\big(\sum_{k=1}^{d-1} y_k v^k_j\sum_{\ell=1}^{d-1} v^\ell_j \dd y_\ell\big)\Big) \\
    &\qquad + \frac{4}{(1+|y|^2)^2} \sum_{j=1}^{d}\sum_{k,\ell=1}^{d-1} v^k_j v^\ell_j \dd y_\ell \dd y_k\\
    &= \frac{4}{(1+|y|^2)^2} \sum_{j=1}^{d-1} (\dd y_j)^2.
\end{aligned}\]
It follows that the components of $g_{S^{d-1}}$ are given by $g_{ij}(y)= \frac{4}{(1+|y|^2)^2}\delta_{ij}$, $y\in \R^{d-1}$.
\subsection{Differential calculus under change of variables}
For $m,n\in \N, M\subseteq \R^m, N\subseteq \R^n$, and a differentiable mapping $f\colon M\rightarrow N$ we denote the Jacobian of $f$ by
\[J f(x):=\begin{pmatrix}
    \frac{\partial f_1}{\partial x_1}&\dots&\frac{\partial f_1}{\partial x_m}\\
    \vdots&&\vdots\\
    \frac{\partial f_n}{\partial x_1}&\dots&\frac{\partial f_n}{\partial x_m}
\end{pmatrix}, x\in M.\]

Let $\calD=\calD_\Omega$ be of the form~\eqref{eq:D:convention} with $\overline{\Omega}\subsetneq S^{d-1}$, 
let $\Pi\colon \Omega\to \calO:=\Pi(\Omega) \subsetneq \R^{d-1}$ be the stereographic projection as described in Appendix \ref{appendix:stereographic:projection}, and let $T_a$ be as defined in  Section~\ref{sec:domain:D}.
Let $X$ be a Banach space. 

\begin{lemma}\label{commuting:gradient:transform}
For all $u\in \scrD'(\calD;X)$ it  holds that 
    \begin{align}\label{eq:commuting:gradient:transform:1}
        T_0 \nabla_{\R^d} u &= M_{-1}((\Pi^{-1},J\Pi^{-1})^{-1})^\top\nabla_{\R\times \calO} T_0 u, 
    \\\label{eq:commuting:gradient:transform:2}
\nabla_{\R\times \calO} T_0u &= T_0 \Bigg[ x\mapsto |x| \bigg(\begin{pmatrix}
    x^\top/|x| \\
    (J\Pi)(\frac{x}{|x|})\big(I_{\R^d}-\frac{1}{|x|^2}xx^\top \big)
\end{pmatrix}^\top\bigg)^{-1} \Bigg]\nabla_{\R^d}  u.
    \end{align}
\end{lemma}
\begin{proof}
    Let 
    $\Xi\colon \R\times \calO \rightarrow \calD$, $ (z,y)\mapsto e^z\,\Pi^{-1}(y)$. This is a $C^\infty$-diffeomorphism, and it holds for all $(z,y)\in \R\times\calO$ that
    \begin{equation}
        J\Xi(z,y) =e^z \big(\Pi^{-1}(y), J\Pi^{-1}(y)\big), \quad \big(J\Xi(z,y)\big)^{-1}= e^{-z}\big(\Pi^{-1}(y),J\Pi^{-1}(y)\big)^{-1}.
    \end{equation}
Denote $T_\Xi u:= u\circ\Xi, u\in \scrD'(\calD;X)$. The chain rule ensures 
    \[
    (\nabla T_\Xi u)^\top = (T_\Xi \nabla u)^\top J\Xi, \quad \text{hence}\quad \nabla T_\Xi u = (J\Xi)^\top T_\Xi \nabla u,
    \]
    and thus
    \[T_\Xi \nabla u = ((J\Xi)^\top)^{-1} \nabla T_\Xi u. \]
    In particular, as $\Xi$ is a diffeomorphism,  $((J\Xi)^\top)^{-1}$ exists.
This proves \eqref{eq:commuting:gradient:transform:1}. 

The proof of \eqref{eq:commuting:gradient:transform:2} is analogous: 
\[\Xi^{-1}= \Big( \calD\ni x \mapsto(\log(|x|), \Pi\big(\tfrac{x}{|x|})\big)\in \R\times \calO \Big).\]
Then 
\[
J(\Xi^{-1})(x)= \frac{1}{|x|} \begin{pmatrix}
    x^\top/|x| \\
    (J\Pi)(\frac{x}{|x|})\big(I_{\R^d}-\frac{1}{|x|^2}xx^\top \big)
\end{pmatrix}\in \R^{d\times d},
\]
which  ensures \eqref{eq:commuting:gradient:transform:2}.    
\end{proof}
For the next lemma, for $b\in \R$ and $\varphi\in C_c^\infty(\calD)$, we write $|x|^b \varphi $ for the mapping 
\[ C_c^\infty(\calD) \ni \varphi \mapsto \{\calD\ni x \mapsto |x|^b\varphi(x) \in \C\}\in C_c^\infty(\calD),\]
and extend this to $\scrD'(\calD;X)$ by setting $$(|x|^b u)(\varphi):=u(|x|^b \varphi), \qquad u\in \scrD'(\calD;X),\,\varphi\in C_c^\infty(\calD).$$
\begin{lemma}\label{lem:DifferentialCalculus:EulerCoordinate}
    For all $a\in \R$, $\alpha,\beta\in \N^d_0$ with $|\beta|\le |\alpha|$ there exist $P_{\Pi^{-1},\alpha,\beta,a}\in C^\infty(\calO)$, $P_{\Pi,\alpha,\beta,a}\in C^\infty(\Omega)$ such that   for all $u\in \scrD'(\calD;X)$ it holds 
    \begin{equation}\label{eq:DifferentialCalculus:EulerCoordinate:1}
        T_a D^\alpha u = M_{-|\alpha|}\sum_{|\beta|\le |\alpha|} P_{\Pi^{-1},\alpha,\beta,a}D^\beta T_a u
    \end{equation}
    as well as 
    \begin{equation}\label{eq:DifferentialCalculus:EulerCoordinate:2}
        D^\alpha T_a u = T_a \sum_{|\beta|\le |\alpha|} |x|^{|\beta|} P_{\Pi,\alpha,\beta,a} D^\beta u
    \end{equation}
\end{lemma}
\begin{proof}
    Both \eqref{eq:DifferentialCalculus:EulerCoordinate:1} and \eqref{eq:DifferentialCalculus:EulerCoordinate:2} follow from induction on $|\alpha|\in\N$. To not overburden the notation, we abbreviate both  $P_{\Pi,\alpha,\beta,a}$ and $P_{\Pi^{-1},\alpha,\beta,a}$ with  $P_{\alpha,\beta}, \alpha,\beta\in\N_0^d$.
    We begin by proving~\eqref{eq:DifferentialCalculus:EulerCoordinate:1}. 
    The case $|\alpha|=k=0$ is clear, the case $|\alpha|=k=1$ follows from Lemma \ref{commuting:gradient:transform},
    noting that 
    \[\begin{aligned}
        T_a\nabla_{\R^d} u &= M_a M_{-1}\Big( (\Pi^{-1}, J\Pi^{-1})^{-1}\Big)^\top \nabla_{\R\times\calO}T_0u\\ 
        &= M_{-1} \Big( (\Pi^{-1}, J\Pi^{-1})^{-1}\Big)^\top \bigg(\nabla_{\R\times\calO} -\binom{a}{\bm{0}_{d-1}}\bigg) T_au,
    \end{aligned}\]
    where $\bm{0}_{d-1}:=(0,\ldots,0)\in \R^{d-1}.$
    For the induction step $|\alpha|\mapsto |\alpha|+1$  let $|\beta|=1$ and consider $D^{\alpha+\beta}u$. Then the induction hypothesis, the case $|\alpha|=1$, and Leibniz' rule  ensure 
    \[\begin{aligned}
        T_aD^{\alpha+\beta} u &= M_{-|\alpha|}\sum_{|\gamma|\le |\alpha|} P_{\alpha,\gamma} D^\gamma T_a D^\beta u\\
        &=M_{-|\alpha|}\sum_{|\gamma|\le |\alpha|}P_{\alpha,\gamma} D^\gamma M_{-1} \sum_{|\delta|\le |\beta|} P_{\beta,\delta} D^\delta T_a u\\
        &=M_{-|\alpha|}\sum_{|\gamma|\le |\alpha|} \sum_{|\delta|\le |\beta|} P_{\alpha,\gamma}\sum_{\varepsilon\le \beta}\Big(D^{\beta-\varepsilon}M_{-1}P_{\beta,\delta}\Big)\Big(D^{\varepsilon+\delta} T_a u\Big)\\
        &= M_{-|\alpha|-1} \sum_{|\gamma|\le |\alpha+\beta|} \widetilde{P}_{\alpha+\beta,\gamma} D^\gamma T_a u,
    \end{aligned}\]
    where $\widetilde P_{\alpha+\beta,\gamma}, \alpha,\beta,\gamma\in \N^d_0$, denote suitable $C^\infty(\calO)$–functions.
    This proves \eqref{eq:DifferentialCalculus:EulerCoordinate:1}. 
    To prove \eqref{eq:DifferentialCalculus:EulerCoordinate:2}, note that Lemma \ref{commuting:gradient:transform} combined with the fact that  $\nabla T_a u = M_a\big(\nabla +\binom{a}{\bm{0}_{d-1}}\big)T_0u$ covers the case $|\alpha|=1$. Assume that the hypothesis holds for $|\alpha|=k\in\N$, and let $|\beta|=1$. Then the induction hypothesis ensures 
\begin{equation}\label{eq:DifferentialCalculus:EulerCoordinate:proof:1}
    \begin{aligned}
        D^{\beta+\alpha}T_a u &= D^\beta T_a \sum_{|\gamma|\le |\alpha|} |x|^{|\gamma|} P_{\alpha,\gamma} D^\gamma u\\
        &= T_a \sum_{|\delta|\le |\beta|} |x|^{|\delta| }P_{\beta,\delta} D^\delta\Big(\sum_{|\gamma|\le|\alpha|} |x|^{|\gamma|} P_{\alpha,\gamma}D^\gamma u\Big) \\
        &= T_a \sum_{|\delta|\le |\beta|, |\gamma|\le |\alpha|} |x|^{|\delta|} P_{\beta,\delta} \sum_{\varepsilon\le \delta} D^{\delta-\varepsilon}\Big( |x|^{|\gamma|}P_{\alpha,\gamma}\Big) D^{\gamma+\varepsilon}u\\
        &= T_a \underbrace{\sum_{|\gamma|\le |\alpha|} P_{\beta,0} |x|^{|\gamma|}P_{\alpha,\gamma} D^\gamma u}_{\delta=0} \\
        &\qquad\qquad +T_a \sum_{|\delta|=1,|\gamma|\le |\alpha|} |x| P_{\beta,\delta} \sum_{\varepsilon\le \delta}D^{\delta-\varepsilon}\Big(|x|^{|\gamma|} P_{\alpha,\gamma}\Big) D^{\gamma+\varepsilon}u.
    \end{aligned}
\end{equation}
    Note that the chain rule implies that for $|\eta|=1$ it holds that
    \[D^\eta\Big( |x|^{|\gamma|} P_{\alpha,\gamma}\big(\tfrac{x}{|x|}\big) \Big)= |x|^{|\gamma|-1} \widehat{P}_{\alpha,\gamma}\big(\tfrac{x}{|x|}\big),\]
    where $\widehat P_{\alpha,\gamma}, \alpha,\gamma\in \N_0^d$, denote suitable $C^\infty(\Omega)$-functions.
    Therefore
    \[\begin{aligned}
        \sum_{|\delta|=1,|\gamma|\le |\alpha|} &|x| P_{\beta,\delta} \sum_{\varepsilon\le \delta}D^{\delta-\varepsilon}\Big(|x|^{|\gamma|} P_{\alpha,\gamma}\Big) D^{\gamma+\varepsilon}u\\
        &= \sum_{|\delta|=1,|\gamma|\le |\alpha|} |x| P_{\beta,\delta}\Big[ \underbrace{|x|^{|\gamma|-1}\widehat P_{\alpha,\gamma}D^\gamma u}_{\varepsilon=0} + \underbrace{|x|^{|\gamma|}P_{\alpha,\gamma} D^{\gamma+\delta}u}_{\varepsilon=\delta}\Big].
    \end{aligned}\]
    This and \eqref{eq:DifferentialCalculus:EulerCoordinate:proof:1} hence ensure the claim. 
    
\end{proof}
The following lemma links the Kondratiev type weighted Sobolev spaces $\Dot{W}^{k,p}(\calD,w_{\gamma,\nu}^\calD;X)$ on $\calD$ from Section~\ref{sec:Kondratiev} and the uniformly weighted Sobolev spaces $W^{k,p}(\widehat\calD,w_{\gamma,\nu}^{\widehat\calD};X)$. 

\begin{lemma}\label{DifferentialCalculus:Euclid:Euler:Sobolev:normequivalence}
    Let $k\in \N_0$, $p\in (1,\infty)$, $a,\gamma,\nu\in \R$. 
    Then it holds that 
    \begin{equation}
        \sum_{|\alpha|\le k} \nrm{D^\alpha u}_{L^p(\calD,w_{\gamma,\nu+(|\alpha|-k)p}^\calD)}
        \eqsim_{\Pi,a,k} \sum_{|\alpha|\le k}\nrm{D^\alpha T_a u}_{L^p(\widehat\calD, w_{\gamma,\nu+d-(a+k)p}^{\widehat{\calD}})}
        , \qquad u\in \sD'(\calD;X)
    \end{equation}
\end{lemma}
\begin{proof}
    This follows directly from Lemma \ref{lem:DifferentialCalculus:EulerCoordinate}. Note that 
    \[\sum_{|\alpha|\le k} \nrm{D^\alpha u}_{L^p(\calD,w_{\gamma,\nu+(|\alpha|-k)p}^\calD)} \eqsim \sum_{|\alpha|\le k} \nrm{T_aD^\alpha u}_{L^p(\widehat\calD, w_{\gamma,\nu+d-(a+k)p}^{\widehat{\calD}})}, \quad u\in \sD'(\calD;X).\]
    Hence, '$\lesssim$' follows from \eqref{eq:DifferentialCalculus:EulerCoordinate:1}, while '$\gtrsim$' follows from \eqref{eq:DifferentialCalculus:EulerCoordinate:2}.
\end{proof}

\section{The Dirichlet Laplace--Beltrami operator on a subset of the sphere}\label{appendix:Dirichlet:Laplace:Beltrami} In this section we present two results addressing the Dirichlet Laplace--Beltrami operator on smooth subsets of the unit sphere $S^{d-1}$. 
Recall that for a Riemannian manifold $(M,g)$ with $\dim M=m$ the Laplace--Beltrami operator is given by $\Delta^M u := \textup{div}_M(\nabla_M u),$ where $\textup{div}_M$ and $\nabla_M$ denote surface divergence and surface gradient, respectively.  In local coordinates $y_1,\ldots,y_m$, one obtains the representation 
\begin{equation}\label{eq:Laplace:Beltrami:local:coordinates:general}
    \Delta^M u=\frac{1}{\sqrt{\det g}} \sum_{i,j=1}^m\frac{\partial}{\partial y_i}\Big( \sqrt{\det g} g^{ij} \frac{\partial}{\partial y_j}u\Big).
\end{equation}
Let $\Omega$ be an open and connected set of $S^{d-1}$ with $\overline{\Omega}\subsetneq S^{d-1}$, and assume that $\Omega$ has a $C^2$-boundary.
Let $\Pi$ denote the stereographic projection constructed in Appendix \ref{appendix:stereographic:projection}, let $\calO:= \Pi(\Omega)$, and let $h(y):= \frac{2}{1+|y|^2}, y\in \calO$. In stereographic coordinates, the components of $g_{S^{d-1}}$ are 
\[ g_{ij}(y)= h(y)^2 \delta_{ij}, \quad y\in\calO, \quad i,j\in \{1,...,d-1\}\]
and the surface measure on $\Omega$ is given by 
\[ \bm\sigma_\Omega(A)= \int_A h(y)^{d-1}\dd y, \quad A\in \calB(\calO).\]
Since $\calO$ is bounded and $h$ is bounded from above and below on $\calO$, it holds that $\bm\sigma_\Omega (A) \eqsim \lambda^{d-1}(A), A\in \calB(\calO).$ Let $X$ be a Banach space, let $1<p<\infty, k\in \N_0$, and let $w$ be a weight. Then clearly
\[W^{k,p}(\calO,w h^{d-1};X)\simeq W^{k,p}(\calO, w;X).\]
Let $\gamma\in (-1,2p-1)\setminus\{p-1\}$ and assume $k\in \{0,1,2\}.$ Then 
\[ W^{k,p}_\Dir(\calO,w_\gamma^\calO h^{d-1};X)\simeq W^{k,p}_\Dir(\calO,w_\gamma^\calO;X).\]
For any $u\in W^{2,p}(\calO,w_\gamma^\calO;X)$ it holds that 
\[\begin{aligned}
 \Delta^\calO u(y) &:= \frac{1}{h(y)^{d-1}} D_j h(y)^{d-1} h(y)^{-2} D_j u(y) \\
    &=\frac{1}{h(y)^{d-1}}\Big[ (D_j h(y)^{d-3})D_j u(y) + h(y)^{d-3}D_j^2 u(y) \Big]\\
    &=\frac{1}{h(y)^{d-1}}\Big[- (d-3)h(y)^{d-4} y_j h(y)^2 D_j u(y) + h(y)^{d-3}D_j^2u(y) \Big]\\
    &=h(y)^{-2}\Delta u(y) - (d-3) h(y)^{-1} y\cdot \nabla u(y).
\end{aligned}\]

\begin{theorem}\label{Dirichlet:Laplace:Beltrami:vectorvalued:functionalcalculus}
    Let $d\geq 2$ and let $\calO\subseteq \R^{d-1}$ be a bounded $C^2$-domain, let $X$ be a UMD space, $1<p<\infty$, $\gamma\in (-1,2p-1)\setminus\{p-1\}$, and let $f \in C^1(\calO;\R_+)$. On $L^p(\calO,w_\gamma^\calO;X)$ we define the linear operators $\Delta_\Dir$ and $\Delta_\Dir^\calO$  with domain $\sfD(\Delta_\Dir):=\sfD(\Delta_\Dir^\calO):=W^{2,p}_\Dir(\calO,w_\gamma^\calO;X)$ by 
\begin{align*}
\Delta_\Dir u&:= \Delta u, \quad &&u \in W^{2,p}_\Dir(\calO,w_\gamma^\calO;X),\\ 
\Delta_\Dir^\calO u& := \Delta^\calO u, \quad &&u\in W^{2,p}_\Dir(\calO,w_\gamma^\calO;X).
\end{align*}
    Then 
    \begin{enumerate}[label=\textup{(\roman*)}]
        \item\label{it:Dirichlet:Laplace:Beltrami:vectorvalued:functionalcalculus:1} For every $\varphi>0$ there exists a $\tilde\lambda \in \R$ such that for all $\lambda\ge \tilde\lambda$ the operator $\lambda-f\Delta_\Dir$ admits a bounded $H^\infty$-calculus with $\omega_{H^\infty}(\lambda-f\Delta_\Dir)\le \varphi.$
        \item\label{it:Dirichlet:Laplace:Beltrami:vectorvalued:functionalcalculus:2} For every $\varphi>0$ there exist a  $\tilde\lambda\in \R$ such that for all $\lambda\ge \tilde \lambda$ the operator $\lambda-\Delta^\calO_\Dir$ has a bounded $H^\infty$-calculus with $\omega_{H^\infty}(\lambda-\Delta^\calO_\Dir)\le \varphi$. 
    \end{enumerate}
\end{theorem}
\begin{proof}
    First note that part \ref{it:Dirichlet:Laplace:Beltrami:vectorvalued:functionalcalculus:2} follows from part \ref{it:Dirichlet:Laplace:Beltrami:vectorvalued:functionalcalculus:1} by means of \cite[Theorem 13.1]{Kunstmann2004}. 
    Second, note that \cite[Theorem 6.1]{LindemulderVeraar2020} ensures that for every $\varphi>0$ there exists a $\hat\lambda\in \R$ such that for all $\lambda\ge \hat\lambda $ the operator $\lambda-\Delta_\Dir$ has a bounded $H^\infty$-calculus with $\omega_{H^\infty}(\lambda-\Delta_\Dir)\le \varphi$.
    Let $\lambda \ge \hat\lambda$ be fixed but arbitrary. 
    We prove the general case in two steps. 

\medskip

   \noindent\emph{Step 1: Small perturbations.} Let $m\in C^1(\calO)$ with $\nrm{m}_\infty<\eta,$ where $\eta$ is sufficiently small. Since $m\in C^1(\calO)$, it is ensured that $m$ is a pointwise multiplier in $W^{1,p}_\Dir(\calO,w_\gamma^\calO;X)$. Since $\nrm{m}_\infty<\eta$, it holds that \begin{equation}\label{Dirichlet:Laplace:Beltrami:vectorvalued:functionalcalculus:proof:1}
        \nrm{m\Delta_\Dir u}_{L^p(\calO,w_\gamma^\calO)}\le \nrm{m}_\infty \nrm{\Delta_\Dir u}_{L^p(\calO,w_\gamma^\calO)}< \eta\nrm{\Delta_\Dir u}_{L^p(\calO,w_\gamma^\calO)}. 
    \end{equation}
        Since $\lambda-\Delta_\Dir$ has a bounded $H^\infty$-calculus, it holds by \cite[Lemma 6.10]{LindemulderVeraar2020} 
        \[ \sfD((-\Delta_\Dir)^\frac{1}{2})=[L^p(\calO,w_\gamma^\calO;X),W^{2,p}_\Dir(\calO,w_\gamma^\calO;X)]_\frac{1}{2}= \begin{cases}
            W^{1,p}_\Dir(\calO,w_\gamma^\calO;X)& \gamma\in (-1,p-1),\\
            W^{1,p}(\calO,w_\gamma^\calO;X) & \gamma\in (p-1,2p-1).
        \end{cases}\]
        This and the fact that $m$ is a pointwise multiplier on $W^{1,p}_\Dir(\calO,w_\gamma^\calO;X)$ ensures for all $u\in \sfD((-\Delta_\Dir)^\frac{3}{2})$ \begin{equation}\label{Dirichlet:Laplace:Beltrami:vectorvalued:functionalcalculus:proof:2}
                \big\|(-\Delta_\Dir)^\frac{1}{2}m\Delta_\Dir u\big\|_{L^p(\calO,w_\gamma^\calO)}\lesssim \nrm{m\Delta_\Dir u}_{W^{1,p}(\calO,w_\gamma^\calO;X)}\le \nrm{m}_{C^1(\calO)}\nrm{u}_{\sfD((-\Delta_\Dir)^{\frac{3}{2}})},
        \end{equation}
        which proves $m\Delta_\Dir \in \sL(\sfD((-\Delta_\Dir)^\frac{3}{2}),\sfD((-\Delta_\Dir)^{\frac{1}{2}}))$. Therefore \cite[Theorem 3.2]{DenkDoreHieberPruesVenni_2004} ensures that $\lambda-(1+m)\Delta_\Dir$ has a bounded $H^\infty$-calculus of angle $\varphi$. 

\medskip

        \noindent\emph{Step 2: Localization.} A standard localization procedure now yields the claim. This follows along the lines of \cite[Section 6.6, Theorem 13.13]{Kunstmann2004}. Note that $\tilde\lambda$ depends on $\hat\lambda$ and $f$. 
 \end{proof}

In case $X = \C$ in Theorem \ref{Dirichlet:Laplace:Beltrami:vectorvalued:functionalcalculus}, we can say more about the spectrum of $\Delta_{\Dir}^{\calO}$.

 \begin{theorem}\label{Dirichlet:Laplace:Beltrami:functionalcalculus}
     Let $d\ge 2$, $1<p<\infty$, $\gamma\in (-1,2p-1)\setminus\{p-1\}$, and let $\calO\subseteq \R^{d-1}$ be a bounded $C^2$-domain. 
     On $L^p(\calO,w_\gamma^\calO)$ we define the linear operator $\Delta_\Dir^\calO$ with domain $\sfD(\Delta_\Dir^\calO):=W^{2,p}_\Dir(\calO,w_\gamma^\calO)$ by $\Delta_\Dir^\calO u:=\Delta^\calO u, u\in W^{2,p}_\Dir(\calO,w_\gamma^\calO)$. Then
     \begin{enumerate}[label=\textup{(\roman*)}]
        \item\label{it:Dirichlet:Laplace:Beltrami:functionalcalculus:1} $\sigma(-\Delta^\calO_\Dir) = \{\lambda_j\colon j\in\N\}$ is a discrete subset of $(0,\infty)$, which is independent of $p$ and $\gamma.$ In particular, $-\Delta^\calO_\Dir$ is invertible. 
         \item\label{it:Dirichlet:Laplace:Beltrami:functionalcalculus:2} For all $\lambda>-\min\{\lambda_j\colon j\in \N\}$, $\lambda-\Delta^\calO_\Dir$ has a bounded $H^\infty$-calculus of angle zero. 
     \end{enumerate}
 \end{theorem}
 \begin{proof}
     First note that Theorem \ref{Dirichlet:Laplace:Beltrami:vectorvalued:functionalcalculus}\ref{it:Dirichlet:Laplace:Beltrami:vectorvalued:functionalcalculus:2} ensures that $\lambda-\Delta^\calO_\Dir$ admits a bounded $H^\infty$-calculus of any angle $\varphi>0$, provided that $\lambda$ is large enough. 

     Next, \cite[Theorem 8.8]{OpicGurka1989} implies that $\sfD(\Delta^\calO_\Dir)\hookrightarrow L^p(\calO,w_\gamma^\calO)$ compactly. Therefore, for any $\lambda\in \rho(-\Delta^\calO_\Dir)$ it holds that $R(\lambda, -\Delta^\calO_\Dir)$ is a compact operator. 
     The spectral theorem for compact operators then ensures that $0 \in \sigma(R(\lambda,-\Delta^\calO_\Dir))$, that  $ \sigma(R(\lambda,-\Delta^\calO_\Dir))\setminus\{0\}\subseteq \sigma_p(R(\lambda,-\Delta^\calO_\Dir))$, and that $\sigma(R(\lambda,-\Delta^\calO_\Dir))\setminus\{0\}$ is either empty, contains finitely many eigenvalues $\mu_j\neq 0$, or a null sequence of eigenvalues $\mu_j \neq 0$. 
     Therefore, \cite[Theorem IV.1.13(ii)]{EngelNagel_Evo} ensures 
          \[\begin{aligned}
              \sigma(-\Delta^\calO_\Dir)&=\{\lambda -\mu^{-1}\colon \mu \in \sigma(R(\lambda,-\Delta^\calO_\Dir))\setminus\{0\}\}\\
              &\subseteq \{\lambda -\mu^{-1}\colon \mu \in \sigma_p(R(\lambda,-\Delta^\calO_\Dir))\}\\
              &=\sigma_p(-\Delta^\calO_\Dir).
          \end{aligned}\]
    In case $p=2$, $\gamma=0$, one can see that the spectrum has the required form in the following way:
    First, note that by \cite[Proposition 3.8]{LindemulderVeraar2020} it holds that 
    \[\sfD(\Delta^\calO_\Dir)=W_\Dir^{2,2}(\calO)=W^{1,2}_\Dir(\calO)\cap W^{2,2}(\calO)= W_0^{1,2}(\calO)\cap W^{2,2}(\calO),\]
    $W_0^{1,2}(\calO)$ being the closure of $C_c^\infty(\calO)$ in $W^{1,2}(\calO).$
     In addition, for each $u\in \sfD(\Delta^\calO_\Dir)$ it holds that 
    \[ \nrm{u}_{W^{2,2}(\calO)}\lesssim \nrm{u}_{W^{1,2}(\calO)}+ \nrm{\Delta^\calO u}_{L^2(\calO)}.\]
    On the other hand, for each $u\in W_\Dir^{2,2}(\calO)$ it holds 
    \[ \nrm{u}_{W^{1,2}(\calO)}+ \nrm{\Delta^\calO_\Dir u}_{L^2(\calO)} \lesssim \nrm{u}_{W^{2,2}(\calO)}.\]
    This proves $\sfD(\Delta^\calO_\Dir)= \{u\in W_0^{1,2}(\calO)\colon \Delta^\calO_\Dir u \in L^2(\calO)\}$ and therefore
    \[\sfD(\Delta^\calO_\Dir)= \{u\in W_0^{1,2}(\calO,h^{d-1})\colon \Delta^\calO_\Dir u \in L^2(\calO, h^{d-1})\}.\]
hence, $\Delta^\calO_\Dir$ is the Dirichlet Laplace--Beltrami operator appearing in \cite[Section 4.2]{Grigoryan2009_Book}. Then \cite[Theorem 10.20 \& Theorem 10.22]{Grigoryan2009_Book} proves that $\sigma(-\Delta^\calO_\Dir)$ has the required form. This proves \ref{it:Dirichlet:Laplace:Beltrami:functionalcalculus:1}.
    
    To see that spectrum and resolvent are independent of $p$ and $\gamma$, one argues by consistency in the same way as in \cite[Corollary 1.6.2]{Davies_1989_HeatKernels_spectral_theory}.
    By Theorem \ref{Dirichlet:Laplace:Beltrami:vectorvalued:functionalcalculus} there exists a $\nu>0 $ such that $\kappa-\Delta^\calO_\Dir$ is sectorial. Let $\lambda >-\min\{\lambda_j\colon j\in\N\}$. 
    Analyticity of the resolvent, the sectoriality of $\kappa-\Delta^\calO_\Dir$, and the description of $\sigma(-\Delta^\calO_\Dir) $ ensure that $\lambda-\Delta^\calO_\Dir$ is sectorial. Therefore, Theorem \ref{Dirichlet:Laplace:Beltrami:vectorvalued:functionalcalculus} and \cite[Corollary 5.5.5]{Haase2006_FunctionalCalculus} hence ensure that $\lambda-\Delta^\calO_\Dir$ has a bounded $H^\infty$-calculus of angle $0$, proving \ref{it:Dirichlet:Laplace:Beltrami:functionalcalculus:2}. This finishes the proof.
 \end{proof}
 \begin{remark}
     The results and the proofs in this Appendix extend to  subsets of manifolds which are conformally isometric to flat domains, under the condition that the conformal factor and its reciprocal are bounded. 
 \end{remark}
\bibliography{references.bib}

@book {aibs1,
	AUTHOR = {Hyt\"{o}nen, Tuomas and  Neerven, Jan van and Veraar, Mark and
	Weis, Lutz},
	TITLE = {Analysis in {B}anach spaces. {V}ol. {I}. {M}artingales and
	{L}ittlewood-{P}aley theory},
	SERIES = {Ergebnisse der Mathematik und ihrer Grenzgebiete. 3. Folge},
	VOLUME = {63},
	PUBLISHER = {Springer, Cham},
	YEAR = {2016},
	PAGES = {xvi+614},
	ISBN = {978-3-319-48519-5; 978-3-319-48520-1},
	MRCLASS = {46-02 (42B35 46E30)},
	MRNUMBER = {3617205},
	MRREVIEWER = {Adam\ Os\polhk{e}kowski},
}

@book {aibs2,
	AUTHOR = {Hyt\"{o}nen, Tuomas and  Neerven, Jan van and Veraar, Mark and
	Weis, Lutz},
	TITLE = {Analysis in {B}anach spaces. {V}ol. {II}. {P}robabilistic {M}ethods and {O}perator {T}heory},
	SERIES = {Ergebnisse der Mathematik und ihrer Grenzgebiete. 3. Folge},
	VOLUME = {67},
	PUBLISHER = {Springer, Cham},
	YEAR = {2017},
	PAGES = {xxi+616},
	ISBN = {978-3-319-69807-6; 978-3-319-69808-3},
	MRCLASS = {46-02 (42B35 46E30 47-02 60B11 60H30)},
	MRNUMBER = {3752640},
	MRREVIEWER = {Adam\ Os\polhk{e}kowski},
	DOI = {10.1007/978-3-319-69808-3},
	URL = {https://doi.org/10.1007/978-3-319-69808-3},
}

@book {aibs3,
	AUTHOR = {Hyt\"{o}nen, Tuomas and  Neerven, Jan van and Veraar, Mark and
	Weis, Lutz},
	TITLE = {Analysis in {B}anach spaces. {V}ol. {III}. {H}armonic {A}nalysis
	and {S}pectral {T}heory},
	SERIES = {Ergebnisse der Mathematik und ihrer Grenzgebiete. 3. Folge},
	VOLUME = {76},
	PUBLISHER = {Springer, Cham},
	YEAR = { 2023},
	PAGES = {xxi+826},
	ISBN = {978-3-031-46597-0; 978-3-031-46598-7},
	MRCLASS = {46-02 (35Kxx 42Bxx 46Bxx 46Exx)},
	MRNUMBER = {4696978},
	DOI = {10.1007/978-3-031-46598-7},
	URL = {https://doi.org/10.1007/978-3-031-46598-7},
}

@article {Kalton2001TheH,
	AUTHOR = {Kalton, Nigel J. and Weis, Lutz},
	TITLE = {The {$H^\infty$}-calculus and sums of closed operators},
	JOURNAL = {Math. Ann.},
	FJOURNAL = {Mathematische Annalen},
	VOLUME = {321},
	YEAR = {2001},
	NUMBER = {2},
	PAGES = {319--345},
	ISSN = {0025-5831,1432-1807},
	MRCLASS = {47A60 (34G10 47A13 47D06)},
	MRNUMBER = {1866491},
	MRREVIEWER = {Ian\ Raymond\ Doust},
	DOI = {10.1007/s002080100231},
	URL = {https://doi.org/10.1007/s002080100231},
}

@article{KalKunWei2006,
    AUTHOR = {Kalton, Nigel and Kunstmann, Peer and Weis, Lutz},
     TITLE = {Perturbation and interpolation theorems for the
              {$H^\infty$}-calculus with applications to differential
              operators},
   JOURNAL = {Math. Ann.},
  FJOURNAL = {Mathematische Annalen},
    VOLUME = {336},
      YEAR = {2006},
    NUMBER = {4},
     PAGES = {747--801},
      ISSN = {0025-5831,1432-1807},
   MRCLASS = {47A60 (35J10 35J15 42B30 46B70 47D06 47F05)},
  MRNUMBER = {2255174},
MRREVIEWER = {Christian\ Le Merdy},
       DOI = {10.1007/s00208-005-0742-3},
       URL = {https://doi.org/10.1007/s00208-005-0742-3},
}

@article {Prss2007HinftycalculusFT,
	AUTHOR = {Pr\"{u}ss, Jan and Simonett, Gieri},
	TITLE = {{$H^\infty$}-calculus for the sum of non-commuting operators},
	JOURNAL = {Trans. Amer. Math. Soc.},
	FJOURNAL = {Transactions of the American Mathematical Society},
	VOLUME = {359},
	YEAR = {2007},
	NUMBER = {8},
	PAGES = {3549--3565},
	ISSN = {0002-9947,1088-6850},
	MRCLASS = {47A60 (35K20 35K90 47A13 47N20)},
	MRNUMBER = {2302505},
	MRREVIEWER = {Atsushi\ Yagi},
	DOI = {10.1090/S0002-9947-07-04291-2},
	URL = {https://doi.org/10.1090/S0002-9947-07-04291-2},
}

@Inbook{Weis2006,
	author="Weis, Lutz",
	title="The $ H^\infty $ Holomorphic Functional Calculus for Sectorial Operators --- a Survey",
	bookTitle="Partial Differential Equations and Functional Analysis: The Philippe Cl{\'e}ment Festschrift",
	year="2006",
	publisher="Birkh{\"a}user Basel",
	address="Basel",
	pages="263--294",
	isbn="978-3-7643-7601-7",
	doi="10.1007/3-7643-7601-5_16",
	url="https://doi.org/10.1007/3-7643-7601-5_16"
}

@incollection {Kunstmann2004,
	AUTHOR = {Kunstmann, Peer C. and Weis, Lutz},
	TITLE = {Maximal {$L_p$}-regularity for parabolic equations, {F}ourier
	multiplier theorems and {$H^\infty$}-functional calculus},
	BOOKTITLE = {{Functional Analytic Methods for Evolution Equations}},
	SERIES = {Lecture Notes in Math.},
	VOLUME = {1855},
	PAGES = {65--311},
	PUBLISHER = {Springer, Berlin},
	YEAR = {2004},
	ISBN = {3-540-23030-0},
	MRCLASS = {47D06 (34G10 35D10 35J55 35K20 35K90 42B20 47A60)},
	MRNUMBER = {2108959},
	MRREVIEWER = {Xuan\ Thinh\ Duong},
	DOI = {10.1007/978-3-540-44653-8\_2},
	URL = {https://doi.org/10.1007/978-3-540-44653-8_2},
}

@article{KunWei2017,
    AUTHOR = {Kunstmann, Peer Christian and Weis, Lutz},
     TITLE = {New criteria for the {$H^\infty$}-calculus and the {S}tokes
              operator on bounded {L}ipschitz domains},
   JOURNAL = {J. Evol. Equ.},
  FJOURNAL = {Journal of Evolution Equations},
    VOLUME = {17},
      YEAR = {2017},
    NUMBER = {1},
     PAGES = {387--409},
      ISSN = {1424-3199,1424-3202},
   MRCLASS = {42B30 (35Q35)},
  MRNUMBER = {3630327},
       DOI = {10.1007/s00028-016-0360-4},
       URL = {https://doi.org/10.1007/s00028-016-0360-4},
}

@article{Kozlov2014TheDP,
	title={{The Dirichlet problem for non‐divergence parabolic equations with discontinuous in time coefficients in a wedge}},
	author={Vladimir A. Kozlov and Alexander I. Nazarov},
	journal={Math. Nachr.},
	year={2014},
	volume={287}
}

@article{KimLeeSeoRefinedGreenCone,
	title={{A Refined Green's Function Estimate of the Time Measurable Parabolic Operators with Conic Domains}},
	author={Kim, Kyeong-Hun and Lee, Kijung and Seo, Jinsol},
	journal={Potential Anal. 56},
	year={2022},
	pages={317-331},
	DOI={https://doi.org/10.1007/s11118-020-09886-w}
}

@article {LindemulderVeraar2020,
	AUTHOR = {Lindemulder, Nick and Veraar, Mark},
	TITLE = {The heat equation with rough boundary conditions and
	holomorphic functional calculus},
	JOURNAL = {J. Differential Equations},
	FJOURNAL = {Journal of Differential Equations},
	VOLUME = {269},
	YEAR = {2020},
	NUMBER = {7},
	PAGES = {5832--5899},
	ISSN = {0022-0396,1090-2732},
	MRCLASS = {47A60 (35K20 46B70 46E35 46E40)},
	MRNUMBER = {4104944},
	MRREVIEWER = {Elena\ Cordero},
	DOI = {10.1016/j.jde.2020.04.023},
	URL = {https://doi.org/10.1016/j.jde.2020.04.023},
}

@book {Prss1993EvolutionaryIE,
    AUTHOR = {Pr{\"u}ss, Jan},
     TITLE = {{Evolutionary Integral Equations and Applications}},
    SERIES = {Modern Birkh{\"a}user Classics},
      
 PUBLISHER = {Birkh\"auser/Springer Basel AG, Basel},
      YEAR = {1993},
     PAGES = {xxvi+366},
      ISBN = {978-3-0348-0498-1},
   MRCLASS = {45N05 (45-02 47D06)},
  MRNUMBER = {2964432},
       DOI = {10.1007/978-3-0348-8570-6},
       URL = {https://doi.org/10.1007/978-3-0348-8570-6},
}

@book {PrssSimonett2016,
	AUTHOR = {Pr\"{u}ss, Jan and Simonett, Gieri},
	TITLE = {{Moving Interfaces and Quasilinear Parabolic Evolution
	Equations}},
	SERIES = {Monographs in Mathematics},
	VOLUME = {105},
	PUBLISHER = {Birkh\"{a}user/Springer, [Cham]},
	YEAR = {2016},
	PAGES = {xix+609},
	ISBN = {978-3-319-27697-7; 978-3-319-27698-4},
	MRCLASS = {35-02 (35B30 35K93 35R35 47F05 58Jxx 76A15 80A22)},
	MRNUMBER = {3524106},
	MRREVIEWER = {Glen\ E.\ Wheeler},
	DOI = {10.1007/978-3-319-27698-4},
	URL = {https://doi.org/10.1007/978-3-319-27698-4},
}

@article{Nazarov2001,
    author ={Nazarov, Alexander I.} ,
    title ={{$L_p$}-estimates for a solution to the {D}irichlet problem
	and to the {N}eumann problem for the heat equation in a wedge
	with edge of arbitrary codimension} ,
    journal ={J. Math. Sci. (New York)} ,
    VOLUME = {106},
	YEAR = {2001},
	NUMBER = {3},    pages ={2989--3014},
    DOI = {10.1023/A:1011319521775},
	URL = {https://doi.org/10.1023/A:1011319521775},
}

@book {Grafakos2014classical,
	AUTHOR = {Grafakos, Loukas},
	TITLE = {Classical {F}ourier {A}nalysis},
	SERIES = {Graduate Texts in Mathematics},
	VOLUME = {249},
	EDITION = {Third},
	PUBLISHER = {Springer, New York},
	YEAR = {2014},
	PAGES = {xviii+638},
	ISBN = {978-1-4939-1193-6; 978-1-4939-1194-3},
	MRCLASS = {42-01 (42Bxx)},
	MRNUMBER = {3243734},
	MRREVIEWER = {Atanas\ G.\ Stefanov},
	DOI = {10.1007/978-1-4939-1194-3},
	URL = {https://doi.org/10.1007/978-1-4939-1194-3},
}

@book {Carslaw1952ConductionOH,
	AUTHOR = {Carslaw, Horatio S. and Jaeger, John C.},
	TITLE = {Conduction of {H}eat in {S}olids},
	PUBLISHER = {Oxford, at the Clarendon Press },
	YEAR = {1947},
	PAGES = {viii+386},
	MRCLASS = {36.0X},
	MRNUMBER = {22294},
	MRREVIEWER = {R.\ V.\ Churchill},
}

@book {EngelNagel_Evo,
	AUTHOR = {Engel, Klaus-Jochen and Nagel, Rainer},
	TITLE = {One-parameter semigroups for linear evolution equations},
	SERIES = {Graduate Texts in Mathematics},
	VOLUME = {194},
	NOTE = {With contributions by S. Brendle, M. Campiti, T. Hahn, G.
	Metafune, G. Nickel, D. Pallara, C. Perazzoli, A. Rhandi, S.
	Romanelli and R. Schnaubelt},
	PUBLISHER = {Springer-Verlag, New York},
	YEAR = {2000},
	PAGES = {xxii+586},
	ISBN = {0-387-98463-1},
	MRCLASS = {47D06 (34G10 35K90 47N20)},
	MRNUMBER = {1721989},
	MRREVIEWER = {C.\ J. K. Batty},
}

@article {Krylov1999c,
	AUTHOR = {Krylov, Nicolai V.},
	TITLE = {Weighted {S}obolev spaces and {L}aplace's equation and the
	heat equations in a half space},
	JOURNAL = {Comm. Partial Differential Equations},
	FJOURNAL = {Communications in Partial Differential Equations},
	VOLUME = {24},
	YEAR = {1999},
	NUMBER = {9-10},
	PAGES = {1611--1653},
	ISSN = {0360-5302,1532-4133},
	MRCLASS = {46E35 (35J15 35K10 46N20)},
	MRNUMBER = {1708104},
	MRREVIEWER = {Anna\ Mercaldo},
	DOI = {10.1080/03605309908821478},
	URL = {https://doi.org/10.1080/03605309908821478},
}

@article {GalVer2017,
	AUTHOR = {Gallarati, Chiara and Veraar, Mark},
	TITLE = {Maximal regularity for non-autonomous equations with
	measurable dependence on time},
	JOURNAL = {Potential Anal.},
	FJOURNAL = {Potential Analysis. An International Journal Devoted to the
	Interactions between Potential Theory, Probability Theory,
	Geometry and Functional Analysis},
	VOLUME = {46},
	YEAR = {2017},
	NUMBER = {3},
	PAGES = {527--567},
	ISSN = {0926-2601,1572-929X},
	MRCLASS = {35K90 (34G10 34G20 35B65 42B15 42B20 42B37 47D06)},
	MRNUMBER = {3630407},
	DOI = {10.1007/s11118-016-9593-7},
	URL = {https://doi.org/10.1007/s11118-016-9593-7},
}

@article {BanSmi1997,
	AUTHOR = {Ba\~nuelos, Rodrigo and Smits, Robert G.},
	TITLE = {Brownian motion in cones},
	JOURNAL = {Probab. Theory Related Fields},
	FJOURNAL = {Probability Theory and Related Fields},
	VOLUME = {108},
	YEAR = {1997},
	NUMBER = {3},
	PAGES = {299--319},
	ISSN = {0178-8051,1432-2064},
	MRCLASS = {60G40 (60J65)},
	MRNUMBER = {1465162},
	MRREVIEWER = {Zhan\ Shi},
	DOI = {10.1007/s004400050111},
	URL = {https://doi.org/10.1007/s004400050111},
}

@article {CioKimKyeLeeLin2018,
	AUTHOR = {Cioica-Licht, Petru A. and Kim, Kyeong-Hun and Lee, Kijung and
	Lindner, Felix},
	TITLE = {An {$L_p$}-estimate for the stochastic heat equation on an
	angular domain in {$\mathbb{R}^2$}},
	JOURNAL = {Stoch. Partial Differ. Equ. Anal. Comput.},
	FJOURNAL = {Stochastic Partial Differential Equations. Analysis and
	Computations},
	VOLUME = {6},
	YEAR = {2018},
	NUMBER = {1},
	PAGES = {45--72},
	ISSN = {2194-0401,2194-041X},
	MRCLASS = {60H15 (35B45 35K20 35R60)},
	MRNUMBER = {3768994},
	DOI = {10.1007/s40072-017-0102-9},
	URL = {https://doi.org/10.1007/s40072-017-0102-9},
}

@article{Lindemulder2024functionalcalculusweightedsobolev,
    author ={Nick Lindemulder and Emiel Lorist and Floris Roodenburg and Mark Veraar} ,
    title ={Functional calculus on weighted {S}obolev spaces for the {L}aplacian on the half-space} ,
    journal ={J. Funct. Anal.} ,
    year = {2025},
    volume = {289},
    number = {8},
    pages = {110985}
}

@article{cioicalicht2024sobolevspacesmixedweights,
      title={{Sobolev spaces with mixed weights and the Poisson equation on angular domains}}, 
      author={Petru A. Cioica-Licht and Cornelia Schneider and Markus Weimar},
      journal = {Preprint, arXiv:2409.18615},      
      year={2024},
      note = {To appear in \textit{Ann. Sc. Norm. Super. Pisa Cl. Sci. (5)}},
      archivePrefix={arXiv},
      primaryClass={math.AP},
      url={https://arxiv.org/abs/2409.18615} 
}

@book{isem27_HarmonicAnalysis,
    title={Harmonic Analysis Techniques for
Elliptic Operators},
    author={Moritz Egert and Robert Haller and Sylvie Monniaux and Patrick Tolksdorf},
publisher={Lecture Notes of the ISEM seminar 2023/24},
    year={2024},
    url={https://www.mathematik.tu-darmstadt.de/media/analysis/lehrmaterial_anapde/ISem_complete_lecture_notes.pdf}
}

@book{Are06,
    title={Heat Kernels},
    author={Wolfgang Arendt},
    publisher={Lecture Notes of the ISEM seminar 2005/06},
    year={2006},
    url={https://www.uni-ulm.de/fileadmin/website_uni_ulm/mawi.inst.020/arendt/downloads/internetseminar.pdf}
}

@book {MitMitMit2022a,
    AUTHOR = {Mitrea, Dorina and Mitrea, Irina and Mitrea, Marius},
     TITLE = {{Geometric Harmonic Analysis {I}. {A} Sharp Divergence Theorem
              with Nontangential Pointwise Traces}},
    SERIES = {Developments in Mathematics},
    VOLUME = {72},
 PUBLISHER = {Springer, Cham},
      YEAR = {2022},
     PAGES = {xxviii+924},
      ISBN = {978-3-031-05949-0; 978-3-031-05950-6},
   MRCLASS = {42-02 (26B20 28Axx 28C15 31B10 31C12 49Q15)},
  MRNUMBER = {4559592},
MRREVIEWER = {Juha\ Lehrb\"ack},
       DOI = {10.1007/978-3-031-05950-6},
       URL = {https://doi.org/10.1007/978-3-031-05950-6},
}

@article {DenkDoreHieberPruesVenni_2004,
    AUTHOR = {Denk, Robert and Dore, Giovanni and Hieber, Matthias and
              Pr\"uss, Jan and Venni, Alberto},
     TITLE = {New thoughts on old results of {R}.{T}.\ {S}eeley},
   JOURNAL = {Math. Ann.},
  FJOURNAL = {Mathematische Annalen},
    VOLUME = {328},
      YEAR = {2004},
    NUMBER = {4},
     PAGES = {545--583},
      ISSN = {0025-5831,1432-1807},
   MRCLASS = {35J45 (35B20 47A55 47A60 47G30)},
  MRNUMBER = {2047641},
MRREVIEWER = {Florin\ Iacob},
       DOI = {10.1007/s00208-003-0493-y},
       URL = {https://doi.org/10.1007/s00208-003-0493-y},
}

@book {Grigoryan2009_Book,
    AUTHOR = {Grigor'yan, Alexander},
     TITLE = {Heat kernel and analysis on manifolds},
    SERIES = {AMS/IP Studies in Advanced Mathematics},
    VOLUME = {47},
 PUBLISHER = {American Mathematical Society, Providence, RI; International
              Press, Boston, MA},
      YEAR = {2009},
     PAGES = {xviii+482},
      ISBN = {978-0-8218-4935-4},
   MRCLASS = {58J35 (31B05 31C12 35K08 35P15 35R01 47D07 58J50)},
  MRNUMBER = {2569498},
MRREVIEWER = {Thierry\ Coulhon},
       DOI = {10.1090/amsip/047},
       URL = {https://doi.org/10.1090/amsip/047},
}

@book {Davies_1989_HeatKernels_spectral_theory,
    AUTHOR = {Davies, E. B.},
     TITLE = {Heat kernels and spectral theory},
    SERIES = {Cambridge Tracts in Mathematics},
    VOLUME = {92},
 PUBLISHER = {Cambridge University Press, Cambridge},
      YEAR = {1989},
     PAGES = {x+197},
      ISBN = {0-521-36136-2},
   MRCLASS = {35P15 (35J25 35P20 47F05 58G05 58G11 58G25)},
  MRNUMBER = {990239},
MRREVIEWER = {H.\ Triebel},
       DOI = {10.1017/CBO9780511566158},
       URL = {https://doi.org/10.1017/CBO9780511566158},
}

@book {Haase2006_FunctionalCalculus,
    AUTHOR = {Haase, Markus},
     TITLE = {The {F}unctional {C}alculus for {S}ectorial {O}perators},
    SERIES = {Operator Theory: Advances and Applications},
    VOLUME = {169},
 PUBLISHER = {Birkh\"auser Verlag, Basel},
      YEAR = {2006},
     PAGES = {xiv+392},
      ISBN = {978-3-7643-7697-0; 3-7643-7697-X},
   MRCLASS = {47A60 (30E05 44A15 46B70 47A55 47D03 47E05 47F05)},
  MRNUMBER = {2244037},
MRREVIEWER = {Christian\ Le Merdy},
       DOI = {10.1007/3-7643-7698-8},
       URL = {https://doi.org/10.1007/3-7643-7698-8},
}

@article {BernicotFreyPetermichl_2016,
    AUTHOR = {Bernicot, Fr\'ed\'eric and Frey, Dorothee and Petermichl,
              Stefanie},
     TITLE = {Sharp weighted norm estimates beyond {C}alder\'on-{Z}ygmund
              theory},
   JOURNAL = {Anal. PDE},
  FJOURNAL = {Analysis \& PDE},
    VOLUME = {9},
      YEAR = {2016},
    NUMBER = {5},
     PAGES = {1079--1113},
      ISSN = {2157-5045,1948-206X},
   MRCLASS = {42B20 (58J35)},
  MRNUMBER = {3531367},
MRREVIEWER = {Elena\ Cordero},
       DOI = {10.2140/apde.2016.9.1079},
       URL = {https://doi.org/10.2140/apde.2016.9.1079},
}

@article {Wood_2007,
    AUTHOR = {Wood, Ian},
     TITLE = {Maximal {$L^p$}-regularity for the {L}aplacian on {L}ipschitz
              domains},
   JOURNAL = {Math. Z.},
  FJOURNAL = {Mathematische Zeitschrift},
    VOLUME = {255},
      YEAR = {2007},
    NUMBER = {4},
     PAGES = {855--875},
      ISSN = {0025-5874,1432-1823},
   MRCLASS = {35K20 (35B65)},
  MRNUMBER = {2274539},
MRREVIEWER = {Luca\ Lorenzi},
       DOI = {10.1007/s00209-006-0055-6},
       URL = {https://doi.org/10.1007/s00209-006-0055-6},
}

@article{Cio20,
author = {Cioica-Licht, Petru A.},
title = {An ${L}_p$-theory for the stochastic heat equation on angular domains in  $\mathbb{R}^2$ with mixed weights},
journal = {Preprint},
volume = {arXiv:2003.03782v2},
year  = {2020}
}

@article{CioKimLee2019,
author = {Cioica-Licht, Petru A. and Kim, Kyeong-Hun and Lee, Kijung},
title = {On the regularity of the stochastic heat equation on polygonal domains in $\mathbb{R}^2$},
journal = {J. Differential Equations},
volume = {267},
pages = {6447--6479},
year  = {2019},
}

@article {KimLeeSeo2022b,
    AUTHOR = {Kim, Kyeong-Hun and Lee, Kijung and Seo, Jinsol},
     TITLE = {Sobolev space theory and {H}\"older estimates for the
              stochastic partial differential equations on conic and
              polygonal domains},
   JOURNAL = {J. Differential Equations},
  FJOURNAL = {Journal of Differential Equations},
    VOLUME = {340},
      YEAR = {2022},
     PAGES = {463--520},
      ISSN = {0022-0396,1090-2732},
   MRCLASS = {60H15 (35R05 35R60)},
  MRNUMBER = {4483346},
       DOI = {10.1016/j.jde.2022.09.003},
       URL = {https://doi.org/10.1016/j.jde.2022.09.003},
}

@article{KimLeeSeo2021,
 author = {Kim, Kyeong-Hun and Lee, Kijung and Seo, Jinsol},
 year = {2021},
 title = {A weighted {S}obolev regularity theory of the parabolic equations with measurable coefficients on conic domains in $\mathbb{R}^d$},
 journal = {J. Differential Equations},
volume = {291},
pages = {154--194}
}

@article{Kim2004,
 author = {Kim, Kyeong-Hun},
 year = {2004},
 title = {On stochastic partial differential equations with variable coefficients in {$C^1$} domains},
 pages = {261--283},
 volume = {112},
 number = {2},
 journal = {Stochastic Process. Appl.}
}

@article{Lot2000,
 author = {Lototsky, Sergey V.},
 year = {2000},
 title = {Sobolev Spaces with Weights in Domains and Boundary Value Problems for Degenerate Elliptic Equations},
 pages = {195--204},
 volume = {7},
 number = {1},
 journal = {Methods Appl. Anal.}
}

@article {NeeVerWei2012,
    AUTHOR = {van Neerven, Jan and Veraar, Mark and Weis, Lutz},
     TITLE = {Stochastic maximal {$L^p$}-regularity},
   JOURNAL = {Ann. Probab.},
  FJOURNAL = {The Annals of Probability},
    VOLUME = {40},
      YEAR = {2012},
    NUMBER = {2},
     PAGES = {788--812},
      ISSN = {0091-1798,2168-894X},
   MRCLASS = {60H15 (35B65 42B25 47A60 47D06)},
  MRNUMBER = {2952092},
MRREVIEWER = {Feng-Yu\ Wang},
       DOI = {10.1214/10-AOP626},
       URL = {https://doi.org/10.1214/10-AOP626},
}

@article {NeeVerWei2012b,
    AUTHOR = {van Neerven, Jan and Veraar, Mark and Weis, Lutz},
     TITLE = {Maximal {$L^p$}-regularity for stochastic evolution equations},
   JOURNAL = {SIAM J. Math. Anal.},
  FJOURNAL = {SIAM Journal on Mathematical Analysis},
    VOLUME = {44},
      YEAR = {2012},
    NUMBER = {3},
     PAGES = {1372--1414},
      ISSN = {0036-1410,1095-7154},
   MRCLASS = {60H15 (35R60 46B09 47D06)},
  MRNUMBER = {2982717},
MRREVIEWER = {Elisa\ Al\`os},
       DOI = {10.1137/110832525},
       URL = {https://doi.org/10.1137/110832525},
}

@article {RoidosSchroheSeiler2021,
    AUTHOR = {Roidos, Nikolaos and Schrohe, Elmar and Seiler, J\"org},
     TITLE = {Bounded {$H_\infty$}-calculus for boundary value problems on
              manifolds with conical singularities},
   JOURNAL = {J. Differential Equations},
  FJOURNAL = {Journal of Differential Equations},
    VOLUME = {297},
      YEAR = {2021},
     PAGES = {370--408},
      ISSN = {0022-0396,1090-2732},
   MRCLASS = {58J32 (35Jxx 47A60 47G30)},
  MRNUMBER = {4280268},
MRREVIEWER = {Roland\ Schnaubelt},
       DOI = {10.1016/j.jde.2021.06.005},
       URL = {https://doi.org/10.1016/j.jde.2021.06.005},
}

@book {KozMazRos1997,
    AUTHOR = {Kozlov, Vladimir A. and Maz'ya, Vladimir G. and Rossmann, Jürgen},
     TITLE = {{Elliptic Boundary Value Problems in Domains with Point
              Singularities}},
    SERIES = {Mathematical Surveys and Monographs},
    VOLUME = {52},
 PUBLISHER = {American Mathematical Society, Providence, RI},
      YEAR = {1997},
     PAGES = {x+414},
      ISBN = {0-8218-0754-4},
   MRCLASS = {35J40 (35A20 35S15 46N20 47F05)},
  MRNUMBER = {1469972},
MRREVIEWER = {M.\ S.\ Agranovich},
       DOI = {10.1090/surv/052},
       URL = {https://doi.org/10.1090/surv/052},
}

@article {LoristVeraar2021,
    AUTHOR = {Lorist, Emiel and Veraar, Mark},
     TITLE = {Singular stochastic integral operators},
   JOURNAL = {Anal. PDE},
  FJOURNAL = {Analysis \& PDE},
    VOLUME = {14},
      YEAR = {2021},
    NUMBER = {5},
     PAGES = {1443--1507},
      ISSN = {2157-5045,1948-206X},
   MRCLASS = {60H15 (35B65 35R60 42B20 42B37 47D06)},
  MRNUMBER = {4307214},
MRREVIEWER = {Baris\ Evren\ Ugurcan},
       DOI = {10.2140/apde.2021.14.1443},
       URL = {https://doi.org/10.2140/apde.2021.14.1443},
}

@article {KoehneSaalWestermann2021,
    AUTHOR = {K{\"o}hne, Matthias and Saal, J{\"u}rgen and Westermann, Laura},
     TITLE = {Optimal {S}obolev regularity for the {S}tokes equations on a
              2{D} wedge domain},
   JOURNAL = {Math. Ann.},
  FJOURNAL = {Mathematische Annalen},
    VOLUME = {379},
      YEAR = {2021},
    NUMBER = {1-2},
     PAGES = {377--413},
      ISSN = {0025-5831,1432-1807},
   MRCLASS = {35Q30 (35K65 35K67 76D03 76D07)},
  MRNUMBER = {4211091},
MRREVIEWER = {Piotr\ Biler},
       DOI = {10.1007/s00208-019-01928-y},
       URL = {https://doi.org/10.1007/s00208-019-01928-y},
}

@article{koehne2024optimalregularitystokesequations,
      title={{Optimal regularity for the Stokes equations on a 2D wedge domain subject to Navier boundary conditions}}, 
      author={Matthias Köhne and Jürgen Saal and Laura Westermann},
      year={2024},
      journal = {Preprint, arXiv:2410.24063}
}

@article{lindemulder2025functionalcalculusweightedsobolev,
title = {{Functional calculus on weighted Sobolev spaces for the Laplacian on rough domains}},
journal = {Journal of Differential Equations},
volume = {454},
pages = {113884},
year = {2026},
issn = {0022-0396},
doi = {https://doi.org/10.1016/j.jde.2025.113884},
url = {https://www.sciencedirect.com/science/article/pii/S0022039625009118},
author = {Nick Lindemulder and Emiel Lorist and Floris B. Roodenburg and Mark C. Veraar}
}

@article{roodenburg2025complexinterpolationweightedsobolev,
    author ={Floris  Roodenburg} ,
    title ={Complex interpolation of weighted {S}obolev spaces with boundary conditions} ,
    journal = {Preprint, arXiv:2503.14636},
    year = {2025},
}

@article{DesLon2013,
    AUTHOR = {Desch, Gertrud and Londen, Stig-Olof},
     TITLE = {Maximal regularity for stochastic integral equations},
   JOURNAL = {J. Appl. Anal.},
  FJOURNAL = {Journal of Applied Analysis},
    VOLUME = {19},
      YEAR = {2013},
    NUMBER = {1},
     PAGES = {125--140},
      ISSN = {1425-6908,1869-6082},
   MRCLASS = {45R05},
  MRNUMBER = {3069768},
MRREVIEWER = {Natali\ Hritonenko},
       DOI = {10.1515/jaa-2013-0006},
       URL = {https://doi.org/10.1515/jaa-2013-0006},
}

@article {CowlingDoustMcIntoshYagi1996,
    AUTHOR = {Cowling, Michael and Doust, Ian and McIntosh, Alan and Yagi,
              Atsushi},
     TITLE = {Banach space operators with a bounded {$H^\infty$} functional
              calculus},
   JOURNAL = {J. Austral. Math. Soc. Ser. A},
  FJOURNAL = {Australian Mathematical Society. Journal. Series A. Pure
              Mathematics and Statistics},
    VOLUME = {60},
      YEAR = {1996},
    NUMBER = {1},
     PAGES = {51--89},
      ISSN = {0263-6115},
   MRCLASS = {47A60 (42B15 47B44 47D03)},
  MRNUMBER = {1364554},
MRREVIEWER = {J\"org\ Eschmeier},
}

@article {CoriascoSchroheSeiler2007,
    AUTHOR = {Coriasco, Sandro and Schrohe, Elmar and Seiler, Jörg},
     TITLE = {Bounded {$H_\infty$}-calculus for differential operators on
              conic manifolds with boundary},
   JOURNAL = {Comm. Partial Differential Equations},
  FJOURNAL = {Communications in Partial Differential Equations},
    VOLUME = {32},
      YEAR = {2007},
    NUMBER = {1-3},
     PAGES = {229--255},
      ISSN = {0360-5302,1532-4133},
   MRCLASS = {58J32 (35J70 47A60 47G30)},
  MRNUMBER = {2304149},
MRREVIEWER = {Thomas\ Krainer},
       DOI = {10.1080/03605300600910290},
       URL = {https://doi.org/10.1080/03605300600910290},
}

@article {MaierSaal2014,
    AUTHOR = {Maier, Siegfried and Saal, J\"urgen},
     TITLE = {Stokes and {N}avier-{S}tokes equations with perfect slip on
              wedge type domains},
   JOURNAL = {Discrete Contin. Dyn. Syst. Ser. S},
  FJOURNAL = {Discrete and Continuous Dynamical Systems. Series S},
    VOLUME = {7},
      YEAR = {2014},
    NUMBER = {5},
     PAGES = {1045--1063},
      ISSN = {1937-1632,1937-1179},
   MRCLASS = {35Q30 (35B30 35Q35 76D03)},
  MRNUMBER = {3252892},
       DOI = {10.3934/dcdss.2014.7.1045},
       URL = {https://doi.org/10.3934/dcdss.2014.7.1045},
}

@book{BorKon2006,
 author = {Borsuk, Mikhail and Kondratiev, Vladimir A.},
 year = {2006},
 title = {{Elliptic Boundary Value Problems of Second Order in Piecewise Smooth Domains}},
 address = {Amsterdam and Boston},
 edition = {1st},
 volume = {69},
 publisher = {Elsevier},
 isbn = {9780080461731},
 series = {North-Holland Mathematical Library}
}

@book{Dau1988,
 author = {Dauge, Monique},
 year = {1988},
 title = {{Elliptic Boundary Value Problems on Corner Domains: Smoothness and Asymptotics of Solutions}},
 address = {Berlin},
 volume = {1341},
 publisher = {Springer},
 isbn = {978-3-540-50169-5},
 series = {{Lecture Notes in Math.}},
 doi = {10.1007/BFb0086682}
}

@book{MazRos2010,
 author = {Maz'ya, Vladimir G. and Ro{\ss}mann, Jürgen},
 year = {2010},
 title = {{Elliptic Equations in Polyhedral Domains}},
 address = {Providence, RI},
 volume = {162},
 publisher = {AMS},
 isbn = {978-0-8218-4983-5},
 series = {Math. Surveys Monogr.}
}

@article{Sol2001,
 author = {Solonnikov, Vsevolod A.},
 year = {2001},
 title = {{$L_p$-estimates for solutions of the heat equation in a dihedral angle}},
 pages = {1--15},
 volume = {21},
 number = {1},
 journal = {{Rend. Mat. Appl. (7)}}
}

@article {AgrestiVeraar2022a,
    AUTHOR = {Agresti, Antonio and Veraar, Mark},
     TITLE = {Nonlinear parabolic stochastic evolution equations in critical
              spaces part {I}. {S}tochastic maximal regularity and local
              existence},
   JOURNAL = {Nonlinearity},
  FJOURNAL = {Nonlinearity},
    VOLUME = {35},
      YEAR = {2022},
    NUMBER = {8},
     PAGES = {4100--4210},
      ISSN = {0951-7715,1361-6544},
   MRCLASS = {60H15 (35K59 35R60 42B37 47D06)},
  MRNUMBER = {4459102},
       DOI = {10.1088/1361-6544/abd613},
       URL = {https://doi.org/10.1088/1361-6544/abd613},
}

@article {AgrestiVeraar2022b,
    AUTHOR = {Agresti, Antonio and Veraar, Mark},
     TITLE = {Nonlinear parabolic stochastic evolution equations in critical
              spaces part {II}: {B}low-up criteria and instataneous
              regularization},
   JOURNAL = {J. Evol. Equ.},
  FJOURNAL = {Journal of Evolution Equations},
    VOLUME = {22},
      YEAR = {2022},
    NUMBER = {2},
     PAGES = {Paper No. 56, 96},
      ISSN = {1424-3199,1424-3202},
   MRCLASS = {35R60 (35K59 60H15)},
  MRNUMBER = {4437443},
MRREVIEWER = {Le\ Chen},
       DOI = {10.1007/s00028-022-00786-7},
       URL = {https://doi.org/10.1007/s00028-022-00786-7},
}

@article {AgrVer2025,
    AUTHOR = {Agresti, Antonio and Veraar, Mark},
     TITLE = {Nonlinear {SPDE}s and {M}aximal {R}egularity: {A}n {E}xtended
              {S}urvey},
   JOURNAL = {NoDEA Nonlinear Differential Equations Appl.},
  FJOURNAL = {NoDEA. Nonlinear Differential Equations and Applications},
    VOLUME = {32},
      YEAR = {2025},
    NUMBER = {6},
     PAGES = {Paper No. 123},
      ISSN = {1021-9722,1420-9004},
   MRCLASS = {60H15 (35A01 35B65 35K57 35K59)},
  MRNUMBER = {4952170},
       DOI = {10.1007/s00030-025-01090-2},
       URL = {https://doi.org/10.1007/s00030-025-01090-2},
}

@book{DenKai2013,
    AUTHOR = {Denk, Robert and Kaip, Mario},
     TITLE = {General parabolic mixed order systems in {${L_p}$} and
              applications},
    SERIES = {Operator Theory: Advances and Applications},
    VOLUME = {239},
 PUBLISHER = {Birkh\"auser/Springer, Cham},
      YEAR = {2013},
     PAGES = {viii+250},
      ISBN = {978-3-319-01999-4; 978-3-319-02000-6},
   MRCLASS = {35-02 (35K40 35K41 35Q30 35Q79 35R35 76T10)},
  MRNUMBER = {3134525},
MRREVIEWER = {Thomas\ Krainer},
       DOI = {10.1007/978-3-319-02000-6},
       URL = {https://doi.org/10.1007/978-3-319-02000-6},
}

@article {cho2006,
    AUTHOR = {Cho, Sungwon},
     TITLE = {Two-sided global estimates of the {G}reen's function of
              parabolic equations},
   JOURNAL = {Potential Anal.},
  FJOURNAL = {Potential Analysis. An International Journal Devoted to the
              Interactions between Potential Theory, Probability Theory,
              Geometry and Functional Analysis},
    VOLUME = {25},
      YEAR = {2006},
    NUMBER = {4},
     PAGES = {387--398},
      ISSN = {0926-2601,1572-929X},
   MRCLASS = {35A08 (31B25 35K10)},
  MRNUMBER = {2255354},
MRREVIEWER = {Rodica\ Luca},
       DOI = {10.1007/s11118-006-9026-0},
       URL = {https://doi.org/10.1007/s11118-006-9026-0},
}

@book {Kufner1985,
    AUTHOR = {Kufner, Alois},
     TITLE = {Weighted {S}obolev spaces},
    SERIES = {A Wiley-Interscience Publication},
      NOTE = {Translated from the Czech},
 PUBLISHER = {John Wiley \& Sons, Inc., New York},
      YEAR = {1985},
     PAGES = {116},
      ISBN = {0-471-90367-1},
   MRCLASS = {46E35},
  MRNUMBER = {802206},
}

@article {OpicGurka1989,
    AUTHOR = {Opic, Bohum\'ir and Gurka, Petr},
     TITLE = {Continuous and compact imbeddings of weighted {S}obolev
              spaces. {II}},
   JOURNAL = {Czechoslovak Math. J.},
  FJOURNAL = {Czechoslovak Mathematical Journal},
    VOLUME = {39(114)},
      YEAR = {1989},
    NUMBER = {1},
     PAGES = {78--94},
      ISSN = {0011-4642},
   MRCLASS = {46E35},
  MRNUMBER = {983485},
MRREVIEWER = {Dinh Dung},
}
\bibliographystyle{plain}
\end{document}